\documentclass[11pt,twoside]{amsart}
\usepackage{latexsym,amssymb,amsmath}
\usepackage[all]{xy}
\usepackage{enumerate}
\usepackage{extpfeil}
\usepackage{afterpage}
\usepackage{float}

\numberwithin{equation}{section}
\def\demo{\noindent{\it Proof. }}
\newtheorem{theorem}{Theorem}[section]
\newtheorem{lemma}[theorem]{Lemma}
\newtheorem{proposition}[theorem]{Proposition}
\newtheorem{corollary}[theorem]{Corollary}

\theoremstyle{definition}
\newtheorem{definition}[theorem]{Definition}
\newtheorem{procedure}[theorem]{Procedure}
\newtheorem{remark}[theorem]{Remark}
\newtheorem{example}[theorem]{Example}

\begin{document}

%%%%%%%%%%%%%%%%%%%%%%%%%%%%%%%%%%%%%%%%%%%%%%%%%%%%%%%%%%%%%%%%%%%%%

\title[The v-number of edge ideals]
{The v-number of edge ideals}
%{v-number, regularity and combinatorial properties of 
%edge ideals with second symbolic power Cohen--Macaulay}

\author[D. Jaramillo]{Delio Jaramillo}
\address{
Departamento de
Matem\'aticas\\
Centro de Investigaci\'on y de Estudios
Avanzados del
IPN\\
Apartado Postal
14--740 \\
07000 Mexico City, CDMX.
}
\email{djaramillo@math.cinvestav.mx}
%\email{delio.jaramillo@cimat.mx}

\author[R. H. Villarreal]{Rafael H. Villarreal}
\address{
Departamento de
Matem\'aticas\\
Centro de Investigaci\'on y de Estudios
Avanzados del
IPN\\
Apartado Postal
14--740 \\
07000 Mexico City, CDMX.
}
\email{vila@math.cinvestav.mx}

%\date{Last edited Sep 2, 2019}

\begin{abstract} 
The aim of this work is to study the
v-number of edge ideals of clutters and graphs. We relate the v-number
with the regularity of edge ideals and study the combinatorial structure of
the graphs whose edge ideals
have their second symbolic power Cohen-Macaulay. 
%Let $\mathcal{C}$ be a clutter with vertex set 
%$V(\mathcal{C})=\{t_1,\ldots,t_s\}$ and without isolated vertices, let
%$S=K[t_1,\ldots,t_s]$ be a polynomial ring over a field $K$, and let
%$I=I(\mathcal{C})$ 
%be its edge ideal. We show a combinatorial expression for ${\rm v}(I)$, the
%v-number of $I$, and prove that the independent domination number
%$i(\mathcal{C})$ of $\mathcal{C}$ is an upper bound
%for ${\rm v}(I)$. If the independence complex $\Delta_\mathcal{C}$ of $\mathcal{C}$ is 
%vertex decomposable, we show that ${\rm v}(I)\leq{\rm reg}(S/I)$, and
%give the first example of a graph where ${\rm v}(I)>{\rm reg}(S/I)$. If $G$ is a graph
%and $\Delta_G$ is a pure shellable complex, we prove that the v-number
%of the ideal of vertex covers $I_c(G)$ of $G$ is $\alpha_0(G)-1$,
%where $\alpha_0(G)$ is the covering number of $G$. A graph $G$ belongs to class $W_2$ if
%$G$ has at least two vertices and any $2$ disjoint stable sets are
%contained in $2$ disjoint maximum stable sets. We use v-numbers to classify the class
%of $W_2$ graphs and show that $G$ is in $W_2$ if and only if 
%${\rm v}(I(G))=\dim(S/I(G))=\beta_0(G)$, where $\beta_0(G)$ is the
%independence number of $G$. This gives an algebraic method to determine if
%a given graph is in $W_2$. If $G$ has no isolated vertices and the
%symbolic square $I(G)^{(2)}$ is
%Cohen--Macaulay, we show that $G$ is edge-critical and is in $W_2$. 
\end{abstract}

\maketitle

\section{Introduction}\label{intro-section}
%\vskip .1in

Let $S=K[t_1,\ldots,t_s]=\oplus_{d=0}^{\infty} S_d$ be a polynomial ring over
a field $K$ with the standard grading and let $\mathcal{C}$ be a
\textit{clutter} with vertex 
set $V(\mathcal{C})=\{t_1,\ldots,t_s\}$, that is, $\mathcal{C}$ is a family of subsets of
$V(\mathcal{C})$, called \textit{edges}, none of which is included in
another. The set of edges of $\mathcal C$ is denoted by
$E(\mathcal{C})$. The primer example of a clutter is a simple graph
$G$. The \textit{edge ideal} of $\mathcal{C}$, denoted $I(\mathcal{C})$, is the ideal of $S$ generated
by all squarefree monomials $t_e:=\prod_{t_i\in e}t_i$ such that
$e\in E(\mathcal{C})$. In what follows $I$ denotes the edge ideal of $\mathcal{C}$.

A prime ideal $\mathfrak{p}$ of $S$ is an \textit{associated prime} of $S/I$ if
$(I\colon f)=\mathfrak{p}$ for some $f\in S_d$, where $(I\colon f)$ is
the set of all $g\in S$ such that $gf\in I$. The set of associated primes of $S/I$ 
is denoted ${\rm Ass}(S/I)$ or simply ${\rm Ass}(I)$. The v-{\em number} of $I$, denoted ${\rm
v}(I)$, is the following invariant of
$I$ that was introduced in \cite{min-dis-generalized} to study  
Reed--Muller-type codes:
$$
{\rm v}(I):=\begin{cases}\min\{d\geq 1 \mid\, \exists\, f 
\in S_d \mbox{ and }\mathfrak{p} \in {\rm Ass}(I) \mbox{ with } (I\colon f)
=\mathfrak{p}\} & \mbox{ if }\ I\subsetneq\mathfrak{m},\\  
0 &\mbox{ if }\ I=\mathfrak{m},
\end{cases}$$
where $\mathfrak{m}=(t_1,\ldots,t_s)$ is the maximal ideal
of $S$. This invariant can be computed using the software system
\textit{Macaulay}$2$
\cite{mac2} (Proposition~\ref{lem:vnumber},
Procedure~\ref{nov11-19}). The Castelnuovo--Mumford \textit{regularity} of $S/I$,
denoted ${\rm reg}(S/I)$, is bounded from above by $\dim(S/I)$
(Proposition~\ref{nov6-19}). We will show that in many interesting
cases ${\rm reg}(S/I)$ is bounded from below by ${\rm v}(I)$. 

A subset $A$ of $V(\mathcal{C})$ is called 
{\it independent\/} or {\it
stable\/} if $e\not\subset A$ for any  
$e\in E(\mathcal{C})$. The dual concept of a stable vertex set
is a {\it vertex cover\/}, i.e., a subset $C$ of $V(\mathcal{C})$ is a vertex
cover if and only if $V(\mathcal{C})\setminus C$ is a stable vertex set. A 
{\it minimal vertex cover\/} is a vertex cover which
is minimal with respect to inclusion.  
If $A$ is a stable set of 
$\mathcal{C}$, the \textit{neighbor set} of $A$, denoted 
$N_\mathcal{C}(A)$, is the set of all vertices $t_i$ 
such that $\{t_i\}\cup A$ contains an edge of $\mathcal{C}$. 
For use below $\mathcal{F}_\mathcal{C}$ denotes the family of all
maximal stable sets of $\mathcal{C}$ and $\mathcal{A}_\mathcal{C}$
denotes the family of all stable sets $A$ of $\mathcal{C}$ whose neighbor
set $N_\mathcal{C}(A)$ is a minimal vertex
cover of $\mathcal{C}$. 

Our next result gives a combinatorial formula of the 
v-number of the edge ideal of a clutter. In particular, as is seen below, 
it shows that the v-number is bounded from above by the independent
domination number of the clutter.   

\noindent \textbf{Theorem~\ref{v-number-clutters-graphs}.}\textit{ Let
$I=I(\mathcal{C})$ be the edge ideal of $\mathcal{C}$. Then, 
$\mathcal{F}_\mathcal{C}\subset\mathcal{A}_\mathcal{C}$ and 
$$\mathrm{v}(I)=\min\{|A|\colon 
A\in\mathcal{A}_\mathcal{C}\}.
$$
}
\quad The number of
vertices in any smallest vertex cover of $\mathcal{C}$, denoted 
$\alpha_0({\mathcal C})$, is called the \textit{vertex covering
number} of $\mathcal{C}$. The \textit{independence
number} of $\mathcal{C}$, denoted by $\beta_0(\mathcal{C})$,
is the number 
of vertices in 
any largest stable set of vertices of $\mathcal{C}$. The Krull dimension of $S/I(\mathcal{C})$,
denoted $\dim(S/I(\mathcal{C}))$, is equal to $\beta_0(\mathcal{C})$
and the height of $I(\mathcal{C})$, denoted ${\rm ht}(I(
\mathcal{C}))$, is equal to $\alpha_0(C)$.  A \textit{dominating set} of a
graph $G$ is a set $A$ of vertices of $G$ such that every vertex 
not in $A$ is adjacent to a vertex in $A$. The \textit{domination
number} of $G$, denoted $\gamma(G)$, is the minimum size of a
dominating set. An \textit{independent
dominating set} 
of $G$ is a set that is
both dominating and independent in $G$. Equivalently, an independent
dominating set is a maximal independent set, see
\cite[Proposition~2]{Allan-Laskar} and
Lemma~\ref{v-number-clutters-graphs-lemma}(b).   
The \textit{independent domination number} of $G$, denoted by $i(G)$, is the minimum size of
an independent dominating set. Thus, one has 
$$i(G)=\min\{|A|\colon
A\in\mathcal{F}_G\}.$$ 
\quad By analogy, we define the \textit{independent domination
number} of a clutter $\mathcal{C}$ by 
$$i(\mathcal{C}):=\min\{|A|\colon
A\in\mathcal{F}_\mathcal{C}\}.$$ 
\quad For every graph $G$, one has $\gamma(G)\leq i(G)\leq\beta_0(G)$
\cite[Theorem~2.7]{independent-domination}. The equality
$\gamma(G)=i(G)$ holds if $G$ is claw-free, that is,
if the complete bipartite graph $\mathcal{K}_{1,3}$ is not an induced subgraph of $G$
\cite[p.~75]{Allan-Laskar}. As a consequence of
Theorem~\ref{v-number-clutters-graphs} one has
$$
{\rm v}(I(\mathcal{C}))\leq i(\mathcal{C})\leq\beta_0(\mathcal{C}).
$$
\quad A clutter is
\textit{well-covered} if every maximal
stable set is a maximum stable set. 
We prove that ${\rm
v}(I(\mathcal{C}))=\beta_0(\mathcal{C})$ if and only if $\mathcal{C}$  is
well-covered and $\mathcal{F}_\mathcal{C}=\mathcal{A}_\mathcal{C}$
(Corollary~\ref{sep3-19}). 

A simplicial complex $\Delta$ is {\it vertex
decomposable\/} \cite{bjorner-topological} if either $\Delta$ is a
simplex, or $\Delta=\emptyset$, 
or $\Delta$ contains a vertex $v$, called a {\it shedding
vertex\/}, such that both the link ${\rm lk}_\Delta(v)$ and the 
deletion ${\rm del}_\Delta(v)$ of the vertex $v$ are vertex-decomposable, and
such that every facet of ${\rm del}_\Delta(v)$ is a facet of $\Delta$
(Section~\ref{prelim-section}).  
The simplicial complex $\Delta_\mathcal{C}$ whose faces are the
independent vertex sets of $\mathcal{C}$ is called the {\it independence
complex} of $\mathcal{C}$. If all maximal faces of
$\Delta_\mathcal{C}$ are of the same size the independence complex
$\Delta_\mathcal{C}$ is
called \textit{pure}. A clutter is called {\it vertex
decomposable\/} if 
its independence complex is vertex decomposable. 

We show a family of ideals where the
v-number is a lower bound for the regularity. 

\noindent\textbf{Theorem~\ref{vertex-decomposable-vnumber}.}
\textit{ If the independence complex $\Delta_\mathcal{C}$ of a clutter
$\mathcal{C}$ is vertex decomposable, then
$${\rm v}(I(\mathcal{C}))\leq{\rm reg}(S/I(\mathcal{C})).
$$
}
\quad If $I$ is the edge ideal of the clutter of circuits $\mathcal{C}_M$ of a matroid
$M$ on the set $X=\{t_1,\ldots,t_s\}$, then by a result of Provan and
Billera \cite[Theorem~3.2.1]{provan-billera} the independence complex
$\Delta$ of $M$ whose facets are the bases of $M$ is vertex
decomposable, $I=I_\Delta$ and, by
Theorem~\ref{vertex-decomposable-vnumber}, ${\rm v}(I)\leq {\rm
reg}(S/I)$ (Corollary~\ref{circuits-matroid}). If $G$ is a graph with no chordless
cycles of length other than $3$ or $5$, then by a result of Woodroofe
\cite[Theorem~1.1]{Woodroofe} the independence complex of $G$ is
vertex decomposable and, by Theorem~\ref{vertex-decomposable-vnumber},
 ${\rm v}(I(G))\leq {\rm
reg}(S/I(G))$ (Corollary~\ref{vnumber-chordless-criterion}). The work
done in \cite{footprint-ci}, though formulated in a different 
language, shows that ${\rm v}(I(G))\leq {\rm reg}(S/I(G))$ for any
Cohen--Macaulay bipartite graph $G$ \cite[Proposition~4.7]{footprint-ci}.
We extend this result to any sequentially Cohen--Macaulay bipartite
graph (Corollary~\ref{vnumber-scm-bipartite}).
 We give an example of a graph $G$ with ${\rm v}(I(G))>{\rm reg}(S/I(G))$
(Example~\ref{example-graph3}), disproving \cite[Conjecture~4.2]{footprint-ci} that the v-number of a
squarefree monomial ideal $I$ is a lower
bound for the regularity of $S/I$. In dimension $1$,
the conjecture is true as was shown in
\cite{min-dis-generalized}. It is an open problem whether or
not ${\rm v}(I)\leq {\rm reg}(S/I)+1$ holds for any squarefree
monomial ideal. 

The clutter of minimal vertex covers of $\mathcal{C}$, denoted 
$\mathcal{C}^\vee$, is called the {\it blocker} of $\mathcal{C}$. The 
edge ideal of $\mathcal{C}^\vee$, denoted by $I_c(\mathcal{C})$, is
called the {\it ideal of covers\/} of 
$\mathcal{C}$. If the independence complex
$\Delta_G$ of $G$ is pure and shellable, we prove ${\rm
v}(I_c(G))=\alpha_0(G)-1$ (Proposition~\ref{nov14-19}).

If $G$ is a graph without isolated vertices and $I(G)$ has a linear
resolution, we prove that ${\rm v}(I(G))=1$ 
(Proposition~\ref{dec7-19}). If $W_G$ is the 
whisker graph of a graph $G$, we show that the v-number of $I(W_G)$ is
equal to $i(G)$ and that the v-number of
$I(W_G)$  is bounded
from above by the regularity of $K[V(W_G)]/I(W_G)$
(Theorem~\ref{nov4-19}). 

A graph $G$ belongs to class $W_2$ if
$|V(G)|\geq 2$ and any $2$ disjoint stable sets are contained in $2$
disjoint maximum stable sets. A complete graph $\mathcal{K}_m$ on $m$ vertices belongs to
$W_2$ for $m\geq 2$.  
%The \textit{girth} of a graph is the size of a smallest 
%cycle in the graph. We say a graph with no cycles has infinite girth.
A graph $G$ is
\textit{$1$-well-covered} if $G$ is well-covered and
$G\setminus v$ is well-covered for all $v\in
V(G)$. A graph $G$ is in $W_2$ if and only if $G$ is $1$-well-covered and has
no isolated vertices \cite[Theorem~2.2]{Levit-Mandrescu}. For other
characterizations of graphs in $W_2$ see
\cite{Levit-Mandrescu,Staples} and 
the references therein.   

Our next two results classify the class of $W_2$ graphs. 

\noindent {\bf Theorem~\ref{W2-graphs}.}\textit{ Let $G$ be a graph
without isolated vertices. Then $G$ is in $W_2$ if and only if 
$G$ is well-covered and $\mathcal{F}_G=\mathcal{A}_G$.
}

The next result gives us an algebraic method to determine if
a given graph is in $W_2$ using \textit{Macaulay}$2$ \cite{mac2}
(Procedure~\ref{nov11-19}) and shows that for edge ideals of $W_2$ graphs, the
v-number is an upper bound for the regularity.

\noindent \textbf{Theorem~\ref{reg-w2}.}\textit{ Let $G$ be a graph
without isolated vertices and let 
$I=I(G)$ be its edge ideal. Then, $G$ is in $W_2$  if and only if
${\rm v}(I)=\dim(S/I)$.  
}

The $n$-th symbolic power of an edge ideal $I$, denoted
$I^{(n)}$, is given by $I^{(n)}:=\bigcap_{i=1}^r\mathfrak{p}_i^n$, where
$\mathfrak{p}_1,\ldots,\mathfrak{p}_r$ are the associated primes of
$I$. There are algebraic characterizations of the Cohen--Macaulay
property of $I^{(2)}$ and $I(G)^{(2)}$, $G$ a graph, 
given by N. C. Minh and N. V. Trung
\cite[Theorem~2.1]{Minh-Trung-adv} and D. T. Hoang, N. C. Minh and
T. N. Trung \cite[Theorem~2.2]{Hoang-etal}, respectively. We are interested in the 
combinatorial properties of graphs  whose edge ideals have their 
second symbolic power Cohen--Macaulay.

An edge in a graph is \textit{critical} if its removal 
increases the independence number. An \textit{edge-critical graph} 
is a graph  with only
critical edges. The concept of an edge-critical graph is introduced by
Ore \cite{Ore}. An edge-critical graph must be a
block and any two adjacent edges in such a graph must lie on a common
odd cycle \cite{Beineke-Harary-Plummer}. A structural
characterization of edge-critical graphs remains unknown
\cite{Plummer,Plummer-survey}.

Staples proves that a triangle-free $W_2$ graph is
edge-critical \cite{Staples-thesis} and that a connected 
$W_2$ graph different from $\mathcal{K}_2$ 
cannot have endvertices \cite[Theorem~4]{Staples}. If $G$ is a graph
with no isolated vertices, D. T. Hoang and T. N.
Trung prove that $I(G)^2$  is Cohen--Macaulay if and only if $G$ is
a triangle-free member of $W_2$
\cite[Theorem~4.4]{hoang-gorenstein-second-jaco}. From these 
results, one obtains that a graph $G$ is edge-critical
if $I(G)^2$ is Cohen--Macaulay. 

A graph $G$ is triangle-free if and
only if $I(G)^2=I(G)^{(2)}$
\cite[Theorem~4.13]{Dao-Stefani-Grifo-Huneke-Nunez}. If $I(G)^{(2)}$
is Cohen--Macaulay, our  
next result shows that $I(G)$ is edge-critical regardless of whether
$G$ has triangles or not.
This result---together with the tables of edge-critical graphs
given in \cite{Plummer,B-Small}---allows us to give the list of all
connected graphs $G$ with fewer than $10$ vertices such that
$I(G)^{(2)}$ is Cohen--Macaulay over a field of
characteristic $0$ (Remark~\ref{edge-criticalcm}, Table~\ref{tab:C-M,G}).

\noindent \textbf{Theorem~\ref{c-m-edge-critical}.}\textit{ If $G$ is
a graph and $I(G)^{(2)}$ is Cohen--Macaulay, then $G$ is
edge-critical.
}

If a graph $G$ is in $W_2$ and $v$ is a vertex of $G$, then Pinter
shows that $G_v:=G\setminus N_G[v]$ is in $W_2$ and 
$\beta_0(G_v)=\beta_0(G)-1$, where $N_G[v]$ is the closed
neighborhood of $v$ \cite[Theorem~5]{Pinter-jgt}. A recent result of Levit and
Mandrescu shows that the converse holds for well-covered graphs
without isolated vertices \cite[Theorem~3.9]{Levit-Mandrescu}.
 As an application of our classification of
$W_2$ graphs in terms of the v-number and the independence number
(Theorem~\ref{reg-w2}), for well-covered graphs without isolated
vertices we give a proof of the converse quite different from that of \cite{Levit-Mandrescu}.
%This converse is used to show another of our main results (Theorem~\ref{c-m-w2}).
%The content of this paper is as follows. In Section~\ref{prelim-section} we
%present some of the results and terminology that will be needed
%throughout the paper. 
%Let $G$ be a graph without isolated vertices. If $G$ is
%edge-critical and $t_i$ is a vertex of $G$, we prove that $G\setminus
%N_G[t_i]$ has no isolated vertices
%(Lemma~\ref{edge-critical-isolated}). 
Using
\cite[Theorem~2.2]{Hoang-etal} and a result of D. T. Hoang
\cite[Lemma~8]{Hoang-VJM} it follows that all graphs $G$ with $I(G)^{(2)}$ Cohen--Macaulay 
are in class $W_2$. There are Cohen--Macaulay connected edge-critical
graphs $G$ in $W_2$ with $I(G)^{(2)}$ not Cohen--Macaulay
(Example~\ref{example-graph3}). 

If a graph $G$ has triangles, 
to the best of our knowledge there is no characterization
of the Cohen--Macaulayness $I(G)^{(2)}$ in terms of the graph $G$
\cite[p.~1079]{Hoang-etal}. The Serre condition $\mathrm(S_2)$ for
$S/I(G)^{(2)}$ has been nicely classified in graph theoretical 
terms by D. T. Hoang, G. Rinaldo and N. Terai
\cite[Lemma~1]{Hoang-Rinaldo-Terai}. 
An open question is 
whether the $\mathrm(S_2)$ property of $S/I(G)^{(2)}$ implies that 
$I^{(2)}$ is Cohen--Macaulay \cite[p.~5]{Hoang-Rinaldo-Terai}.

For a squarefree monomial ideal $I$ of dimension $2$, the associated
Stanley--Reisner complex $\Delta_I$ of $I$ is a graph, N. C. Minh and N. V.
Trung \cite[Theorem~2.3]{Minh-Trung} give a combinatorial
classification in terms of $\Delta_I$ for the Cohen--Macaulay
property of $I^{(2)}$. For the edge ideal $I(G)$ of a graph $G$ with
independence number $2$, we give a
characterization in terms of the combinatorics of $G$ for the Cohen--Macaulayness of
$I(G)^{(2)}$. 

\noindent \textbf{Theorem~\ref{c-m-edge-critical-dim1}.}\textit{ Let
$G$ be a graph. If $\beta_0(G)=2$, then $I(G)^{(2)}$ is
Cohen--Macaulay if and only if $G$ is edge-critical.
}

If $G$ is an edge-critical graph without isolated
vertices and $\beta_0(G)=2$, we prove that $G$ is in $W_2$, 
then using that any $1$-dimensional
connected complex is vertex decomposable
\cite[Theorem~3.1.2]{provan-billera}, we prove that  
${\rm v}(I(G))={\rm reg}(S/I(G))=2$
(Corollary~\ref{c-m-edge-critical-dim1-coro}).

For all unexplained
terminology and additional information  we refer to
\cite{Har,Sta2,monalg-rev} (for graph theory, Stanley--Reisner rings and edge ideals), and
\cite{Eisen,Mats} (for commutative ring theory).

\section{Preliminaries}\label{prelim-section}

In this section we
present some of the results that will be needed throughout the paper
and introduce some more notation. All results of this
section are well-known. To avoid repetitions, we continue to employ
the notations and
definitions used in Section~\ref{intro-section}.

\begin{definition}\label{regularity-socle-degree}\rm Let $I\subset S$ be a graded ideal and let
${\mathbf F}$ be the minimal graded free resolution of $S/I$ as an
$S$-module:
\[
{\mathbf F}:\ \ \ 0\rightarrow
\bigoplus_{j}S(-j)^{b_{g,j}}
\stackrel{}{\rightarrow} \cdots
\rightarrow\bigoplus_{j}
S(-j)^{b_{1,j}}\stackrel{}{\rightarrow} S
\rightarrow S/I \rightarrow 0.
\]
The {\it Castelnuovo--Mumford regularity\/} of $S/I$ ({\it
regularity} of $S/I$ for short) is defined as
$${\rm reg}(S/I)=\max\{j-i \mid b_{i,j}\neq 0\}.
$$
If $g=\dim(S)-\dim(S/I)$, we say that the ring $S/I$ and the ideal 
$I$ are \textit{Cohen-Macaulay}. 
\end{definition}

For squarefree monomial ideals the regularity is additive.

\begin{proposition}{\rm(The regularity is additive
\cite[Lemma~7]{woodroofe-matchings})}\label{additivity-reg}  
Let $R_1=K[\mathbf{x}]$ and 
$R_2=K[\mathbf{y}]$ be two polynomial rings over a field $K$ and let 
$R=K[\mathbf{x},\mathbf{y}]$. If $I_1$ and $I_2$ are squarefree
monomial ideals of
$R_1$ and $R_2$, respectively, then  
$$
{\rm reg}(R/(I_1R+I_2R))={\rm reg}(R_1/I_1) +{\rm reg}(R_2/I_2).
$$
\end{proposition}

%The regularity is subadditive on monomial ideals 
%\cite{herzog-reg,kalai-meshulam}, that is, 
%if $I_1,I_2$ are monomial ideals of $S$, then 
%$${\rm reg}(S/(I_1+I_2))\leq {\rm reg}(S/I_1) +{\rm reg}(S/I_2).
%$$

%Another useful property of regularity is that one can delete isolated
%vertices of a squarefree monomial ideal without changing the regularity. 

\begin{lemma}\cite[Lemma~3.5]{edge-ideals}\label{addvariables}
Let $S'=K[t_1, \ldots ,t_{s-1}]$ and let $I'$ be a squarefree monomial ideal of
$S'$. If $S=K[t_1,\ldots,t_s]$ and $I=I'S$, 
then ${\rm reg}(S/I)={\rm reg}(S'/I')$. 
\end{lemma}

Let $\Delta$ be a \textit{simplicial complex} with vertex set 
$V$ contained in $\{t_1,\ldots ,t_s\}$, that is, $\Delta$ is a family
of subsets of $V$ called \textit{faces} such that $F\subset G\in\Delta$ implies $F\in\Delta$
and $\{v\}\in \Delta$ for all $v\in V$. For $v\in V$, define the {\it star\/} of $v$ as 
$$
{\rm star}(v):=
\{G\in \Delta\vert\, \{v\}\cup G\in \Delta\}.
$$
\quad 
The {\it deletion} of $v$, denoted 
${\rm del}_\Delta(v)$, is the subcomplex on $V\setminus\{v\}$ of all faces of $\Delta$ that
do not contain $v$. The \textit{induced subcomplex} on a set of vertices $A$ 
of $\Delta$, denoted $\Delta[A]$, is obtained by successively deleting
the vertices of $A$. The {\it link\/}
of $v$ is the subcomplex of ${\rm del}_\Delta(v)$ given by  
$${\rm lk}(v):=\{H\in \Delta\vert\ v\notin H\ \mbox{and}\
H\cup\{v\}\in
\Delta\}.$$ 
\quad The \textit{Stanley--Reisner ideal} $I_\Delta$ is the ideal of
$S$ defined as 
$$I_{\Delta}:=\left(\{t_{i_1}\cdots t_{i_r}\vert\ i_1<\cdots < i_r 
{\rm,}\ \{t_{i_1},\ldots, t_{i_r}\}\notin \Delta\}\right),
$$ 
\noindent and its \textit{Stanley--Reisner ring} $K[\Delta]$ 
is the quotient ring $S/I_\Delta$. Note that 
$t_i\notin I_\Delta$ for all $t_i\in S$.
\begin{proposition}{\rm(\cite[Lemma~2.3]{Ha-Woodroofe},
\cite[p.~1588]{kalai-meshulam})}\label{regularity-homology}
If $\Delta$ is a simplicial complex with vertex set $V$ and $I_\Delta$ is its
Stanley--Reisner ideal over a field $K$, then
\begin{eqnarray*}
{\rm reg}(K[V]/I_\Delta)&=&\max\{d\,\vert\,
\widetilde{H}_{d-1}(\Delta[A])\neq(0)\mbox{ for some }A\subset V\}\\
&=&\min\{d\,\vert\,
\widetilde{H}_{i}(\Delta[A])=(0)\mbox{ for all }A\subset V\mbox{
and }i\geq d\},
\end{eqnarray*}
where $\widetilde{H}_{i}(\Delta[A])$ is the $i$-th simplicial
homology module with
coefficients in $K$.
% and $\Delta[A]$ is the induced subcomplex on $A$.
\end{proposition}

Let $\mathcal{C}$ be a clutter with vertex set
$V(\mathcal{C})=\{t_1,\ldots,t_s\}$ and edge set $E(\mathcal{C})$. 
A clutter is also called a \textit{simple
hypergraph}. A clutter with no edges is called a \textit{discrete
clutter}. The edge ideal of a discrete clutter is $(0)$ by convention.
An \textit{isolated vertex} of $ \mathcal{C}$ is a vertex
that is not in any edge of $\mathcal{C}$. Any squarefree monomial ideal is
the edge ideal of a clutter 
\cite[pp.~220--221]{monalg-rev}.
The Krull dimension of $S/I(\mathcal{C})$ is equal to
$\max\{\dim(S/\mathfrak{p})\vert\, \mathfrak{p}\in{\rm
Ass}(I(\mathcal{C}))\}$.

A graded ideal is called \textit{unmixed} if all its associated
primes have the same height. A clutter is called \textit{unmixed} if its edge
ideal is unmixed. The next result tells us that a clutter
$\mathcal{C}$ is unmixed if and only if $\mathcal{C}$ is well-covered.

\begin{lemma}\cite[Lemma~6.3.37]{monalg-rev}\label{jul1-01} Let $C$
be a set of vertices of a clutter
$\mathcal{C}$. Then $C$ is a minimal vertex cover of $\mathcal{C}$
if and only if the ideal of $S$ generated by $C$ 
is an associated prime of $I(\mathcal{C})$.
\end{lemma}

The following result of Woodroofe gives a sufficient condition for
vertex decomposability of graphs and is an extension of the fact that chordal
graphs are shellable \cite{bipartite-scm}.

\begin{theorem}{\rm\cite[Theorem~1]{Woodroofe}}\label{Woodroofe-vertex-deco}
 If $G$ is a graph with no chordless
cycles of length other 
than $3$ or
$5$, then $G$ is vertex decomposable $($hence shellable and
sequentially Cohen-Macaulay$)$. 
\end{theorem}

\section{The v-number of a squarefree monomial
ideal}\label{v-number-section}
In this section we study the v-number and the algebraic invariants of
edge ideals of clutters and give a combinatorial description of the
v-number of a squarefree monomial ideal $I$ by considering the clutter
$\mathcal{C}$ whose edge ideal $I(\mathcal{C})$ is equal to $I$. To
avoid repetitions, we continue to employ 
the notations and  definitions used in Sections~\ref{intro-section}
and \ref{prelim-section}. 

For a graded module $M\neq 0$ we denote
$\alpha(M)=\min\{\deg(f) \mid f\in M, f\neq 0\}$. By convention, for
$M=0$ we set $\alpha(0)=0$. The next result was shown in
\cite[Proposition~4.2]{min-dis-generalized} for unmixed graded ideals. For squarefree 
monomial ideals the unmixed assumption is not needed.

\begin{proposition}\label{lem:vnumber}%\cite{min-dis-generalized} 
Let $I\subset S$ be a squarefree monomial ideal. Then 
$$
{\rm v}(I)=\min\{\alpha\left((I\colon\mathfrak{p})/{I}\right)\vert\,
\mathfrak{p}\in{\rm Ass}(I)\}.
$$
\end{proposition}

\begin{proof} Let $\mathfrak{p}_1,\ldots,\mathfrak{p}_r$ be the
associated primes of $I$. If $I$ is prime, then $(I\colon 1)=I$,
$(I\colon I)=S$, and ${\rm v}(I)=\alpha(S/I)=0$. Thus, we may assume
that $I$ has at least $2$ associated primes. There is $\mathfrak{p}_i$
and $f\in S_d$ such that $(I\colon f)=\mathfrak{p}_i$ and ${\rm
v}(I)=\deg(f)$. Then $f\in(I\colon\mathfrak{p}_i)\setminus I$ and 
$$
{\rm v}(I)=\deg(f)\geq\alpha((I\colon\mathfrak{p}_i)/I)\geq \min\{\alpha\left((I\colon\mathfrak{p})/{I}\right)\vert\,
\mathfrak{p}\in{\rm Ass}(I)\}.
$$
\quad Since $I=\bigcap_{i=1}^r\mathfrak{p}_i\subsetneq\bigcap_{i\neq
k}\mathfrak{p}_i=(I\colon\mathfrak{p}_k)$ for $1\leq k\leq r$, we can
pick a homogeneous polynomial $f_k$ in
$(I\colon\mathfrak{p}_k)\setminus I$ such that
$\alpha((I\colon\mathfrak{p}_k)/I)=\deg(f_k)$. Note that
$f_k\not\in\mathfrak{p}_k$ since $f_k\notin I$. Therefore, from the inclusions
$$  
\mathfrak{p}_k\subset(I\colon
f_k)=\bigcap_{i=1}^r(\mathfrak{p}_i\colon
f_k)=\bigcap_{f_k\not\in\mathfrak{p}_i}\mathfrak{p}_i\subset\mathfrak{p}_k, 
$$
we get $(I\colon f_k)=\mathfrak{p}_k$, and consequently ${\rm
v}(I)\leq\deg(f_k)$ for $k=1,\ldots,r$.
\end{proof}

\begin{proposition}\label{nov6-19} If $I$ is a squarefree monomial ideal of $S$, then
${\rm reg}(S/I)\leq\dim(S/I)$. 
\end{proposition}

\begin{proof} Let $V$ be the set of variables of $S$. There is a simplicial complex $\Delta$ with
vertex set $V$ such that $I$ is the Stanley--Reisner ideal $I_\Delta$
of $\Delta$ and $\dim(S/I)=\dim(\Delta)+1$. 
For $A\subset V$, let $\Delta[A]$ be the induced subcomplex on $A$ and
let $\widetilde{H}_{i}(\Delta[A])$ be the $i$-th simplicial homology module with
coefficients in $K$. Then, $\widetilde{H}_{i}(\Delta[A])=(0)$ for all 
$A\subset V$ and $i\geq\dim(S/I)$. Hence, the result follows from
Proposition~\ref{regularity-homology}. 
\end{proof}

A set of edges in a graph $G$ is called \textit{matching} if no two of them have a
vertex in common. By \cite[p.~295]{woodroofe-matchings}, 
${\rm reg}(S/I(G))$ is at most
the size of a minimum maximal matching. 

\begin{proposition}\label{upper-bound-reg} 
Let $G$ be a graph and let ${\rm isol}(G)$ be the
set of isolated vertices of $G$. Then
$${\rm reg}(S/I(G))\leq\dim(K[V(G)\setminus{\rm
isol}(G)]/I(G\setminus{\rm isol}(G)).$$
\end{proposition}

\begin{proof} By Lemma~\ref{addvariables}, ${\rm reg}(S/I(G))={\rm
reg}(K[V(G)\setminus{\rm 
isol}(G)]/I(G\setminus{\rm isol}(G))$. Hence, the inequality follows from
Proposition~\ref{nov6-19}. 
\end{proof}

Let $S=K[t_1,\ldots,t_s]=\oplus_{d=0}^{\infty} S_d$ be a polynomial ring over
a field $K$ with the standard grading and let $\mathcal{C}$ be a clutter with vertex set
$V(\mathcal{C})=\{t_1,\ldots,t_s\}$ and edge set $E(\mathcal{C})$. We
shall always assume that $I(\mathcal{C})$ is not a prime ideal and 
$E(\mathcal{C})\neq\emptyset$. If $I(\mathcal{C})$ is prime its v-number is $0$.

For use below we introduce the following  two families of stable
sets:
\begin{eqnarray*}
\mathcal{F}_\mathcal{C}&:=&\{A\vert\, A\mbox{ is a maximal stable set 
of }\ \mathcal{C}\},\mbox{ and}\\
\mathcal{A}_\mathcal{C}&:=&\{A\vert\, A\ \mbox{ is a stable set 
of }\ \mathcal{C}\mbox{ and }N_\mathcal{C}(A)\mbox{ is a minimal vertex
cover of }\ \mathcal{C}\}.
\end{eqnarray*}

\begin{lemma}\label{v-number-clutters-graphs-lemma}
Let $I=I(\mathcal{C})$ be the edge ideal of a clutter $\mathcal{C}$. The following hold. 
\begin{enumerate}
\item[\rm(a)] If $A\in\mathcal{A}_\mathcal{C}$ and $t_A=\prod_{t_i\in A}t_i$, then
$(I\colon t_A)=(N_\mathcal{C}(A))$.
\item[\rm(b)] If $A$ is stable and $N_\mathcal{C}(A)$ is a vertex cover,
then $N_\mathcal{C}(A)$ is a minimal vertex cover.
\item[\rm(c)] If $(I\colon f)=\mathfrak{p}$ for some $f\in S_d$ and
some $\mathfrak{p}\in{\rm Ass}(I)$, then there is
$A\in\mathcal{A}_\mathcal{C}$ with $|A|\leq d$ such that
$\mathfrak{p}=(N_\mathcal{C}(A))$ and $(I\colon t_A)=(N_\mathcal{C}(A))$.
\item[\rm(d)] If $A\in \mathcal{F}_\mathcal{C}$, then
$N_\mathcal{C}(A)=V(\mathcal{C})\setminus A$ and $(I\colon
t_A)=(N_\mathcal{C}(A))$. 
\end{enumerate}
\end{lemma}

\begin{proof} (a): To show the inclusion ``$\subset$'' take $t^c\in(I\colon t_A)$, that is,
$t^ct_A=t^bt_e$ for some monomial $t^b$ and some $e\in
E(\mathcal{C})$. As $N_\mathcal{C}(A)$ is a minimal vertex cover, 
one has $e\cap N_\mathcal{C}(A)\neq\emptyset$. 
Pick $t_i$ in $e\cap N_\mathcal{C}(A)$. Then $\{t_i\}\cup A$ contains an
edge of $\mathcal{C}$, and consequently $t_i\notin A$ since $A$ is
stable. Then from the equality $t^ct_A=t^bt_e$, we get that $t_i$
divides $t^c$. Thus $t^c\in N_\mathcal{C}(A)$. To show the inclusion
``$\supset$'' take $t_i\in N_\mathcal{C}(A)$, that is, $\{t_i\}\cup A$
contains an edge $e$ of $\mathcal{C}$, and $t_i\notin A$ since $A$ is
stable. Hence $t_it_A=t^bt_e$ for some monomial $t^b$. Thus, $t_i\in(I\colon t_A)$.

(b): Take $t_i\in N_\mathcal{C}(A)$. It suffices to show that 
$B:=N_\mathcal{C}(A)\setminus\{t_i\}$ is not a vertex cover of
$\mathcal{C}$. As $t_i$ is a neighbor of $A$, there is 
$e\in E(\mathcal{C})$ such that $e\subset \{t_i\}\cup A$, and
$t_i\notin A$ since $A$ is stable. Then $t_it_A=t^bt_e$ for some
monomial $t^b$. We need only show $B\cap e=\emptyset$. By
contradiction assume that $B\cap e\neq\emptyset$. Pick $t_j\in B\cap
e$. Then, $t_j\neq t_i$, $t_j\in e$, $t_j\notin A$ since $\{t_j\}\cup
A$ contains an edge and $A$ is stable. From the equality
$t_it_A=t^bt_e$ we obtain a contradiction.

(c): Writing $f=\sum_{i=1}^r\lambda_it^{c_i}$,
where $0\neq\lambda_i\in K$ and $t^{c_i}\in S_d$ for
all $i$, one has
$$
(I\colon f)=\bigcap_{i=1}^r(I\colon t^{c_i})=\mathfrak{p},
$$ 
and consequently $(I\colon t^{c_k})=\mathfrak{p}$ for some $k$. 
Thus we may assume $f=t^{c_k}=t_1^{a_1}\cdots t_s^{a_s}$. As $I$ is
squarefree, we may also assume that $a_i\in\{0,1\}$ for all $i$. We
set 
$$
A:={\rm supp}(t^{c_k})={\rm supp}(t_1^{a_1}\cdots t_s^{a_s})=\{t_i\vert\, a_i=1\},
$$
and $t_A=\prod_{t_i\in A}t_i$. The set $A$ is stable because
$(I\colon t_A)=\mathfrak{p}\subsetneq S$. Now we show the equality
$\mathfrak{p}=(N_\mathcal{C}(A))$. To show the inclusion ``$\subset$''
take $t_i$ in $\mathfrak{p}$. Then, we can write $t_it_A=t^bt_e$ for
some monomial $t^b$ and some edge $e$ of $\mathcal{C}$. Hence
$\{t_i\}\cup A$ contains $e$, that is, $t_i\in N_\mathcal{C}(A)$. 
To show the inclusion ``$\supset$'' take $t_i\in N_\mathcal{C}(A)$.
Then $\{t_i\}\cup A$ contains an edge $e$ of $ \mathcal{C}$. Thus,
$t_it_A=t^bt_e$ and $t_i$ is in $(I\colon t_A)=\mathfrak{p}$. As $\mathfrak{p}$ is generated 
by a minimal vertex cover of $\mathcal{C}$ (Lemma~\ref{jul1-01}), we obtain
that $N_\mathcal{C}(A)$ is a minimal vertex cover of $\mathcal{C}$. Thus
$A$ is in $\mathcal{A}_\mathcal{C}$.

(d): First we show the equality
$N_\mathcal{C}(A)=V(\mathcal{C})\setminus A$. To show the inclusion
``$\subset$'' take $t_i\in N_\mathcal{C}(A)$. Then $\{t_i\}\cup A$
contains an edge and $t_i\notin A$ since $A$ contains no edges of
$\mathcal{C}$. To show the inclusion ``$\supset$'' take $t_i\notin A$,
then $\{t_i\}\cup A$ contains an edge since $A$ is a maximal stable
set. Thus $t_i\in N_\mathcal{C}(A)$. 

Now we show that $(I\colon t_A)$ is equal to $(N_\mathcal{C}(A))$. By the
previous equality it suffices to show the equality 
$(I\colon t_A)=(V(\mathcal{C})\setminus A)$. To show the inclusion
``$\subset$'' take $t^c\in(I\colon t_A)$, that is, $t^ct_A\in I$. 
Thus $t^ct_A=t^bt_e$ for some monomial $t^b$ and some $e\in E(\mathcal{C})$. As
$e\not\subset A$, there is $t_i\in e\setminus A$. Hence $t_i$ divides
$t^c$, and $t^c\in(V(\mathcal{C})\setminus A)$. To show the inclusion
``$\supset$'' take $t_i\in V(\mathcal{C})\setminus A$. As $A$ is
a maximal stable set, $\{t_i\}\cup A$ contains an edge $e$ of
$\mathcal{C}$ such that
$t_i\in e$. Thus $t_it_A=t^bt_e$ for some
monomial $t^b\in S$, that is, $t_i\in(I\colon t_A)$.
\end{proof}

The next result gives a combinatorial description of the 
v-number of the edge ideal of a clutter. 

\begin{theorem}\label{v-number-clutters-graphs}
Let $I$ be the edge ideal of a clutter $\mathcal{C}$.
If $I$ is not prime, then $\mathcal{F}_\mathcal{C}\subset\mathcal{A}_\mathcal{C}$ and 
$$\mathrm{v}(I)=\min\{|A|\colon 
A\in\mathcal{A}_\mathcal{C}\}.
$$
\end{theorem}

\begin{proof} To show the inclusion
$\mathcal{F}_\mathcal{C}\subset\mathcal{A}_\mathcal{C}$ take
$A\in\mathcal{F}_\mathcal{C}$. 
Then, by Lemma~\ref{v-number-clutters-graphs-lemma}(d), 
$(I\colon t_A)=(N_\mathcal{C}(A))$. Thus, $(N_\mathcal{C}(A))$ is an
associated prime of $I$. According to
Lemma~\ref{jul1-01}
any associated prime of $I$ is generated by a minimal vertex cover 
of $\mathcal{C}$, and consequently $N_\mathcal{C}(A)$ is a minimal
vertex cover of $\mathcal{C}$. Thus $A\in\mathcal{A}_\mathcal{C}$. Now
we prove the formula for ${\rm v}(I)$. For any $A\in\mathcal{A}_\mathcal{C}$, by 
Lemma~\ref{v-number-clutters-graphs-lemma}(a), one has  
$(I\colon t_A)=(N_G(A))$. Hence, 
$$\mathrm{v}(I)\leq \min\{|A|\colon 
A\in\mathcal{A}_\mathcal{C}\}.
$$ 
\quad To show the reverse inequality pick
a polynomial $f\in S_d$ and $\mathfrak{p}\in{\rm Ass}(I)$ such that 
$(I\colon f)=\mathfrak{p}$ and ${\rm v}(I)=d$. By
Lemma~\ref{v-number-clutters-graphs-lemma}(c), there is
$A\in\mathcal{A}_\mathcal{C}$ with $|A|\leq d$ such that $(I\colon
t_A)=\mathfrak{p}$. Therefore, we get $\mathrm{v}(I)\geq \min\{|A|\colon 
A\in\mathcal{A}_\mathcal{C}\}$.
\end{proof}

\begin{corollary}\label{coro1-vnumber} If $\mathcal{C}$ is a clutter
with independent domination number $i(\mathcal{C})$ and independence
number $\beta_0(G)$, then 
${\rm v}(I(\mathcal{C}))\leq i(\mathcal{C})\leq\beta_0(\mathcal{C})$.
\end{corollary}

\begin{proof} 
The inequalities follow at once from Theorem~\ref{v-number-clutters-graphs}. 
\end{proof}

A clutter $\mathcal{C}$ is \textit{well-covered} if every maximal
stable set is a maximum stable set. A clutter $\mathcal{C}$ is
$1$-\textit{well-covered} if $\mathcal{C}$ is well-covered and
$\mathcal{C}\setminus v$ is well-covered for all $v\in
V(\mathcal{C})$. 

\begin{corollary}\label{sep3-19} Let $I=I(\mathcal{C})$ be the edge ideal of 
a clutter $\mathcal{C}$. Then, ${\rm v}(I)=\dim(S/I)$ if and only if $\mathcal{C}$ is
well-covered and $\mathcal{F}_\mathcal{C}=\mathcal{A}_\mathcal{C}$.
\end{corollary}

\demo $\Rightarrow$) By
Theorem~\ref{v-number-clutters-graphs}, one has
$\mathcal{F}_\mathcal{C}\subset\mathcal{A}_\mathcal{C}$.
Assume $B$ is any maximal
stable set. Then
$B\in\mathcal{F}_\mathcal{C}\subset\mathcal{A}_\mathcal{C}$ and 
$\dim(S/I)={\rm v}(I)\leq|B|\leq\dim(S/I)$. Thus $|B|=\dim(S/I)$. This proves
that $\mathcal{C}$ is well-covered. To show the equality 
$\mathcal{F}_\mathcal{C}=\mathcal{A}_\mathcal{C}$ take
$B\in\mathcal{A}_\mathcal{C}$. Pick a maximal stable set $B'$ that
contains $B$. Then, $\dim(S/I)={\rm v}(I)\leq|B|\leq|B'|=\dim(S/I)$ and 
$B=B'$. Thus $B\in\mathcal{F}_\mathcal{C}$. 

$\Leftarrow$) As $\mathcal{C}$ is well-covered and
$\mathcal{F}_\mathcal{C}=\mathcal{A}_\mathcal{C}$, one has 
$${\rm v}(I)=\min\{|A|\colon  
A\in\mathcal{F}_\mathcal{C}\}=\max\{|A|\colon
A\in\mathcal{F}_\mathcal{C}\}=\dim(S/I).\quad \Box
$$

For squarefree monomial ideals the ${\rm v}$-number is additive.

\begin{proposition}{\rm(The {\rm v}-number is additive)}\label{additivity-vnumber} 
Let $R_1=K[\mathbf{x}]$ and 
$R_2=K[\mathbf{y}]$ be two polynomial rings over a field $K$ and let 
$R=K[\mathbf{x},\mathbf{y}]$. If $I_1$ and $I_2$ are squarefree
monomial ideals of $R_1$ and $R_2$, respectively, then  
$$
{\rm v}(I_1R+I_2R)={\rm v}(I_1) +{\rm v}(I_2).
$$
\end{proposition}

\begin{proof} We set $I=I_1R+I_2R$. Pick $f=x^ay^b$ and
$\mathfrak{p}\in{\rm Ass}(R/I)$ such that $(I\colon f)=\mathfrak{p}$ 
and $\deg(f)={\rm v}(I)$. Then, by \cite[Lemma~3.4]{HaM}, 
we can write $\mathfrak{p}=\mathfrak{p}_1R + \mathfrak{p}_2R$, 
where $\mathfrak{p}_1 \in {\rm Ass}(R_1/I_1)$ and ${\mathfrak p}_2 \in 
{\rm Ass}(R_2/I_2)$. It is not hard to prove that the following
equalities hold: 
$$
(I_1\colon x^a)=\mathfrak{p}_1\ \mbox{ and }\ (I_2\colon
y^b)=\mathfrak{p}_2. 
$$
\quad Thus, ${\rm v}(I_1)+{\rm
v}(I_2)\leq\deg(x^a)+\deg(y^b)=\deg(f)={\rm v}(I)$. To show the
reverse inequality, for $i=1,2$ pick $\mathfrak{p}_i\in{\rm
Ass}(R_i/I_i)$ and $f_i$ a squarefree monomial in $R_i$ such that 
$(I_i\colon f_i)=\mathfrak{p}_i$. Then,
$\mathfrak{p}:=\mathfrak{p}_1R+\mathfrak{p}_2R$ is an associated prime
of $I$ \cite[Lemma~3.4]{HaM} and $(I\colon f_1f_2)=\mathfrak{p}$.
Hence, setting $f:=f_1f_2$, one has ${\rm v}(I)\leq\deg(f)={\rm
v}(I_1)+{\rm v}(I_2)$.
\end{proof}

\begin{proposition}\label{ci-vnumber} If $I$ is a complete intersection squarefree monomial ideal
minimally generated by $G(I)=\{t^{a_1},\ldots,t^{a_r}\}$ and let $d_i$
be the degree of $t^{a_i}$ for $i=1,\ldots,r$. Then 
$$
{\rm v}(I)=d_1+\cdots+d_r-r={\rm reg}(S/I).
%\sum_{i=1}^r(d_i-1)
$$
\end{proposition}

\begin{proof} As $I$ is a complete intersection $t^{a_i}$ and $t^{a_j}$
do not have common variables for $i\neq j$. Note that the regularity of $S/(t^{a_i})$
and the v-number of $(t^{a_i})$ are equal to $d_i-1$. Hence, the equalities 
follow from Propositions~\ref{additivity-reg} and
\ref{additivity-vnumber}. 
\end{proof}

\begin{lemma}\label{vnumber-comparison-clutter} 
Let $I=I(\mathcal{C})$ be the edge ideal of a clutter $\mathcal{C}$
and let $t_i$ be a vertex of $\mathcal{C}$ such that $\{t_i\}\notin
E(\mathcal{C})$. Then, ${\rm v}(I)\leq {\rm v}(I\colon t_i)+1$.
\end{lemma}

\begin{proof} If $(I\colon t_i)$ is generated by a set of variables,
that is, $(I\colon t_i)$ is an associated prime of $I$, 
then ${\rm v}(I)\leq 1$ and, by
Proposition~\ref{ci-vnumber}, 
${\rm v}(I\colon t_i)=0$. Thus, we may assume
that $(I\colon t_i)$ is not an associated prime of $I$. 
By Lemma~\ref{v-number-clutters-graphs-lemma}(a) and Theorem~\ref{v-number-clutters-graphs}, we
can pick a squarefree monomial $f$ and an associated prime $\mathfrak{p}$ of $(I\colon t_i)$ such
that
$$
((I\colon t_i)\colon f)=\mathfrak{p}\ \mbox{ and }\
\deg(f)={\rm v}(I\colon t_i).
$$
\quad Then, $(I\colon t_if)=\mathfrak{p}$ and, by definition of 
${\rm v}(I)$, one has ${\rm v}(I)\leq 1+\deg(f)=1+{\rm v}(I\colon t_i)$.
\end{proof}

If $v$ is a vertex of a graph $G$, then we denote the neighborhood of
$v$ by $N_G(v)$ and the closed neighborhood $N_G(v)\cup\{v\}$
of $v$ by $N_G[v]$. Recall that $N_G(v)$ is the set of all vertices 
of $G$ that are adjacent to $v$. The subgraph $G\setminus N_G[v]$ is denoted by $G_v$. 
Note that $G_v$ is
the induced subgraph $G[V(G)\setminus N_G[v]]$ on the vertex set 
$V(G)\setminus N_G[v]$ whose edges are all the edges of $G$ that are
contained in $V(G)\setminus N_G[v]$. 

\begin{lemma}\label{colon-dim} Let $G$ be graph and let $I=I(G)$ be
its edge ideal. If $t_i$ is a vertex of
$G$, then  
$$
\dim(S/(I\colon t_i))=1+\dim(K[V(G)\setminus
N_G[t_i]]/I(G\setminus N_G[t_i])).
$$
\end{lemma}

\begin{proof} One has $(I\colon t_i)=(I(G\setminus
N_G(t_i)),N_G(t_i))=(I(G_{t_i}),N_G(t_i))$. As $t_i$ does not occurs in a minimal
generating set of $I(G_{t_i})$, the equality follows. 
\end{proof}

\begin{proposition}\label{vnumber-comparison} Let $I=I(G)$ be the
edge ideal of a graph $G$. The 
following hold.
\begin{enumerate}
\item[\rm(a)] If $t_i$ is a vertex of $G$, then ${\rm v}(I)\leq
{\rm v}(I\colon t_i)+1$.
\item[\rm(b)] ${\rm v}(I\colon t_i)\leq {\rm v}(I)$ for some
$t_i\in V(G)$.
\item[\rm(c)] If ${\rm v}(I)\geq 2$, then ${\rm v}(I\colon t_i)<{\rm
v}(I)$ for some $t_i\in V(G)$.
\item[\rm(d)] If $t_i$ is a vertex of $G$, then ${\rm v}(I)\leq
{\rm v}(I,t_i)+1$.
\item[\rm(e)] ${\rm v}(I,t_i)\leq {\rm v}(I)$ for some
$t_i\in V(G)$.
\end{enumerate}
\end{proposition}

\begin{proof} (a): As $G$ is a simple graph, $\{t_i\}\notin E(G)$.
Thus, this part follows from Lemma~\ref{vnumber-comparison-clutter}.

(b): By Lemma~\ref{v-number-clutters-graphs-lemma}(a) and Theorem~\ref{v-number-clutters-graphs},
there exist a stable set $A$ of the graph $G$ and 
 an associated prime $\mathfrak{p}$ of $S/I$ such
that
\begin{equation*}
\mathfrak{p}=(I\colon t_A)=(N_G(A))\ \mbox{ and }\
\deg(t_A)={\rm v}(I).
\end{equation*}
\quad Pick a vertex $t_i$ not in $\mathfrak{p}$. Then, 
$\mathfrak{p}=(I\colon t_it_A)$. Indeed, the inclusion
``$\subset$'' is clear. To show the reverse inclusion take
$t^c\in(I\colon t_it_A)$. Then, $t_it^c\in(I\colon t_A)=\mathfrak{p}$.
As $t_i\notin\mathfrak{p}$, we get $t^c\in\mathfrak{p}$. Thus
$$
\mathfrak{p}=(I\colon t_A)=(I\colon t_it_A)=((I\colon t_i)\colon t_A),
$$
and consequently ${\rm v}(I\colon t_i)\leq\deg(t_A)={\rm v}(I)$. 

(c): Let $A$ and $\mathfrak{p}$ be as in part (b). Pick $t_i\in A$ and
set $B=A\setminus\{t_i\}$. Note that $B\neq\emptyset$ since 
${\rm v}(I)\geq 2$. Then 
\begin{equation*}
(I\colon t_A)=((I\colon t_i)\colon t_B)=\mathfrak{p},
\end{equation*}
and consequently ${\rm v}(I\colon t_i)\leq\deg(t_B)<\deg(t_A)={\rm v}(I)$. 

(d): By Proposition~\ref{additivity-vnumber}, one has
${\rm v}(I,t_i)={\rm v}(I(G\setminus t_i),t_i)={\rm
v}(I(G\setminus t_i))$. If $t_i$ is isolated, by
Proposition~\ref{additivity-vnumber}, ${\rm v}(I)={\rm v}(I,t_i)$.
Thus, we may assume that $t_i$ belongs to an edge of $G$. We may
also assume that $G\setminus t_i$ is
not a discrete graph; 
otherwise $G$ is a star $\mathcal{K}_{1,r}$ with some isolated
vertices, ${\rm v}(I)=1$ and
${\rm v}(I,t_i)=0$. 
By Lemma~\ref{v-number-clutters-graphs-lemma}(a) and Theorem~\ref{v-number-clutters-graphs},
there exist a stable set $A$ of $G\setminus t_i$ and an 
associated prime $\mathfrak{p}$ of $I(G\setminus t_i)$ such that
$$
(I(G\setminus t_i)\colon t_A)=\mathfrak{p}\ \mbox{ and }\
\deg(t_A)={\rm v}(I(G\setminus t_i))={\rm v}(I(G\setminus t_i),t_i)={\rm v}(I,t_i).
$$
\quad 
Note that $t_i\notin A$, 
$t_i\notin\mathfrak{p}$, and $t_A\notin\mathfrak{p}$.
We consider two cases.

Case (I): Assume that $N_G(t_i)\cap A=\emptyset$. We claim that
$(I\colon t_it_A)=(\mathfrak{p},N_G(t_i))$. The inclusion
``$\supset$'' is clear. To show the inclusion ``$\subset$'' take $t^c$
in $(I\colon t_it_A)$, that is, $t^ct_it_A=t^bt_e$ for some $e\in
E(G)$ and some monomial $t^b$. If $t_i$ is not in $e$, then $t^c$ is
in $(I(G\setminus t_i)\colon t_A)$, that is, $t^c$ is in
$\mathfrak{p}$. If $t_i$ is in $e$, then $e=\{t_i,t_k\}$ for some
$t_k$ in $N_G(t_i)$. As $t_k$ is not in $A$ since $N_G(t_i)\cap
A$ is empty and $t_k\neq t_i$, we get
that $t_k$ divides $t^c$ and $t^c\in(N_G(t_i))$. This proves the
claim. Hence, in this
case one has ${\rm v}(I)\leq 1+{\rm v}(I,t_i)$.

Case (II): Assume that $N_G(t_i)\cap A\neq\emptyset$. We claim
that $(I\colon t_A)=(\mathfrak{p},t_i)$. First we show the
inclusion ``$\subset$''. Take $t^c$ in $(I\colon t_A)$, that is,
$t^ct_A=t^bt_e$ for some $e\in E(G)$ and some monomial $t^b$. If $t_i$
is not in $e$, then $t^c$ is in $(I(G\setminus t_i)\colon
t_A)=\mathfrak{p}$. If $t_i$ is in $e$, then $t_i$
divides $t^c$ since $t_i$ is not in $A$. Thus, $t^c$ is in
$(\mathfrak{p},t_i)$. Next we show the inclusion ``$\supset$''.
Clearly $\mathfrak{p}$ is contained in $(I\colon t_A)$. Pick $t_j$ in
$N_G(t_i)\cap A$. Then, $t_j\in A$ and there is $e\in E(G)$ of the
form $e=\{t_i,t_j\}$. Then, $t_it_A$ is in $I$ and $t_i$ is in
$(I\colon t_A)$. This proves the claim. Hence, ${\rm v}(I)\leq {\rm v}(I,t_i)$.

(e): Let $A$ and $\mathfrak{p}$ be as in part (b). Then,
$\mathfrak{p}$ is generated by a minimal vertex cover $C$ of $G$. Pick $t_i$ in
$C$. Then, $t_it_A=t^bt_e$ for some $e\in E(G)$ and some monomial
$t^b$.   
Note that $t_i\notin A$. Indeed, if $t_i\in A$, then $t_i^2$ divides
$t_it_A$, and $t_A$ must be a multiple of $t_e$, 
a contradiction since $A$ is stable. 
By Proposition~\ref{additivity-vnumber}, one has
$$
{\rm v}(I,t_i)={\rm v}(I(G\setminus t_i),t_i)={\rm
v}(I(G\setminus t_i)). 
$$
\quad If $C=\{t_i\}$, then $(I,t_i)=(t_i)$ since any edge of $G$ is
adjacent to $t_i$. Thus, $0={\rm v}(I,t_i)\leq{\rm v}(I)$. Assume
$C\neq\{t_i\}$. It suffices to prove the equality $(I(G\setminus t_i)\colon
t_A)=(C\setminus\{t_i\})$. To show the inclusion
``$\subset$'' take a minimal generator $t^c$ of
$(I(G\setminus  t_i)\colon t_A)$. Note that $t^c$ does not contains
$t_i$. Then, as $t^c\in\mathfrak{p}=(C)$, we get $t^c\in
(C\setminus\{t_i\})$. 
To show the inclusion ``$\supset$''
take a variable $t_k$ in $C\setminus\{t_i\}$. Then
$t_kt_A\in I$ and $t_k$ is in $(I(G\setminus t_i)\colon t_A)$ 
since $t_i$ does not divides $t_kt_A$.
\end{proof}

\begin{theorem}\label{vertex-decomposable-vnumber} If $\Delta$ is a vertex decomposable simplicial
complex with vertex set $V$ and $I_\Delta$ is its Stanley--Reisner
ideal over a field $K$, then
$${\rm v}(I_\Delta)\leq{\rm reg}(K[V]/I_\Delta).
$$
\end{theorem}

\begin{proof} Setting $I=I_\Delta$ and $S=K[V]$, we proceed by
induction on $V$. The case $|V|=2$ is easy
to show. We may assume that $\Delta$ is not a simplex, otherwise 
${\rm reg}(S/I)={\rm v}(I)=0$. 
Assume $|V|\geq 3$ and pick a shedding vertex $v$ in $V$. The
minimal generating set of $(I\colon v)$ does not contains
$v$. Then, by Lemma~\ref{addvariables}, we obtain
\begin{equation}\label{jan9-20-0}
{\rm reg}(S/(I\colon v))={\rm reg}(K[V\setminus\{v\}]/(I\colon
v))= {\rm reg}(S/((I\colon v)+(v))).
\end{equation}
\quad To use
induction recall that ${\rm lk}_\Delta(v)$ is vertex decomposable
since $v$ is a shedding vertex and note that---according to
\cite[p.~25]{Ha-Woodroofe}---one has 
\begin{equation}\label{jan9-20-1}
I_{{\rm lk}_\Delta(v)}=(v)+I_{{\rm star}_\Delta(v)}=(v)+(I\colon v).
\end{equation}
\quad 
\quad Hence, using \cite[Theorem~1.5]{Ha-Woodroofe},
Eqs.~(\ref{jan9-20-0})--(\ref{jan9-20-1}) and induction, we get  
\begin{eqnarray*}
{\rm reg}(S/I)&=&\max\{{\rm reg}(S/(I\colon v))+1,\,{\rm
reg}(S/(I,v))\}\geq {\rm reg}(S/(I\colon v))+1\\
&= &{\rm reg}(S/(I\colon v)+(v))+1={\rm reg}(S/I_{{\rm
lk}_\Delta(v)})+1\geq {\rm v}(I_{{\rm
lk}_\Delta(v)})+1\\
&=&{\rm v}(I\colon v)+1\geq {\rm v}(I).
\end{eqnarray*}
\quad The last equality follows from the additivity of the v-number 
of Proposition~\ref{additivity-vnumber}, and the last inequality follows from
Lemma~\ref{vnumber-comparison-clutter} since $\{v\}\in\Delta$ for
all $v\in V$.
\end{proof}

\begin{corollary}\label{circuits-matroid} Let $M$ be a matroid on $X=\{t_1,\ldots,t_s\}$ and
let $\mathcal{C}_M$ be its set of circuits. If 
$I=(t_e\vert\, e\in\mathcal{C}_M)$ is the ideal of circuits of $M$, 
then ${\rm v}(I)\leq {\rm reg}(S/I)$.
\end{corollary}

\begin{proof} Let $\Delta$ be the independence complex of $M$, that
is, the faces of $\Delta$ are the independent sets of $M$. Then, the
Stanley--Reisner ideal $I_\Delta$ of $\Delta$ is equal to $I$. 
By a result of \cite[Theorem~3.2.1]{provan-billera}, $\Delta$ is
vertex decomposable and the inequality follows from
Theorem~\ref{vertex-decomposable-vnumber}. 
\end{proof}

\begin{corollary}\label{vnumber-chordless-criterion} 
If $G$ is a graph with no chordless
cycles of length other 
than $3$ or
$5$, then 
$${\rm v}(I(G))\leq {\rm reg}(S/I(G)).
$$
\end{corollary}

\begin{proof} By \cite[Theorem~1]{Woodroofe}, the independence
complex $\Delta(G)$ of $G$ is vertex decomposable, see 
Theorem~\ref{Woodroofe-vertex-deco}. Since $I_{\Delta(G)}=I(G)$, the
result follows from Theorem~\ref{vertex-decomposable-vnumber}.
\end{proof}

\begin{corollary}\label{vnumber-scm-bipartite} 
If $G$ is a bipartite graph and $S/I(G)$ is sequentially
Cohen-Macaulay, then  
$${\rm v}(I(G))\leq {\rm reg}(S/I(G)).
$$
\end{corollary}

\begin{proof} By a result of \cite{vantuyl}, the independence complex
$\Delta(G)$ of $G$ is vertex decomposable. Since $I_{\Delta(G)}=I(G)$, the
result follows from Theorem~\ref{vertex-decomposable-vnumber}.
\end{proof}

\begin{lemma}\label{nov13-19} If $\mathcal{C}$ is a clutter and 
$I_c(\mathcal{C})$ is its ideal of covers, then 
$\alpha_0(\mathcal{C})-1\leq{\rm v}(I_c(\mathcal{C}))$.
\end{lemma}

\begin{proof} Let $\mathcal{C}^\vee$ be the clutter of minimal vertex
covers of $\mathcal{C}$. By
Lemma~\ref{v-number-clutters-graphs-lemma}(a) 
and Theorem~\ref{v-number-clutters-graphs}, there is a stable set $A$ of $\mathcal{C}^\vee$
and $\mathfrak{p}\in{\rm Ass}(I_c(\mathcal{C}))$ such that
$$(I_c(\mathcal{C})\colon t_A)=(N_{\mathcal{C}^\vee}(A))\mbox{ and
}\deg(t_A)={\rm v}(I_c(\mathcal{C})).
$$ 
\quad Pick $t_i$ in $N_{\mathcal{C}^\vee}(A)$. Then $\{t_i\}\cup A$
contains an edge $c$ of $\mathcal{C}^\vee$. As $c$ is a minimal vertex
cover of $\mathcal{C}$, one has $\alpha_0(\mathcal{C})\leq |c|\leq
1+|A|\leq 1+{\rm v}(I_c(\mathcal{C}))$. 
\end{proof}

\begin{proposition}\label{nov14-19} Let $G$ be a graph and 
let $I_c(G)$ be its ideal of covers. If the independence complex
$\Delta_G$ of $G$ is pure and shellable, then 
$${\rm v}(I_c(G))={\rm
reg}(S/I_c(G))=\alpha_0(G)-1.$$
\end{proposition}

\begin{proof} We may assume that $\alpha_0(G)\geq 2$. If
$\alpha_0(G)=1$, then $I_c(G)$ is
generated by a set of variables and the result is clear. As $\Delta_G$ is pure and shellable, by 
\cite[Theorem~1.4(c)]{hhz-ejc}, $I_c(G)$ has linear quotients, that
is, we can order the minimal vertex covers of $G$ as $c_1,\ldots,c_r$
such that for all $1\leq i\leq r-1$,
the colon ideal $((t_{c_1},\ldots,t_{c_{i}})\colon t_{c_{i+1}})$ is
generated by variables of $S$. It is known that edge ideals with
linear quotients generated by monomials of the same degree $d$ have a
$d$-linear resolution \cite[Lemma~5.2]{Faridi}. Thus, ${\rm
reg}(S/I_c(G))=\alpha_0(G)-1$. Therefore, by Lemma~\ref{nov13-19}, we
need only show the inequality ${\rm v}(I_c(G))\leq \alpha_0(G)-1$.
Since $(t_{c_1}\colon t_{c_2})$ is generated by variables, there are
variables $t_k\in (t_{c_1}\colon t_{c_2})$ and $t_j$, $j\neq k$, such
that $t_kt_{c_2}=t_jt_{c_1}$. Setting $A:=c_1\setminus\{t_k\}$ and
noticing that $t_k$ (resp. $t_j$) does not divides $t_{c_2}$ (resp.
$t_{c_1}$), we get the equalities 
$$
A\cup\{t_k\}=c_1\mbox{ and }A\cup\{t_j\}=c_2.
$$
\quad The second equality follows from the inclusion
$A\cup\{t_j\}\subset c_2$ because $|c_1|=|c_2|$. Hence,
$\{t_j,t_k\}\subset(I_c(G)\colon t_A)$. Note that $A$ contains none of
the $c_i$'s, that is, $A$ is a 
stable set of the clutter $G^\vee$ of minimal vertex covers of $G$.
Therefore, from the embedding 
$$
S/(I_c(G)\colon t_A)\xhookrightarrow[]{\ \ t_A\ \ }S/I_c(G),
$$
and observing that any associated prime $\mathfrak{p}$ of $I_c(G)$ is of the form
$\mathfrak{p}=(e)$ for some $e\in E(G)$ it follows that $(t_j,t_k)$ is an associated
prime of $S/(I_c(G)\colon t_A)$ and this is the only associated prime
of $S/(I_c(G)\colon t_A)$. Thus, $(I_c(G)\colon t_A)=(t_j,t_k)$ and
${\rm v}(I_c(G))\leq \deg(t_A)=\alpha_0(G)-1$.
\end{proof}

A graph $G$ is called {\it chordal\/} if 
every cycle $C_r$ of $G$ of length $r\geq 4$ has an 
edge joining two 
non-adjacent vertices of $C_r$. A vertex $t_i$ of a graph $G$ is
called  {\it simplicial\/} if the subgraph $G[N_G(t_i)]$ induced by the neighbor 
set $N_G(x)$ is a complete subgraph. An \textit{induced matching} in
a graph $G$ is a set of
pairwise disjoint edges $f_1,\ldots,f_r$ such that the only edges of 
$G$ contained in $\cup_{i=1}^rf_i$ are $f_1,\ldots,f_r$. The
\textit{induced matching number} of $G$, denoted by ${\rm im}(G)$,
is the number of edges in a largest induced matching. 

\begin{proposition}\label{dec7-19} Let $G$ be a graph without isolated vertices. If
the edge ideal $I(G)$ of $G$ has a linear resolution, then 
${\rm v}(I(G))={\rm reg}(S/I(G))=1$.
\end{proposition}

\begin{proof} Since $I(G)$ has a $d$-linear resolution 
with $d=2$, directly from the minimal free resolution of $S/I(G)$, we get
${\rm reg}(S/I(G))=d-1=1$. By
Theorem~\ref{v-number-clutters-graphs}, to show ${\rm v}(I(G))=1$ 
it suffices to show that $\{t_i\}$ is in $\mathcal{A}_G$ for some
vertex $t_i$, that is, we need only show that the neighbor set
$N_G(t_i)$ is a minimal vertex cover of $G$ for some $t_i\in V(G)$.
Using a result of Fr\"oberg \cite{Fro4}, one has that $I(G)$ has a linear
resolution if and only if the complement $\overline{G}$ of $G$ is a
chordal graph. Then, according to \cite{Toft}, there exists a
simplicial vertex $t_i$ of $\overline{G}$, that is, the subgraph
$\overline{G}[N_{\overline{G}}(t_i)]$ induced by the neighbor set
$N_{\overline{G}}(t_i)$ is a complete subgraph of $\overline{G}$.
Next, we show the equality 
\begin{equation}\label{sep2-19}
N_G(t_i)=V(G)\setminus N_{\overline{G}}[t_i],
\end{equation}
where $N_{\overline{G}}[t_i]:=N_{\overline{G}}(t_i)\cup\{t_i\}$ 
is the closed neighborhood of $t_i$ in
$\overline{G}$. To show the inclusion ``$\subset$'' take $t_j\in
N_G(t_i)$. Then, $\{t_i,t_j\}\in E(G)$ for some vertex $t_j$. Thus, 
$\{t_i,t_j\}$ is not in $E(\overline{G})$ and $t_j$ is not in
$N_{\overline{G}}[t_i]$.  Now we show the inclusion ``$\supset$''.
Take $t_j\notin N_{\overline{G}}[t_i]$, that is, $t_j\neq t_i$, and
$\{t_i,t_j\}$ is not in $E(\overline{G})$. Thus $\{t_i,t_j\}$ is in
$E(G)$ and $t_j\in N_G(t_i)$. Since $N_{\overline{G}}(t_i)$ induces a
complete subgraph of $\overline{G}$, so does
$N_{\overline{G}}\left[t_i\right]$. Thus
$N_{\overline{G}}\left[t_i\right]$ is a stable set of $G$.
Furthermore 
$N_{\overline{G}}\left[t_i\right]$ is a maximal stable set of $G$.
Indeed, if $t_k\cup N_{\overline{G}}\left[t_i\right]$ is a stable set
of $G$ for some $t_k\in V(G)\setminus N_{\overline{G}}\left[t_i\right]$. Then $\{t_k\}\cup
N_{\overline{G}}\left[t_i\right]$ induces a complete subgraph of
$\overline{G}$. Thus $t_k\in N_{\overline{G}}\left[t_i\right]$, a
contradiction. 
Therefore, by Eq.~(\ref{sep2-19}), $N_G(t_i)$ is a minimal vertex cover of $G$ since 
$N_{\overline{G}}\left[t_i\right]$ is a maximal stable set of $G$. 
\end{proof}

Let $G$ be a graph with vertex set $V=\{t_1,\ldots,t_s\}$ and let 
$U=\{u_1,\ldots, u_s\}$ be a new set of vertices. The {\it whisker graph} or {\it
suspension\/} of $G$, denoted by
$W_G$, is the graph obtained from $G$ by attaching to each 
vertex $t_i$ a new vertex $u_i$ and a new edge $\{t_i,u_i\}$.  
The edge $\{t_i,u_i\}$ is called a {\it whisker} or \textit{pendant
edge}. 

\begin{theorem}\label{nov4-19} Let $G$ be a graph, let $i(G)$ be its
independent domination number and let $W_G$ be its whisker graph. Then
\begin{enumerate}
\item[\rm(a)] ${\rm v}(I(W_G))=i(G)$, and 
\item[\rm(b)] ${\rm v}(I(W_G))\leq{\rm reg}(K[V(W_G)]/I(W_G))$.
\end{enumerate}
\end{theorem}

\demo (a): Let $A$ be a maximal stable set of $G$ such that
$i(G)=|A|$. Then, $N_G(A)=V(G)\setminus A$ and $N_G(A)$ is a minimal
vertex cover of 
$G$. As $A$ is stable in $W_G$ and $N_{W_G}(A)$ is a vertex
cover of $W_G$, by Lemma~\ref{v-number-clutters-graphs-lemma}(a), 
$(I(W_G)\colon t_A)=(N_{W_G}(A))$. Thus, ${\rm
v}(I(W_G))\leq\deg(t_A)=i(G)$. 

Next, we show the inequality ${\rm v}(I(W_G))\geq i(G)$. By 
Lemma~\ref{v-number-clutters-graphs-lemma}(a) and
Theorem~\ref{v-number-clutters-graphs}, there is a stable 
set $A\cup B$ of $W_G$, $A\subset V(G)$,
$B\subset\{u_1,\ldots,u_s\}$, such that 
$$
\mathfrak{p}:=(N_{W_G}(A\cup B))=(I(W_G)\colon t_Au_B),
$$
and ${\rm v}(I(W_G))=\deg(t_Au_B)$. If for some $i$, $u_i\in B$ and
$t_i\in N_G(A)$, then the neighbor sets in $W_G$ of $A\cup B$ and
$A\cup (B\setminus\{u_i\})$ are equal. Indeed, take $v$ in
$N_{W_G}(A\cup B)$, that is, $\{v,v'\}$ is in $E(W_G)$ for some $v'\in
A\cup B$. We may assume $v'=u_i\in B$, otherwise $v$ is in
$N_{W_G}(A\cup (B\setminus\{u_i\}))$. Then, $v=t_i\in N_G(A)$, and
$v$ is in $N_{W_G}(A\cup (B\setminus\{u_i\}))$. The other inclusion is
clear. 
Therefore, by Lemma~\ref{v-number-clutters-graphs-lemma}(a), we get
$$
\mathfrak{p}:=(N_{W_G}(A\cup B))=(N_{W_G}(A\cup
(B\setminus\{u_i\})))=(I(W_G)\colon t_Au_{B\setminus\{u_i\}}). 
$$
\quad Thus, ${\rm v}(I(W_G))<\deg(t_Au_B)$, a contradiction. Hence,
for all $u_i$ in $B$ one has $t_i\notin N_G(A)$, and $t_i\notin A$
because $A\cup B$ is stable in $W_G$, that is, for all $u_i\in B$
one has that $A\cup\{t_i\}$ is a stable set of both $W_G$ and $G$.
Hence, the set 
$$
D:=(A\cup\{t_i\})\cup(B\setminus\{u_i\})
$$
is a stable set of $W_G$.  We claim that $C:=N_{W_G}(D)$ is a
minimal vertex cover of $W_G$. By
Lemma~\ref{v-number-clutters-graphs-lemma}(b) it suffices to show that
$C$ is a vertex cover of $W_G$. Take $e$ an edge of $W_G$. 
If $u_i\in e$, then $t_i$ is in $e$, $u_i\in N_{W_G}(t_i)$, and $e\cap
C\neq\emptyset$. Now, assume $u_i\notin e$. We may assume that $e\cap
N_{W_G}(A\cup\{t_i\})=\emptyset$.
As $N_{W_G}(A\cup B)$ is a
vertex cover of $W_G$, one has $e\cap N_{W_G}(B)\neq\emptyset$. Then,
we can pick $t_k$ in $e\cap N_{W_G}(B)$. If $e\cap B=\emptyset$, then
$e=\{t_k,t_\ell\}$ for some $\ell$. If $k=i$, then $t_\ell$ is in
$e\cap N_{W_G}(t_i)$, a contradiction. If $k\neq i$, then $t_k$ is in
$N_{W_G}(B\setminus\{u_i\})$ since $\{t_k,u_k\}$ is an edge of $W_G$,
and $e\cap C\neq\emptyset$. If $e\cap B\neq\emptyset$, then
$e=\{t_k,u_k\}$ for some $k$ and $k\neq i$ since $u_i\notin e$. Thus, $t_k$ is in
$N_{W_G}(B\setminus\{u_i\})$ and $e\cap C\neq\emptyset$. This proves
the claim. As $D$ is stable and $N_{W_G}(D)$ is a minimal vertex cover of $W_G$, by
Lemma~\ref{v-number-clutters-graphs-lemma}(a), we get the equality
$$
\mathfrak{q}:=(I(W_G)\colon
t_{A\cup\{t_i\}}u_{B\setminus\{u_i\}})=(N_{W_G}((A\cup\{t_i\})\cup(B\setminus\{u_i\}))),
$$
$\mathfrak{q}$ is an associated prime of $I(W_G)$, and ${\rm
v}(I(W_G))=\deg(t_{A\cup\{t_i\}}u_{B\setminus\{u_i\}})$. If
$B\setminus\{u_i\}\neq\emptyset$, we can repeat the previous argument
with $A\cup\{t_i\}$ playing the role of $A$ and $B\setminus\{u_i\}$ 
playing the role of $B$, as many times as necessary until we get
$B\setminus\{t_i\}=\emptyset$. At the end of this process we obtain a
stable set $A'$ of $W_G$ such that $A'\subset V(G)$,
\begin{equation}\label{nov3-09}
\mathfrak{p}':=(I(W_G)\colon t_{A'})=(N_{W_G}(A')),
\end{equation} 
$\mathfrak{p}'$ an associated prime of $I(W_G)$, and ${\rm
v}(I(W_G))=\deg(t_{A'})$. The set $A'$ is also a maximal stable set
of $G$. Indeed, $A'$ is stable in $G$ since $A'$ is stable in $W_G$.
To show maximality, take $t_k\notin A'$, 
then $\{t_k,u_k\}\cap N_{W_G}(A')\neq\emptyset$ because $N_{W_G}(A')$
is a vertex cover of $W_G$. If $u_k$ is in
$N_{W_G}(A')$, then $t_k\in A'$, a contradiction. Thus, $t_k$ is in
$N_{W_G}(A')$, and $t_k$ is in $N_{G}(A')$. This proves that $A'$ is
maximal. By Eq.~(\ref{nov3-09}), one has 
$(I(G)\colon t_{A'})=(N_G(A'))$, and since $A'$ is a maximal stable
set of $G$, we get $i(G)\leq \deg(t_{A'})={\rm v}(I(W_G))$.

(b): It is not hard to see that the induced matching number 
${\rm im}(W_G)$ of the whisker graph $W_G$ is $\beta_0(G)$. Therefore, by
part (a) and \cite[Lemma~2.2]{katzman1}, one has 
$$
{\rm v}(I(W_G))=i(G)\leq\beta_0(G)={\rm im}(W_G)\leq {\rm
reg}(K[V(W_G)]/I(W_G)).\quad \Box$$

\section{Edge-critical and $W_2$ graphs, and the second symbolic power of edge
ideals}\label{edge-critical-section}

By a result of Staples \cite[p.~199]{Staples} a graph $G$ is in $W_2$ if and only if
$\beta_0(G\setminus v)=\beta_0(G)$ and $G\setminus v$ is well covered
for all $v\in V(G)$.

\begin{theorem}\cite[Theorem~2.2]{Levit-Mandrescu}\label{w2-1-well-covered}
Let $G$ be a graph without isolated vertices. Then, $G$ is in $W_2$ if and only if
$G$ is $1$-well-covered. 
\end{theorem}

\begin{lemma}\cite[Proposition~2.5]{Levit-Mandrescu}\label{sep6-19} 
If $G$ is a $W_2$ graph, then $G$ does not contain a vertex $u$ and a
stable set $A$, such that $u\notin A$ and $N_G(u)\subset N_G(A)$.
\end{lemma}

\begin{theorem}\label{W2-graphs} 
Let $G$ be a graph
without isolated vertices. Then, $G$ is in $W_2$ if and only if 
$G$ is well-covered and $\mathcal{F}_G=\mathcal{A}_G$.
\end{theorem}

\begin{proof} $\Rightarrow$) The graph $G$ is $1$-well-covered by
Theorem~\ref{w2-1-well-covered}. By Theorem~\ref{v-number-clutters-graphs} one
has $\mathcal{F}_G\subset\mathcal{A}_G$. To show the reverse 
inclusion take $A\in\mathcal{A}_G$, that is, $A$ is a stable set of
$G$ and $N_G(A)$ is a minimal vertex cover of $G$. By contradiction
assume that there is a maximal stable set $F$ of $G$ with $A\subsetneq
F$. Pick $t_i\in F\setminus A$. We claim that $N_G(t_i)\subset
N_G(A)$. Let $t_j$ be a vertex in $N_G(t_i)$, then $\{t_i,t_j\}$ is in
$E(G)$, and $t_j\notin F$ since $F$ contains no edges of $G$. Then,
$t_j\in N_G(F)$. By Lemma~\ref{v-number-clutters-graphs-lemma}(d), $N_G(F)$ is a minimal
vertex cover of $G$. As $N_G(A)\subset N_G(F)$, we get
$N_G(A)=N_G(F)$. 
Thus, $t_j\in N_G(A)$. This proves the
inclusion $N_G(t_i)\subset N_G(A)$, a contradiction to
Lemma~\ref{sep6-19}.

$\Leftarrow$) Let $t_i$ be a vertex of $G$ and
let $G':=G\setminus t_i$. 
Take a maximal stable set $A'$ of $G'$. 
As $G$ is well-covered, 
by Theorem~\ref{w2-1-well-covered}, we need only show $\beta_0(G)=|A'|$.  
If $A'\cap N_G(t_i)\neq\emptyset$, then $A'$ is a
maximal stable set of $G$ and $|A'|= \beta_0(G)$.
Now we assume $A'\cap N_G(t_i)=\emptyset$. Then, using
Lemma~\ref{v-number-clutters-graphs-lemma}(d) and noticing $t_i\notin
N_G(A')$, we get

(i) $N_G(t_i)\subset N_{G'}(A')=N_G(A')$, and 

(ii) $N_{G'}(A')$ is a minimal vertex cover of $G'$.

We claim that $A'\in\mathcal{A}_G$. Clearly $A'$ is a stable set of
$G$. Thus, we need only show that $N_G(A')$ is a minimal vertex cover of
$G$. Take an edge $e$ of $G$. If $t_i\in e$, one can write $e=\{t_i,t_j\}$ with
$t_j\in N_G(t_i)$. Then, by (i), $t_j\in N_G(A')$. If $t_i\notin e$,
then $e\in E(G')$. Thus, by (ii), $e$ contains a vertex of
$N_{G'}(A')=N_G(A')$. Hence, $N_G(A')$ is a vertex cover of $G$.  
That $N_G(A')$ is minimal follows from
Lemma~\ref{v-number-clutters-graphs-lemma}(b). 
This proves that $A'$ is in $\mathcal{A}_G$. Therefore, $A'$ is
in $\mathcal{F}_G$ and $\beta_0(G)=|A'|$ since any element of
$\mathcal{F}_G$ has cardinality $\beta_0(G)$.
\end{proof}

\begin{lemma}\label{additivity-reg-w2-v} Let $G$ be a graph and let $G_1,\ldots,G_r$ be its
connected components.
\begin{enumerate} 
\item[\rm(1)] \cite[p.~263]{Levit-Mandrescu} 
$G$ is in $W_2$ if and only if $G_i$ is in $W_2$ for all $i$.
\item[\rm(2)] 
 ${\rm reg}(S/I(G))=\sum_{i=1}^r{\rm
reg}(K[V(G_i)]/I(G_i))$.
\item[\rm(3)] ${\rm v}(I(G))=\sum_{i=1}^r{\rm v}(I(G_i))$.
\end{enumerate}
\end{lemma}

\begin{proof} (2), (3): These follow from
Propositions~\ref{additivity-reg} and  \ref{additivity-vnumber},
respectively.
\end{proof}

\begin{theorem}\label{reg-w2} Let $G$ be a graph without isolated vertices and let
$I=I(G)$ be its edge ideal. Then, $G$ is in $W_2$  if and only if
${\rm v}(I)=\dim(S/I)$.  
\end{theorem}

\begin{proof} 
$\Rightarrow$) By Theorem~\ref{W2-graphs}, $G$ is well covered and
$\mathcal{F}_G=\mathcal{A}_G$. Hence, by Corollary~\ref{sep3-19},
one has the equality ${\rm v}(I)=\dim(S/I)$.

$\Leftarrow$) By Corollary~\ref{sep3-19}, $G$ is
well-covered and $\mathcal{F}_G=\mathcal{A}_G$. Then, by
Theorem~\ref{W2-graphs}, $G$ is in $W_2$. 
\end{proof}

\begin{lemma}\label{nov27-19} Let $\mathcal{C}$ be a clutter and let
$I=I(\mathcal{C})$ be its edge ideal. The following hold.
\begin{enumerate}
\item[\rm(a)] $\mathcal{C}$ is edge-critical if and only if
$\beta_0(\mathcal{C}\setminus e)=\beta_0(\mathcal{C})+1$ for all $e\in
E(\mathcal{C})$.
\item[\rm(b)] $\mathcal{C}$ is edge-critical if and only if 
$\dim(S/(I(\mathcal{C}\setminus e)\colon t_e))=
\dim(S/I(\mathcal{C}))+1$ for all $e\in E(\mathcal{C})$.
\item[\rm(c)] If $\mathcal{C}$ is edge-critical, then $\dim(S/(I(\mathcal{C}\setminus e)\colon t_e))=
\dim(S/I(\mathcal{C}\setminus e))$ for all $e\in E(\mathcal{C})$.
\item[\rm(d)] A graph $G$ is edge-critical if and only if
$\beta_0(G\setminus(N_{G}(t_i)\cup N_G(t_j)))=\beta_0(G)-1$ for all
edges $\{t_i,t_j\}$ of $G$. 
\end{enumerate}
\end{lemma}

\begin{proof} (a): It suffices to show that 
$\beta_0(\mathcal{C}\setminus e)\leq \beta_0(\mathcal{C})+1$ for any 
clutter $\mathcal{C}$ and any edge $e$ of $\mathcal{C}$. Pick a stable
set $A$ of $\mathcal{C}\setminus e$ with
$|A|=\beta_0(\mathcal{C}\setminus e)$. If $e\not\subset A$, then $A$ is
a stable set of $\mathcal{C}$ and $\beta_0(\mathcal{C}\setminus e)\leq
\beta_0(\mathcal{C})$. If $e\subset A$, take $t_i\in e$. Then
$A\setminus\{t_i\}$ is a stable set of $\mathcal{C}$ and $|A|-1\leq
\beta_0(\mathcal{C})$.

(b): Let $e$ be any edge of $\mathcal{C}$. We set $d:=|e|$. 
Pick $\mathfrak{p}\in{\rm Ass}(I(\mathcal{C}\setminus e))$ such that
$\dim(S/\mathfrak{p})$ is equal to $\dim(S/I(\mathcal{C}\setminus
e))$. Consider the exact sequence
\begin{equation}\label{nov23-19}
0\longrightarrow S/(I(\mathcal{C}\setminus e)\colon t_e)[-d]\stackrel{t_e}
{\longrightarrow} S/I(\mathcal{C}\setminus e) \longrightarrow
S/I(\mathcal{C})\longrightarrow 0,
\end{equation}
and set $N:=S/(I(\mathcal{C}\setminus e)\colon t_e)$,
$M:=S/I(\mathcal{C}\setminus e)$, and $L:=S/I(\mathcal{C})$. From
Eq.~(\ref{nov23-19}) we get that either $\mathfrak{p}$ is in ${\rm
Ass}(N)$ or $\mathfrak{p}$ is in 
${\rm Ass}(L)$.

(b): $\Rightarrow$) By part (a) one has $\beta_0(\mathcal{C}\setminus
e)=\beta_0(\mathcal{C})+1$, that is, $\dim(M)=\dim(L)+1$. Hence, as 
the first map of Eq.~(\ref{nov23-19}) is an inclusion, we get 
$\dim(N)\leq \dim(M)$. If $\mathfrak{p}\in{\rm Ass}(L)$, then 
$$\dim(L)\geq 
\dim(S/\mathfrak{p})=\dim(M)=\dim(L)+1,$$
a contradiction. Thus, $\mathfrak{p}$ is in ${\rm Ass}(N)$, and
$\dim(N)\geq\dim(M)$. Hence, $\dim(N)$ is equal to $\dim(M)$.
Therefore, $\dim(N)=\dim(M)=\dim(L)+1$. 

(b): $\Leftarrow$) The first map of Eq.~(\ref{nov23-19}) 
is an inclusion and by hypothesis $\dim(N)=\dim(L)+1$. Hence, 
one has $\dim(N)\leq\dim(M)$ and $\dim(L)+1\leq\dim(M)$. If
$\mathfrak{p}\in {\rm Ass}(L)$, then 
$$
\dim(L)+1\leq\dim(M)=\dim(S/\mathfrak{p})\leq\dim(L),
$$
a contradiction. Thus, $\mathfrak{p}\in{\rm Ass}(N)$ and
$\dim(M)\leq\dim(N)=\dim(L)+1$. This proves the equality
$\dim(L)+1=\dim(M)$, that is, $\mathcal{C}$ is edge-critical.

(c): This follows from parts (a) and (b).

(d): Let $e=\{t_i,t_j\}$ be an edge of $G$. From the equalities 
\begin{eqnarray}
&(I(G\setminus e)\colon t_e)=(I(G\setminus(N_G[t_i]\cup
N_G[t_j])),(N_G[t_i]\cup N_G[t_j])\setminus\{t_i,t_j\}),\ \ \ \ \ \ \
\ \ \ \ \ \ &\\
& V(G)\setminus((N_G[t_i]\cup
N_G[t_j])\setminus\{t_i,t_j\})=\{t_i,t_j\}\cup(V(G)\setminus(N_G[t_i]\cup
N_G[t_j])),&
\end{eqnarray}
and setting $G_e:=G\setminus(N_G[t_i]\cup N_G[t_j])$, we obtain
\begin{eqnarray}
& S/(I(G\setminus e)\colon t_e)\simeq K[\{t_i,t_j\}\cup
V(G_e)]/I(G\setminus e),\ \ \ \ \ \ \ \  \ \ \ \ \ \ \ \ \ \ \ \ \ \
\ \ \ \ \ \ \ \ \ 
\ &
\\  
&\dim(S/(I(G\setminus e)\colon
t_e))=2+\dim(K[V(G_e)]/I(G_e))=2+\beta_0(G_e),&\label{nov25-19}
\end{eqnarray}
where the $2$ on the right of Eq.~(\ref{nov25-19}) comes from the fact that $t_i$ and $t_j$ 
do not occur in a minimal generating 
set of $(I(G\setminus e)\colon t_e)$. By part (b), $G$ is
edge-critical if and only if 
\begin{equation}\label{nov25-19-1}
\dim(S/(I(G\setminus e)\colon t_e))=
\dim(S/I(G))+1=\beta_0(G)+1
\end{equation}
for all $e\in E(G)$. Therefore, by Eqs.~(\ref{nov25-19}) and
(\ref{nov25-19-1}), we obtain that the graph $G$ is edge-critical if and only if
$\beta_0(G_e)=\beta_0(G)-1$ for all $e\in E(G)$. 
\end{proof}

In \cite{Hoang-etal,hoang-gorenstein-second-jaco} the Cohen--Macaulay
property of the square and the symbolic square of the edge ideal of a graph is classified. 

\begin{theorem}\label{cm-second-power} Let $G$ be a graph. The following hold.
\begin{enumerate} 
\item[\rm(a)] {\rm \cite[Theorem~2.2]{Hoang-etal}} $I(G)^{(2)}$ is Cohen--Macaulay if and only
if $G$ is a Cohen--Macaulay graph and for any edge $e=\{t_i,t_j\}$ of
$G$, the subgraph $G_e:=G\setminus(N_G[t_i]\cup N_G[t_j]))$ is Cohen--Macaulay and
$\beta_0(G_e)=\beta_0(G)-1$. 
\item[\rm(b)] {\rm \cite[Theorem~4.4]{hoang-gorenstein-second-jaco}}
If $G$ has no isolated vertices, then $I(G)^2$ 
is Cohen--Macaulay if and only if $G$ is a triangle-free member of $W_2$. 
\end{enumerate}
\end{theorem}

The next result---together with the tables of edge-critical graphs
given in \cite{Plummer,B-Small}---allows us to give a list of all
connected graphs with fewer than $10$ vertices such that the symbolic
square of its edge ideal is Cohen--Macaulay over a field of
characteristic $0$ (Table~\ref{tab:C-M,G}). 

\begin{theorem}\label{c-m-edge-critical} If $G$ is a graph and $I(G)^{(2)}$ is Cohen--Macaulay,
then $G$ is edge-critical.
\end{theorem}

\begin{proof} By Theorem~\ref{cm-second-power}(a), 
$\beta_0(G_e)=\beta_0(G)-1$ for all $e\in E(G)$. Hence, by
Lemma~\ref{nov27-19}(d), $G$ is edge-critical.
\end{proof}

\begin{lemma}{\cite[Lemma~4.1]{Caviglia-et-al}}
\label{Caviglia-et-al-square-free} 
Let $I\subset S$ be a squarefree monomial ideal
and let $f$ be a squarefree monomial. Then 
${\rm depth}(S/(I\colon f))\geq{\rm depth}(S/I)$.
\end{lemma}

\begin{proposition}\label{additivity-cm-symbolic}
Let $S=K[T]$ and $B=K[U]$ be polynomial rings over a field $K$, 
let $I$ be a squarefree monomial ideal of $S$, let
$V$ be a subset $U$, and let $R=K[T,U]$. The following hold. 
\begin{itemize}
\item[(a)] $(I,u)^{(2)}=(I^{(2)},uI,u^2)$ for all $u\in U$.
\item[(b)] $(I,V)^{(2)}$ is Cohen--Macaulay if and only if $I^{(2)}$
is Cohen-Macaulay.
\end{itemize}
\end{proposition}

\begin{proof} (a): Let $\mathfrak{p}_1,\ldots,\mathfrak{p}_r$ be the
associated primes of $I$. The set of associated primes of $(I,u)$ 
is $\{(\mathfrak{p}_i,u)\}_{i=1}^r$ \cite[Lemma~3.4]{HaM}. 
The equality follows from the following expressions:
$$
I^{(2)}=\mathfrak{p}_1^2\cap\cdots\cap\mathfrak{p}_r^2\ \mbox{ and }\
(I,u)^{(2)}=(\mathfrak{p}_1,u)^2\cap\cdots\cap(\mathfrak{p}_r,u)^2.
$$

(b): We proceed by induction on $|V|$. Assume that $|V|=1$, that is,
$V=\{u\}$ for some variable $u$ in $U$. 
First assume that $(I,u)^{(2)}$ is Cohen--Macaulay. 
Let $\underline{h}=\{h_1,\ldots,h_d\}$ be a 
homogeneous system of parameters for $S/I^{(2)}$, that is,
$d=\dim(S/I)$ and ${\rm rad}(I^{(2)},\underline{h})=S_+$. By part (a)
it follows that $\underline{h}$ is a system of parameters for
$S[u]/(I,u)^{(2)}$ and, since this ring is Cohen--Macaulay, 
$\underline{h}$ is a regular sequence on $S[u]/(I,u)^{(2)}$. Then, it
is seen that $\underline{h}$ is a regular sequence on
$S/I^{(2)}$, and this ring is Cohen--Macaulay. Conversely, assume that $I^{(2)}$ is 
Cohen--Macaulay. Then, $I$ is Cohen--Macaulay since $I$ is the radical
of $I^{(2)}$ \cite{Radical-Herzog}. By part (a), one has
$$
((I,u)^{(2)}\colon u)=(I,u)\ \mbox{ and }\
((I,u)^{(2)},u)=(I^{(2)},u),
$$
and both ideals are Cohen--Macaulay. 
From the exact sequence 
\begin{equation}\label{dec2-19}
0\longrightarrow R/((I,u)^{(2)}\colon
u)[-1]\stackrel{u}
{\longrightarrow} R/(I,u)^{(2)} \longrightarrow
R/(I^{(2)},u)\longrightarrow 0
\end{equation}
and the depth lemma \cite[Lemma~2.3.9]{monalg-rev} 
it follows that $(I,u)^{(2)}$ is Cohen--Macaulay since all rings in
Eq.~(\ref{dec2-19}) have the same dimension.

To complete the induction process assume $V=\{u_1,\ldots,u_m\}$ and
$m\geq 2$. Then 
$$(I,V)^{(2)}=((I,u_1,\ldots,u_{m-1}),u_m)^{(2)},$$ 
and, by an appropriate application of the previous case, $(I,V)^{(2)}$ is Cohen--Macaulay if and only if
$(I,u_1,\ldots,u_{m-1})^{(2)}$ is Cohen--Macaulay. Then, by induction
on $m$,  $(I,V)^{(2)}$ is Cohen-Macaulay if and only if $I^{(2)}$ is
Cohen--Macaulay.
\end{proof}

\begin{lemma}\label{sunday-morning-dec15-19} Let $G$ be a well-covered graph without isolated
vertices. If ${\rm v}(I(G))=1$ and $G\setminus N_G[t_i]$ has no
isolated vertices for all $t_i\in V(G)$, then $G=\mathcal{K}_s$ is a
complete graph.
\end{lemma}

\begin{proof} As ${\rm v}(I(G))=1$,
by Theorem~\ref{v-number-clutters-graphs}, there is $t_i\in V(G)$ such
that $N_G(t_i)$ is a minimal vertex cover of $G$. Note that
$V(G)=N_G[t_i]$. Indeed, if $V(G)$ is not equal to $N_G[t_i]$, then
any vertex outside $N_G[t_i]$ will be an isolated vertex of
$G\setminus N_G[t_i]$, a contradiction. Hence, $\{t_i\}$ is a maximal
stable set of $G$, and $G=\mathcal{K}_s$ since $G$ is well-covered.
\end{proof}

\begin{theorem}\cite[Theorem~5]{Pinter-jgt}\label{Pinter-thm}
If a graph $G$ is in $W_2$ and $G$ is not complete, then
the subgraph $G_v:=G\setminus N_G[v]$ is
in $W_2$ and $\beta_0(G_v)=\beta_0(G)-1$ for all $v\in V(G)$.
\end{theorem}

The next result of Levit and Mandrescu gives a partial converse of
Theorem~\ref{Pinter-thm}. As an application of our classification of 
$W_2$ graphs in terms of the v-number and the independence number
(Theorem~\ref{reg-w2}), we give a proof of this result. 

\begin{theorem}\cite[Theorem~3.9]{Levit-Mandrescu}\label{Levit-Mandrescu-w2} Let $G$ be a
well-covered graph without isolated vertices. If $G\setminus N_G[t_i]$
is in $W_2$ for all $t_i\in V(G)$, then $G$ is in $W_2$.
\end{theorem}

\begin{proof} By Lemma~\ref{sunday-morning-dec15-19}, we may assume
that ${\rm v}(I(G))\geq 2$. We set $I=I(G)$.  
According to Proposition~\ref{vnumber-comparison}(c), ${\rm v}(I\colon
t_i)<{\rm v}(I)$ for some $t_i\in V(G)$.  By
Lemma~\ref{colon-dim}, one has 
$$
\dim(S/(I\colon t_i))=1+\dim(K[V(G)\setminus
N_G[t_i]]/I(G\setminus N_G[t_i])),
$$
and by Theorem~\ref{reg-w2} we get $\dim(
K[V(G\setminus N_G[t_i])]/I(G\setminus N_G[t_i]))={\rm
v}(I(G\setminus N_G[t_i]))$ since 
$G\setminus
N_G[t_i]$ is in $W_2$. From the equality 
$$
(I\colon t_i)=(I(G\setminus N_G[t_i]),N_G(t_i))
$$
and Proposition~\ref{additivity-vnumber} we get 
${\rm v}(I(G\setminus N_G[t_i]))={\rm v}(I\colon t_i)$. Altogether,
by Corollary~\ref{coro1-vnumber}, 
one has
$$ 
\dim(S/(I\colon t_i))=1+{\rm v}(I\colon t_i)\leq {\rm v}(I)\leq \dim(S/I).
$$
\quad Hence, as $G$ is well-covered, $\dim(S/(I\colon
t_i))=\dim(S/I)$, and we have equality everywhere. Then, ${\rm
v}(I)=\dim(S/I)$, and $G$ is in $W_2$ by Theorem~\ref{reg-w2}.
\end{proof}

\begin{proposition}\label{dec17-19} Let $I\subset S$ be a squarefree
monomial ideal and let $t_i$ be a variable. 
\begin{enumerate}
\item[\rm(a)] If $I^{(2)}$ is Cohen--Macaulay, then $(I\colon
t_i)^{(2)}$ is Cohen--Macaulay.
\item[\rm(b)] If $G$ is a graph and $I(G)^{(2)}$ is Cohen--Macaulay,
then $I(G\setminus N_G[t_i])^{(2)}$ is Cohen--Macaulay.
\end{enumerate} 
\end{proposition}

\begin{proof} (a): Let $\mathfrak{p}_1,\ldots,\mathfrak{p}_r$ be the
associated primes of $I$. Then, from the equalities
$$
(I^{(2)}\colon
t_i^2)=\bigcap_{t_i\notin\mathfrak{p}_k}\mathfrak{p}_k^2\ \mbox{ and
}\ (I\colon t_i)=\bigcap_{t_i\notin\mathfrak{p}_k}\mathfrak{p}_k,
$$
we obtain $(I^{(2)}\colon t_i^2)=(I\colon t_i)^{(2)}$. Hence, 
by Lemma~\ref{Caviglia-et-al-square-free}, $(I\colon t_i)^{(2)}$ is
Cohen--Macaulay.

(b): Since $(I(G)\colon t_i)=(I(G\setminus N_G[t_i]),N_G(t_i))$, by part
(a) and Proposition~\ref{additivity-cm-symbolic}, we obtain that $I(G\setminus N_G[t_i])^{(2)}$ is 
 Cohen--Macaulay.
\end{proof}

\begin{lemma}\label{edge-critical-isolated} 
Let $G$ be a graph without isolated vertices. If $G$ is
edge-critical and $t_i$ is a vertex of $G$, then $G\setminus N_G[t_i]$ has no isolated vertices. 
\end{lemma}

\begin{proof} Let $t_i$ be a vertex of $G$. If $\beta_0(G)=1$, then
any two vertices of $G$ are adjacent, and $G$ is a
complete graph. Thus, $G\setminus N_G[t_i]=\emptyset$. Now, we assume
$\beta_0(G)\geq 2$. 
Since $\beta_0$ is additive on the connected components of $G$, each
component of $G$ is edge-critical and we may assume that $G$ is
connected.  We proceed by contradiction assuming that
there exists a vertex $t_k$ of $G_{t_i}:=G\setminus N_G[t_i]$ which is
isolated in $G_{t_i}$. Then, $N_G(t_k)\subset N_G(t_i)$. We may assume that
$$N_G[t_i]=\{t_i,t_1,\ldots,t_m\}\ \mbox{ and }\
N_G(t_k)=\{t_1,\ldots,t_j\},
$$
where $1\leq j\leq m$ and $t_k\notin N_G[t_i]$. We set $e= \{t_i,t_j\}$ and
$G_e:=G\setminus(N_G[t_i]\cup N_G[t_j]))$. By Lemma~\ref{nov27-19}, 
one has $\beta_0(G_e)=\beta_0(G)-1\geq 1$. Pick a stable set $A$ of
$G_e$ with $|A|=\beta_0(G_e)$ and note that $A$ is a stable set of
$G$. 
We claim that $A\cup\{t_i,t_k\}$ is a
stable set of $G$. Let $f$ be any edge of $G$. It suffices to show
that $f\not\subset A\cup\{t_i,t_k\}$. If $f\subset
A\cup\{t_i,t_k\}$, then either $f=\{t_i,v\}$ with $v\in A$ or
$f=\{t_k,v\}$ with $v\in A$. Then, $v\in N_G(t_i)$ or $v\in N_G(t_k)$,
a contradiction since $N_G(t_k)\subset N_G(t_i)$ and
$A\cap(N_G[t_i]\cup N_G[t_j])=\emptyset$. This proves the claim.
Hence, $\beta_0(G)\geq \beta_0(G_e)+2=\beta_0(G)+1$, a contradiction.
\end{proof}

\begin{corollary}{\rm(\cite[Theorem~2.2]{Hoang-etal},
\cite[Lemma~8]{Hoang-VJM})}\label{c-m-w2}
If $G$ is a graph without
isolated vertices and $I(G)^{(2)}$ is Cohen--Macaulay, then $G$ is in
$W_2$. 
\end{corollary}

\begin{proof} We proceed by induction on $s=|V(G)|\geq 2$. If $s=2$,
then $G$ is a complete graph on $2$ vertices and $G$ is in $W_2$.
Assume $s\geq 3$. Let $t_i$ be any vertex of $G$. By
Theorem~\ref{c-m-edge-critical}, $G$ is edge-critical. Then, by
Lemma~\ref{edge-critical-isolated}, $G\setminus N_G[t_i]$ has no
isolated vertices. Using Proposition~\ref{dec17-19}, we obtain that 
$I(G\setminus N_G[t_i])^{(2)}$ is Cohen--Macaulay. 
Hence, by induction, $G\setminus N_G[t_i]$ is in $W_2$ for all
$t_i\in V(G)$. Now, the ideal $I(G)$ is Cohen--Macaulay because $I(G)$
is the radical of $I(G)^{(2)}$ \cite{Radical-Herzog}, and 
consequently $G$ is well-covered \cite[p.~269]{monalg-rev}. Therefore, by
Theorem~\ref{Levit-Mandrescu-w2}, 
$G$ is in $W_2$.
\end{proof}

 A triangle-free graph is \textit{maximal triangle-free}, if joining
 any two non-adjacent vertices creates a triangle. 
%A graph is 
%called \textit{triangle-critical} if joining any two non-adjacent 
%vertices increases the number of triangles. 
The next lemma was pointed
out to us by Carlos Valencia.

\begin{lemma}\label{carlitos-triangle-free} Let $G$ be a graph. The
following are equivalent.
\begin{enumerate}
\item[\rm(a)] $G$ is edge-critical and $\beta_0(G)=2$. 
\item[\rm(b)] $\overline{G}$ is a maximal triangle-free graph and
$|V(G)|\geq 2$. 
\end{enumerate}
\end{lemma}

\begin{proof} (a)$\Rightarrow$(b): $\overline{G}$ is triangle 
free, otherwise $\beta_0(G)\geq 3$. Take two non-adjacent vertices
$t_i,t_j$ of $\overline{G}$. Then, $e:=\{t_i,t_j\}$ is an edge of $G$
and $\beta_0(G\setminus e)=\beta_0(G)+1=3$ since $G$ is edge-critical.
Thus, there is a stable set $A$ of $G\setminus e$ with $|A|=3$.
Hence, $A$ is the set of vertices of a $3$-cycle of $\overline{G}+e$,
that is, $\overline{G}+e$ contains a triangle. This follows noticing
that $e\subset A$ since $A$ is not a stable set of $G$, and picking a
vertex $t_k$ of $G$ such that $A=\{t_k\}\cup e$. Then, the vertices
$t_i,t_j,t_k$ form a 3-cycle of $\overline{G}+e$.

(b)$\Leftarrow$(a): If $\beta_0(G)=1$, then $G=\mathcal{K}_s$ with 
$s\geq 2$, and $\overline{G}$ is a set of $s$ isolated vertices which
is not a maximal triangle-free graph, a contradiction. Thus,
$\beta_0(G)\geq 2$, and $\beta_0(G)=2$ since $\overline{G}$ is
triangle-free. To show that $G$ is edge-critical take any edge
$e=\{t_i,t_j\}$ of $G$. Then, $t_i$ is not adjacent to $t_j$ in
$\overline{G}$, and $\overline{G}+e$ must contain a triangle with
vertex set $A$ such that $e\subset A$. Then, $A$ is a stable set of $G\setminus e$, and
$\beta_0(G\setminus e)\geq 3$, as required. 
\end{proof}

The diameter of a connected graph $G$,
denoted ${\rm diam}(G)$, is the greatest distance between any two
vertices of $G$. For triangle-free
 graphs of order $s\geq 3$ being maximal triangle-free
 is equivalent to having diameter $2$.   

\begin{theorem}{\rm
\cite[Theorem~2.3]{Minh-Trung}}\label{Minh-Trung-jalgebra} Let
$I\subset S$ be the Stanley--Reisner ideal of a pure simplicial
complex $\Delta$ of dimension $1$. Then, $S/I^{(2)}$ is Cohen--Macaulay if and
only if ${\rm diam}(\Delta)\leq 2$.
\end{theorem}

\begin{theorem}\label{c-m-edge-critical-dim1} Let $G$ be a graph.
 If $\beta_0(G)=2$, then $I(G)^{(2)}$ is
Cohen--Macaulay if and only if $G$ is edge-critical.
\end{theorem}

\begin{proof} Assume that $I(G)^{(2)}$ is Cohen--Macaulay. Then, by
Theorem~\ref{c-m-edge-critical}, $G$ is edge-critical. Conversely 
assume that $G$ is edge-critical.  Let $\overline{G}$ be the
complement of $G$. Then, by
Lemma~\ref{carlitos-triangle-free}, $\overline{G}$ is a maximal
triangle-free graph and $|V(G)|\geq 2$. In particular, $\overline{G}$
is connected. If $s=2$, then $\overline{G}=\mathcal{K}_2$, $I(G)=(0)$,
and $I(G)$ is Cohen--Macaulay. If $s\geq 3$, then ${\rm diam}(G)=2$ and
since the independence complex of $G$ is $\overline{G}$ and the
Stanley--Reisner ideal $I_{\overline{G}}$ of $\overline{G}$ is $I(G)$, by
Theorem~\ref{Minh-Trung-jalgebra} we get that $I(G)^{(2)}$ is
Cohen--Macaulay.
\end{proof}

\begin{corollary}\label{c-m-edge-critical-dim1-coro} If $G$ is an edge-critical graph without isolated
vertices and $\beta_0(G)=2$, then $G$ is in $W_2$ and 
${\rm v}(I(G))={\rm reg}(S/I(G))=2$.
\end{corollary}

\begin{proof} By Theorem~\ref{c-m-edge-critical-dim1} and
Lemma~\ref{carlitos-triangle-free}, $I(G)^{(2)}$ is
Cohen--Macaulay and $\overline{G}$ is a connected graph. Therefore, by
Corollary~\ref{c-m-w2}, $G\in W_2$.  By a result of Provan
and Billera \cite[Theorem~3.1.2]{provan-billera} any $1$-dimensional
connected complex is vertex decomposable. Thus, $\overline{G}$ is
vertex decomposable and the Stanley--Reisner ideal $I_{\overline{G}}$
of $\overline{G}$ is $I(G)$. Hence, by 
Theorems~\ref{vertex-decomposable-vnumber}, \ref{reg-w2} and
Proposition~\ref{nov6-19}, one
has
$$
{\rm reg}(S/I(G))\geq {\rm v}(I(G))=\dim(S/I(G))\geq {\rm
reg}(S/I(G)).
$$
\quad Thus, we have equality everywhere. Note that
$\dim(S/I(G))=\beta_0(G)=2$. 
\end{proof}

\section{Examples}

\begin{remark}\label{edge-criticalcm} There are $53$ connected
edge-critical graphs with at most $9$ vertices and at least $2$
vertices \cite{Plummer,B-Small} of which $31$ have $9$ vertices. Using
Theorem~\ref{c-m-edge-critical} and
Procedure~\ref{Tai-symbolic-powers-algorithm}, 
in Table~\ref{tab:C-M,G} we show 
the list of all connected graphs such that the symbolic
square of its edge ideal is Cohen--Macaulay over a field of
characteristic $0$. Table~\ref{tab:C-M,G} consists of $19$ graphs with
fewer than $9$ vertices and $17$ graphs with $9$ vertices.
%making a total of $36$ graphs. 
%There are $45$ Cohen--Macaulay connected
%edge-critical graphs with at most $9$ vertices of which $25$ have $9$
%vertices. 
\end{remark}

\begin{example}\label{example-graph4} 
Let $G$ be the graph whose edge ideal is given by 
$$
I(G)=(t_1t_2,t_2t_3,t_3t_4,t_4t_5,t_5t_6,t_1t_6,t_1t_8,t_2t_9,t_3t_7,t_4t_8,t_5t_9,
t_6t_7,t_7t_8,t_7t_9,t_8t_9).
$$
\quad Using Procedure~\ref{Tai-symbolic-powers-algorithm} for \textit{Macaulay}$2$
\cite{mac2} we obtain that $I(G)^{(2)}$ is a Cohen--Macaulay ideal
over a field of 
characteristic $0$. Using
Procedure~\ref{nov11-19} we get that ${\rm v}(I(G))={\rm
reg}(S/I(G))=\beta_0(G)=3$, $G$ is a graph in $W_2$, and $G$ is
edge-critical. In the list of 36 graphs of Table~\ref{tab:C-M,G}, this graph 
corresponds to graph number $5$ from bottom.
%\begin{figure}[ht]
%\begin{displaymath}
%\xygraph{
% !{<0cm,.5cm>;<.5cm,0cm>;<0cm,.5cm>} 
%  !{(-1,3.2)}*+{\text{$t_1$}} 
% !{(1,3.2)}*+{\text{$t_2$}} 
%  !{(2.2,2)}*+{\text{$t_3$}} 
% !{(1,-.2)}*+{\text{$t_4$}} 
%  !{(-1,-.2)}*+{\text{$t_5$}} 
%   !{(-2.2,2)}*+{\text{$t_6$}} 
%      !{(0,2.6)}*+{\text{$t_7$}} 
%      !{(-.8,1.3)}*+{\text{$t_8$}} 
%        !{(.8,1.3)}*+{\text{$t_9$}} 
%    !{(-1,3)}*{\bullet}="v1"
%     !{(1,3)}*{\bullet}="v2"
% !{(2,2)}*{\bullet}="v3"
%  !{(1,0)}*{\bullet}="v4"
% !{(-1,0)}*{\bullet}="v5"
% !{(-2,2)}*{\bullet}="v6"  
% !{(0,2.4)}*{\bullet}="v7"
% !{(-.6,1.3)}*{\bullet}="v8"
%  !{(.6,1.3)}*{\bullet}="v9"    
%   "v1"-"v2" "v2"-"v3"   "v3"-"v4"   "v4"-"v5"   "v5"-"v6"  "v6"-"v1"  
%"v1"-"v8" "v2"-"v9"   "v3"-"v7"   "v4"-"v8" "v5"-"v9" "v6"-"v7" "v8"-"v7" "v9"-"v7" "v8"-"v9"    
%  }
%\end{displaymath}
%\caption{Graph with second symbolic power Cohen--Macaulay}\label{2nd-cm}
%\end{figure}
\end{example}

\begin{example}\label{example-graph3} 
The edge ideal of the graph $G$ in Figure~\ref{figure3} is given by 
\begin{align*}\label{eideal1}
I=I(G)=(&t_1t_3,t_1t_4,t_1t_7,t_1t_{10},t_1t_{11},t_2t_4,t_2t_5,t_2t_8,t_2t_{10},\\
&t_2t_{11},t_3t_5,t_3t_6,t_3t_8,t_3t_{11},t_4t_6,t_4t_9,t_4t_{11},t_5t_{7},\\
&t_5t_9,t_5t_{11},t_6t_8,t_6t_9,t_7t_9,t_7t_{10},t_8t_{10}).
\end{align*}
\quad The combinatorial properties and algebraic invariants of $G$,
$I(G)$, and $I(G)^{(2)}$ were
computed using Procedures~\ref{nov11-19} and
\ref{Tai-symbolic-powers-algorithm} 
for \textit{Macaulay}$2$
\cite{mac2}. The invariants of $I(G)$ are shown in Table~\ref{tab:1}. In particular
in characteristic $0$ this gives a counterexample to
\cite[Conjecture~4.2]{footprint-ci} because the v-number of $I(G)$ is $3$ and the
regularity of $S/I(G)$ is $2$. The graph $G$ is edge-critical, it is in
$W_2$, and it is Cohen--Macaulay over $\mathbb{Q}$. The symbolic
square $I(G)^{(2)}$ is not Cohen-Macaulay.

\begin{figure}[ht]
\begin{displaymath}
\xygraph{
 !{<0cm,1cm>;<1cm,0cm>;<0cm,1cm>}
%labels
%level1
!{(2,-.1)}*+{\text{$t_2$}}
!{(0,-.1)}*+{\text{$t_8$}}
!{(-2,-.1)}*+{\text{$t_3$}}
%level2
!{(-2.9,1.1)}*+{\text{$t_6$}}
!{(3,1.1)}*+{\text{$t_{10}$}}
%level3
!{(-3,2)}*+{\text{$t_4$}}
!{(3.1,2)}*+{\text{$t_{1}$}}
%level4
!{(-2.3,2.8)}*+{\text{$t_{11}$}}
!{(2.2,2.8)}*+{\text{$t_{7}$}}
%level5
!{(-1,3.5)}*+{\text{$t_{5}$}}
!{(1,3.5)}*+{\text{$t_{9}$}}
%bullets
%level1
!{(2,.2)}*{\bullet}="v2"
!{(0,.2)}*{\bullet}="v8"
!{(-2,.2)}*{\bullet}="v3"
%level2
!{(-2.8,1.3)}*{\bullet}="v6"
!{(2.8,1.3)}*{\bullet}="v10"
%level3
!{(-2.8,2)}*{\bullet}="v4"
!{(2.8,2)}*{\bullet}="v1"
%level4
!{(-2,2.8)}*{\bullet}="v11"
!{(2,2.8)}*{\bullet}="v7"
%level5
!{(-1,3.2)}*{\bullet}="v5"
!{(1,3.2)}*{\bullet}="v9"
"v1"-"v10""v2"-"v10""v2"-"v8""v3"-"v8""v3"-"v6""v4"-"v6""v4"-"v11"
"v5"-"v11""v9"-"v5""v7"-"v9""v1"-"v7"
"v1"-"v3""v1"-"v4""v1"-"v11"
"v2"-"v4""v2"-"v5""v2"-"v11"
"v3"-"v5""v3"-"v11"     
"v4"-"v9"     
"v5"-"v7"   
"v6"-"v9""v6"-"v8"
"v7"-"v10" 
"v8"-"v10"        
 }
\end{displaymath}
\caption{}\label{figure3}
\end{figure}
\begin{table}[h]
\caption{Invariants of $I=I(G)$ in characteristic $0$ and
characteristic $2$.} %title of the table
\centering 
\begin{tabular}{|c|c|}
\hline\hline 
$\mathbb{Q}$  & $\mathbb{Z}_{2}$\\
\hline
v$(I)=3$&v$(I)=3$\\
\hline
reg$(S/I)=2$&reg$(S/I)=3$\\
\hline
dim$(S/I)=3$&dim$(S/I)=3$\\
\hline
\end{tabular}
\label{tab:1}
\end{table}
\end{example}

%\section{List of graphs with $I(G)^{(2)}$ Cohen-Macaulay}
\begin{appendix}
\label{Appendix}
%\appendix
%Procedures for {\it Macaulay\/}$2$

\section{Procedures for {\it Macaulay\/}$2$}%\label{Appendix}

%\smallskip

\begin{procedure}\label{nov11-19}%\cite{min-dis-generalized} 
Computing the regularity, the v-number, and the dimension of the edge ideal $I=I(G)$
of a graph $G$ with
\textit{Macaulay}$2$ \cite{mac2}. This procedure also checks if a graph $G$
is in $W_2$, is Cohen--Macaulay or is edge-critical. This procedure can also be
applied to clutters and their edge ideals. The v-number is computed
using Proposition~\ref{lem:vnumber}. This procedure corresponds to 
Example~\ref{example-graph3}. To compute other examples just replace the
ideal $I$, the variable list, and the ground field $K$.
\begin{verbatim}
R=QQ[t1,t2,t3,t4,t5,t6,t7,t8,t9,t10,t11]--ground field K=QQ
X=toList{t1,t2,t3,t4,t5,t6,t7,t8,t9,t10,t11}
I=monomialIdeal(t1*t3,t1*t4,t1*t7,t1*t10,t1*t11,t2*t4,t2*t5,t2*t8,
t2*t10,t2*t11,t3*t5,t3*t6,t3*t8,t3*t11,t4*t6,t4*t9,t4*t11,t5*t7,
t5*t9,t5*t11,t6*t8,t6*t9,t7*t9,t7*t10,t8*t10)
--The next is True if and only if I is Cohen-Macaulay
codim(I)==pdim coker gens gb I
dim(I), M=coker gens gb I, L=ass I
f=(n)->flatten flatten  degrees mingens(quotient(I,L#n)/I)
g=(a)->toList(set a-set{0}) 
vnumber=min(flatten apply(apply(0..#L-1,f),g))
regularity M
--The next is True if and only if G is in W2
dim(I)==vnumber
G=flatten entries gens gb I
G1=(a)->toList(set G-set{a}) 
--The next two are True if and only if G is edge-critical
min apply(apply(apply(G,G1),ideal),codim)==max apply(apply
(apply(G,G1),ideal),codim)
codim I-min apply(apply(apply(G,G1),ideal),codim)==1
\end{verbatim}
\end{procedure}

\begin{procedure}{(Jonathan O'Rourke)}\label{Tai-symbolic-powers-algorithm} 
This procedure for \textit{Macaulay}$2$ \cite{mac2} 
computes the $k$-th symbolic power $I^{(k)}$ of an edge ideal 
$I$ and determines whether or not $I^{(k)}$ is Cohen--Macaulay. This procedure corresponds to 
Example~\ref{example-graph4}.
%\begin{small}
\begin{verbatim}
R=QQ[t1,t2,t3,t4,t5,t6,t7,t8,t9]
X=toList flatten entries vars R
I=monomialIdeal(t1*t2,t2*t3,t3*t4,t4*t5,t5*t6,t6*t1,t1*t8,t2*t9,
t3*t7,t4*t8,t5*t9,t6*t7,t7*t8,t7*t9,t8*t9) 
--Computes the k-th Symbolic Power of I
SP = (I,k) ->  (temp = primaryDecomposition I; 
temp2 = ((temp_0)^k); for i from 1 to #temp-1 do(temp2 =                               
     intersect(temp2,(temp_i)^k)); return temp2)                  
--The next is True if and only if 
--the second symbolic power is Cohen-Macaulay
codim(I)==pdim coker gens gb SP(I,2)
\end{verbatim}
%\end{small}
\end{procedure}
%\end{appendix}

%\endappendix

%\begin{appendix}
%\label{Appendix-figures}
%\section{Tables of graphs}
%\newpage
%\clearpage
\end{appendix}

\newpage

\begin{table}[H]
%\caption{Edge-critical graphs with $I(G)^{(2)}$ Cohen-Macaulay}
%\caption{Edge-critical graphs in $W_2$ with $I(G)^{(2)}$ Cohen-Macaulay}
\caption{Connected graphs with $I(G)^{(2)}$
Cohen-Macaulay}
\centering 
\begin{tabular}{|c|c|c|c|}
\hline\hline 
$|V(G)|$  & Graph $G$ & $I=I(G)$\\

%######################## 2vertices #################################

\hline
  \xygraph{
!{<0cm,.5cm>;<.5cm,0cm>;<0cm,.5cm>}
!{(0,1.1)}*+{\text{2}} 
}& 
\xygraph{
 !{<0cm,.5cm>;<.5cm,0cm>;<0cm,.5cm>} 
 !{(0,1.1)}*+{\text{$t_2$}} 
 !{(-1,1.1)}*+{\text{$t_1$}} 
 !{(0,.8)}*{\bullet}="v3"
 !{(-1,.8)}*{\bullet}="v1"
 "v1"-"v3"
}& 
\xygraph{
!{<0cm,.5cm>;<.5cm,0cm>;<0cm,.5cm>}
!{(0,1.1)}*+{\text{$(t_{1}t_{2})$}} 
}
\\

%######################## 3vertices #################################

\hline
\xygraph{
!{<0cm,.5cm>;<.5cm,0cm>;<0cm,.5cm>}
!{(0,1.7)}*+{\text{3}} 
!{(0,1.3)}*+{} 
}  & 
\xygraph{
 !{<0cm,.5cm>;<.5cm,0cm>;<0cm,.5cm>}
  !{(0,1.1)}*+{\text{$t_3$}} 
 !{(-1,1.1)}*+{\text{$t_1$}} 
 !{(-.5,1.5)}*+{\text{$t_2$}} 
 !{(0,.8)}*{\bullet}="v3"
 !{(-1,.8)}*{\bullet}="v1"
 !{(-.5,1.3)}*{\bullet}="v2"
 "v1"-"v3" "v1"-"v2" "v2"-"v3"
}& \xygraph{
!{<0cm,.5cm>;<.5cm,0cm>;<0cm,.5cm>}
!{(0,1.5)}*+{\text{$(t_{1}t_{2},t_{1}t_{3},t_{2}t_{3})$}} 
}
\\

%######################## 4vertices #################################

\hline
\xygraph{
!{<0cm,.5cm>;<.5cm,0cm>;<0cm,.5cm>}
!{(0,1.7)}*+{\text{4}} 
}  & 
\xygraph{
 !{<0cm,.5cm>;<.5cm,0cm>;<0cm,.5cm>}
  !{(0,1.8)}*+{\text{$t_4$}}  
  !{(0.2,1.1)}*+{\text{$t_3$}} 
 !{(-1.2,1.1)}*+{\text{$t_1$}} 
 !{(-1,1.8)}*+{\text{$t_2$}} 
 !{(0,1.6)}*{\bullet}="v4"
 !{(0,.8)}*{\bullet}="v3"
 !{(-1,.8)}*{\bullet}="v1"
 !{(-1,1.6)}*{\bullet}="v2"
 "v1"-"v3" "v1"-"v2" "v2"-"v3""v4"-"v3""v4"-"v2""v4"-"v1"
}& \xygraph{
!{<0cm,.5cm>;<.5cm,0cm>;<0cm,.5cm>}
!{(0,1.7)}*+{\text{$(t_{1}t_{2},t_{1}t_{3},t_{1}t_{4},$}} 
!{(0,1.3)}*+{\text{$t_{2}t_{3},t_{2}t_{4},t_{3}t_{4})$}} 
}
\\

%######################## 5vertices #################################

\hline
\xygraph{
!{<0cm,.5cm>;<.5cm,0cm>;<0cm,.5cm>}
!{(0,1.7)}*+{\text{5}} 
}  & 
\xygraph{
 !{<0cm,.5cm>;<.5cm,0cm>;<0cm,.5cm>}
 !{(-0.5,2.3)}*+{\text{$t_1$}} 
  !{(0.4,1.8)}*+{\text{$t_2$}} 
 !{(0.3,1)}*+{\text{$t_3$}}
 !{(-1.3,1)}*+{\text{$t_4$}} 
 !{(-1.4,1.8)}*+{\text{$t_5$}} 
 !{(-0.5,2.1)}*{\bullet}="v1"
 !{(0.2,1.6)}*{\bullet}="v2"
  !{(0,.8)}*{\bullet}="v3"
 !{(-1,.8)}*{\bullet}="v4"
 !{(-1.2,1.6)}*{\bullet}="v5"
 "v1"-"v2" "v1"-"v5" "v2"-"v3""v4"-"v3""v4"-"v5"
}& \xygraph{
!{<0cm,.5cm>;<.5cm,0cm>;<0cm,.5cm>}
!{(0,2.5)}*+{\text{$(t_{1}t_{2},t_{2}t_{3},t_{3}t_{4},$}} 
!{(0,2)}*+{\text{$t_{4}t_{5},t_{5}t_{1})$}} 
!{(0,1.5)}*+{\text{}}
}
\\

\hline
\xygraph{
!{<0cm,.5cm>;<.5cm,0cm>;<0cm,.5cm>}
!{(0,1.7)}*+{\text{5}} 
}  & 
\xygraph{
 !{<0cm,.5cm>;<.5cm,0cm>;<0cm,.5cm>}
 !{(-0.5,2.3)}*+{\text{$t_1$}} 
  !{(0.4,1.8)}*+{\text{$t_2$}} 
 !{(0.3,1)}*+{\text{$t_3$}}
 !{(-1.3,1)}*+{\text{$t_4$}} 
 !{(-1.4,1.8)}*+{\text{$t_5$}} 
 !{(-0.5,2.1)}*{\bullet}="v1"
 !{(0.2,1.6)}*{\bullet}="v2"
  !{(0,.8)}*{\bullet}="v3"
 !{(-1,.8)}*{\bullet}="v4"
 !{(-1.2,1.6)}*{\bullet}="v5"
 "v1"-"v2" "v1"-"v5" "v2"-"v3""v4"-"v3""v4"-"v5""v4"-"v1""v1"-"v3""v4"-"v2""v3"-"v5""v2"-"v5"
}& \xygraph{
!{<0cm,.5cm>;<.5cm,0cm>;<0cm,.5cm>}
!{(0,2.5)}*+{\text{$(t_{1}t_{2},t_{1}t_{3},t_{1}t_{4},$}} 
!{(0,2)}*+{\text{$t_{1}t_{5},t_{2}t_{3},t_{2}t_{4},$}} 
!{(0,1.5)}*+{\text{$t_{2}t_{5},t_{3}t_{4},t_{3}t_{5},$}}
!{(0,1)}*+{\text{$t_{4}t_{5})$}}
}
\\

%######################## 6vertices #################################

\hline
\xygraph{
!{<0cm,.5cm>;<.5cm,0cm>;<0cm,.5cm>}
!{(0,1.7)}*+{\text{6}} 
}  & 
\xygraph{
 !{<0cm,.5cm>;<.5cm,0cm>;<0cm,.5cm>}
 !{(-0.5,2.3)}*+{\text{$t_1$}} 
  !{(0.4,1.8)}*+{\text{$t_2$}} 
 !{(0.3,1)}*+{\text{$t_3$}}
 !{(-1.3,1)}*+{\text{$t_4$}} 
 !{(-1.4,1.8)}*+{\text{$t_5$}} 
 !{(-0.5,1.2)}*+{\text{$t_6$}} 
 !{(-0.5,2.1)}*{\bullet}="v1"
 !{(0.2,1.6)}*{\bullet}="v2"
  !{(0,.8)}*{\bullet}="v3"
 !{(-1,.8)}*{\bullet}="v4"
 !{(-1.2,1.6)}*{\bullet}="v5"
  !{(-0.5,1.4)}*{\bullet}="v6"
 "v1"-"v2" "v1"-"v5" "v2"-"v3""v4"-"v3""v4"-"v5""v1"-"v6""v2"-"v6""v6"-"v5"
}& \xygraph{
!{<0cm,.5cm>;<.5cm,0cm>;<0cm,.5cm>}
!{(0,2.5)}*+{\text{$(t_{1}t_{2},t_{1}t_{5},t_{1}t_{6},$}} 
!{(0,2)}*+{\text{$t_{2}t_{3},t_{2}t_{6},$}} 
!{(0,1.5)}*+{\text{$t_{3}t_{4},t_{4}t_{5},t_{5}t_{6})$}}
}
\\

\hline
\xygraph{
!{<0cm,.5cm>;<.5cm,0cm>;<0cm,.5cm>}
!{(0,1.7)}*+{\text{6}} 
}  & 
\xygraph{
 !{<0cm,.5cm>;<.5cm,0cm>;<0cm,.5cm>}
 !{(-1,2.3)}*+{\text{$t_1$}} 
  !{(0.4,1.6)}*+{\text{$t_3$}} 
 !{(0.3,1)}*+{\text{$t_4$}}
 !{(-1.3,1)}*+{\text{$t_5$}} 
 !{(-1.4,1.6)}*+{\text{$t_6$}} 
 !{(0,2.3)}*+{\text{$t_2$}} 
 !{(-1,2.1)}*{\bullet}="v1"
 !{(0.2,1.4)}*{\bullet}="v3"
  !{(0,.8)}*{\bullet}="v4"
 !{(-1,.8)}*{\bullet}="v5"
 !{(-1.2,1.4)}*{\bullet}="v6"
  !{(0,2.1)}*{\bullet}="v2"
  "v1"-"v2""v3"-"v2""v4"-"v3""v5"-"v4""v6"-"v5""v1"-"v6"
"v1"-"v3""v1"-"v4""v1"-"v5""v4"-"v2""v5"-"v2""v6"-"v2"
"v5"-"v3""v6"-"v3""v6"-"v4"
}& \xygraph{
!{<0cm,.5cm>;<.5cm,0cm>;<0cm,.5cm>}
!{(0,2.5)}*+{\text{$(t_{1}t_{2},t_{1}t_{3},t_{1}t_{4},t_{1}t_{5},$}} 
!{(0,2)}*+{\text{$t_{1}t_{6},t_{2}t_{3},t_{2}t_{4},t_{2}t_{5},$}} 
!{(0,1.5)}*+{\text{$t_{2}t_{6},t_{3}t_{4},t_{3}t_{5},t_{3}t_{6},$}}
!{(0,1)}*+{\text{$t_{4}t_{5},t_{4}t_{6},t_{5}t_{6})$}}
}
\\

%######################## 7vertices #################################

%______________________________________________1____________________________________

\hline
\xygraph{
!{<0cm,.5cm>;<.5cm,0cm>;<0cm,.5cm>}
!{(0,2.4)}*+{\text{7}} 
!{(0,1.3)}*+{} 
}  & 
\xygraph{
 !{<0cm,.5cm>;<.5cm,0cm>;<0cm,.5cm>} 
 !{(-0.5,2.7)}*+{\text{$t_1$}} 
 !{(.15,2.2)}*+{\text{$t_2$}} 
 !{(-1.15,2.2)}*+{\text{$t_7$}} 
 !{(.5,1.3)}*+{\text{$t_3$}} 
 !{(-1.5,1.3)}*+{\text{$t_6$}} 
 !{(0,.6)}*+{\text{$t_4$}} 
 !{(-1,.6)}*+{\text{$t_5$}} 
 !{(-.5,2.5)}*{\bullet}="v1"
  !{(.15,2)}*{\bullet}="v2"
  !{(0.3,1.3)}*{\bullet}="v3"
  !{(0,.8)}*{\bullet}="v4"
 !{(-1,.8)}*{\bullet}="v5"
  !{(-1.3,1.3)}*{\bullet}="v6"
 !{(-1.15,2)}*{\bullet}="v7"
 "v1"-"v2" "v1"-"v3" "v2"-"v3""v2"-"v4""v4"-"v3""v5"-"v4"
"v6"-"v1""v7"-"v1""v6"-"v7""v6"-"v5""v5"-"v7"
}& \xygraph{
!{<0cm,.5cm>;<.5cm,0cm>;<0cm,.5cm>}
!{(0,2.5)}*+{\text{$(t_{1}t_{2},t_{1}t_{3},t_{1}t_{6},$}} 
!{(0,2)}*+{\text{$t_{1}t_{7},t_{2}t_{3},t_{2}t_{4},$}} 
!{(0,1.5)}*+{\text{$t_{3}t_{4},t_{4}t_{5},t_{5}t_{6},$}}
!{(0,1)}*+{\text{$t_{5}t_{7},t_{6}t_{7})$}}
}
\\

%______________________________________________2____________________________________

\hline
\xygraph{
!{<0cm,.5cm>;<.5cm,0cm>;<0cm,.5cm>}
!{(0,2.4)}*+{\text{7}} 
!{(0,1.3)}*+{} 
}  & 
\xygraph{
 !{<0cm,.5cm>;<.5cm,0cm>;<0cm,.5cm>} 
 !{(-0.5,2.7)}*+{\text{$t_1$}} 
 !{(.15,2.2)}*+{\text{$t_2$}} 
 !{(-1.15,2.2)}*+{\text{$t_7$}} 
 !{(.5,1.3)}*+{\text{$t_3$}} 
 !{(-1.5,1.3)}*+{\text{$t_6$}} 
 !{(0,.6)}*+{\text{$t_4$}} 
 !{(-1,.6)}*+{\text{$t_5$}} 
 !{(-.5,2.5)}*{\bullet}="v1"
  !{(.15,2)}*{\bullet}="v2"
  !{(0.3,1.3)}*{\bullet}="v3"
  !{(0,.8)}*{\bullet}="v4"
 !{(-1,.8)}*{\bullet}="v5"
  !{(-1.3,1.3)}*{\bullet}="v6"
 !{(-1.15,2)}*{\bullet}="v7"
 "v1"-"v2"  "v2"-"v3""v4"-"v3""v5"-"v4""v7"-"v1""v6"-"v7"
"v6"-"v5""v7"-"v5""v4"-"v2""v6"-"v3""v3"-"v5""v6"-"v4"
 }& \xygraph{
!{<0cm,.5cm>;<.5cm,0cm>;<0cm,.5cm>}
!{(0,2.5)}*+{\text{$(t_{1}t_{2},t_{5}t_{3},t_{4}t_{6},$}} 
!{(0,2)}*+{\text{$t_{1}t_{7},t_{2}t_{3},t_{2}t_{4},$}} 
!{(0,1.5)}*+{\text{$t_{3}t_{4},t_{4}t_{5},t_{5}t_{6},$}}
!{(0,1)}*+{\text{$t_{5}t_{7},t_{6}t_{7},t_{3}t_{6})$}}
}
\\

\hline

\end{tabular}
\label{tab:C-M,G}
\end{table}

\newpage

\begin{table}[H]
%\caption{Edge-critical graphs with $I(G)^{(2)}$ Cohen-Macaulay}
%\caption{}
\centering 
\begin{tabular}{|c|c|c|c|}
\hline\hline 
$|V(G)|$  & Graph $G$ & $I=I(G)$\\

%______________________________________________3____________________________________

\hline
\xygraph{
!{<0cm,.5cm>;<.5cm,0cm>;<0cm,.5cm>}
!{(0,2.4)}*+{\text{7}} 
!{(0,1.3)}*+{} 
}  & 
\xygraph{
 !{<0cm,.5cm>;<.5cm,0cm>;<0cm,.5cm>} 
 !{(-0.5,2.7)}*+{\text{$t_1$}} 
 !{(.15,2.2)}*+{\text{$t_2$}} 
 !{(-1.15,2.2)}*+{\text{$t_7$}} 
 !{(.5,1.3)}*+{\text{$t_3$}} 
 !{(-1.5,1.3)}*+{\text{$t_6$}} 
 !{(0,.6)}*+{\text{$t_4$}} 
 !{(-1,.6)}*+{\text{$t_5$}} 
 !{(-.5,2.5)}*{\bullet}="v1"
  !{(.15,2)}*{\bullet}="v2"
  !{(0.3,1.3)}*{\bullet}="v3"
  !{(0,.8)}*{\bullet}="v4"
 !{(-1,.8)}*{\bullet}="v5"
  !{(-1.3,1.3)}*{\bullet}="v6"
 !{(-1.15,2)}*{\bullet}="v7"
 "v1"-"v2"  "v2"-"v3""v4"-"v3""v5"-"v4""v7"-"v1""v6"-"v7""v6"-"v5"
"v3"-"v1""v6"-"v1""v6"-"v2""v7"-"v3""v7"-"v2"
}& \xygraph{
!{<0cm,.5cm>;<.5cm,0cm>;<0cm,.5cm>}
!{(0,2.5)}*+{\text{$(t_{1}t_{2},t_{1}t_{3},t_{1}t_{6},$}} 
!{(0,2)}*+{\text{$t_{1}t_{7},t_{2}t_{3},t_{2}t_{6},$}} 
!{(0,1.5)}*+{\text{$t_{3}t_{4},t_{4}t_{5},t_{5}t_{6},$}}
!{(0,1)}*+{\text{$t_{3}t_{7},t_{6}t_{7},t_{2}t_{7})$}}
}
\\

%______________________________________________4____________________________________

\hline
\xygraph{
!{<0cm,.5cm>;<.5cm,0cm>;<0cm,.5cm>}
!{(0,2.4)}*+{\text{7}} 
!{(0,1.3)}*+{} 
}  & 
\xygraph{
 !{<0cm,.5cm>;<.5cm,0cm>;<0cm,.5cm>} 
 !{(-0.5,2.7)}*+{\text{$t_1$}} 
 !{(.15,2.2)}*+{\text{$t_2$}} 
 !{(-1.15,2.2)}*+{\text{$t_7$}} 
 !{(.5,1.3)}*+{\text{$t_3$}} 
 !{(-1.5,1.3)}*+{\text{$t_6$}} 
 !{(0,.6)}*+{\text{$t_4$}} 
 !{(-1,.6)}*+{\text{$t_5$}} 
 !{(-.5,2.5)}*{\bullet}="v1"
  !{(.15,2)}*{\bullet}="v2"
  !{(0.3,1.3)}*{\bullet}="v3"
  !{(0,.8)}*{\bullet}="v4"
 !{(-1,.8)}*{\bullet}="v5"
  !{(-1.3,1.3)}*{\bullet}="v6"
 !{(-1.15,2)}*{\bullet}="v7"
 "v1"-"v2"  "v2"-"v3""v4"-"v3""v5"-"v4""v7"-"v1""v6"-"v7""v6"-"v5"
"v3"-"v1""v6"-"v1""v6"-"v2""v7"-"v3""v7"-"v2"
 "v1"-"v4""v1"-"v5""v4"-"v2""v5"-"v2""v5"-"v3""v6"-"v3""v6"-"v4""v7"-"v4""v7"-"v5"         
}& \xygraph{
!{<0cm,.5cm>;<.5cm,0cm>;<0cm,.5cm>}
!{(0,2.5)}*+{\text{$(t_1t_2,t_1t_3,t_1t_4,t_1t_5,t_1t_6,$}} 
!{(0,2)}*+{\text{$t_2t_3,t_2t_4,t_2t_5,t_2t_6,t_3t_4,$}} 
!{(0,1.5)}*+{\text{$t_3t_5,t_3t_6,t_4t_5,t_4t_6,t_5t_6,$}}
!{(0,1)}*+{\text{$t_7t_1,t_7t_2,t_7t_3,t_7t_4,t_7t_5,t_7t_6)$}}
}
\\
\hline

%######################## 8vertices #################################

%______________________________________________1____________________________________

\hline
\xygraph{
!{<0cm,.5cm>;<.5cm,0cm>;<0cm,.5cm>}
!{(0,2.4)}*+{\text{8}} 
!{(0,1.3)}*+{} 
}  & 
\xygraph{
 !{<0cm,.5cm>;<.5cm,0cm>;<0cm,.5cm>} 
 !{(-0.5,2.5)}*+{\text{$t_1$}} 
 !{(.6,2)}*+{\text{$t_2$}} 
 !{(.25,1.3)}*+{\text{$t_3$}} 
 !{(0.6,.6)}*+{\text{$t_4$}} 
 !{(-1.6,.6)}*+{\text{$t_5$}} 
 !{(-1.25,1.3)}*+{\text{$t_6$}} 
  !{(-1.6,2)}*+{\text{$t_7$}} 
   !{(-0.5,.1)}*+{\text{$t_8$}} 
 !{(-.5,2.3)}*{\bullet}="v1"
  !{(0.5,1.8)}*{\bullet}="v2"
  !{(0,1.3)}*{\bullet}="v3"
  !{(0.5,.8)}*{\bullet}="v4"
 !{(-1.5,.8)}*{\bullet}="v5"
  !{(-1,1.3)}*{\bullet}="v6"
 !{(-1.5,1.8)}*{\bullet}="v7"
  !{(-.5,.3)}*{\bullet}="v8"
 "v1"-"v2" "v1"-"v3""v1"-"v6""v1"-"v7""v1"-"v2" "v4"-"v2""v1"-"v2"
"v8"-"v4""v5"-"v8"
"v7"-"v5""v1"-"v7""v6"-"v8""v8"-"v3""v7"-"v2""v4"-"v5"
"v7"-"v3""v5"-"v3""v6"-"v2""v4"-"v6"                                         
}& \xygraph{
!{<0cm,.5cm>;<.5cm,0cm>;<0cm,.5cm>}
!{(0,2.5)}*+{\text{$(t_1t_3,t_1t_2,t_1t_6,t_1t_7,t_2t_7,t_2t_6,$}} 
!{(0,2)}*+{\text{$t_2t_4,t_3t_7,t_3t_5,t_3t_8,t_4t_5,t_4t_6,$}} 
!{(0,1.5)}*+{\text{$t_4t_8,t_5t_7,t_5t_8,t_6t_8)$}}
!{(0,1)}*+{\text{}}
}
\\
\hline

%______________________________________________2____________________________________

\hline
\xygraph{
!{<0cm,.5cm>;<.5cm,0cm>;<0cm,.5cm>}
!{(0,2.4)}*+{\text{8}} 
!{(0,1.3)}*+{} 
}  & 
\xygraph{
 !{<0cm,.5cm>;<.5cm,0cm>;<0cm,.5cm>} 
 !{(-0.3,2)}*+{\text{$t_1$}} 
 !{(.8,2)}*+{\text{$t_2$}} 
 !{(1.2,1.3)}*+{\text{$t_3$}} 
 !{(0.8,.4)}*+{\text{$t_4$}} 
 !{(-1.8,.4)}*+{\text{$t_5$}} 
 !{(-2.2,1.3)}*+{\text{$t_6$}} 
  !{(-1.8,2)}*+{\text{$t_7$}} 
   !{(-0.5,.8)}*+{\text{$t_8$}} 
 !{(-.3,1.8)}*{\bullet}="v1"
  !{(0.5,1.8)}*{\bullet}="v2"
  !{(1,1)}*{\bullet}="v3"
  !{(0.5,.4)}*{\bullet}="v4"
 !{(-1.5,.4)}*{\bullet}="v5"
  !{(-2,1)}*{\bullet}="v6"
 !{(-1.7,1.8)}*{\bullet}="v7"
  !{(-.5,1)}*{\bullet}="v8"
   "v1"-"v2" "v3"-"v2" "v3"-"v4" "v8"-"v4" "v5"-"v8" "v6"-"v5" "v7"-"v6" "v1"-"v7" "v1"-"v2" "v1"-"v4" "v1"-"v3" "v8"-"v2" "v8"-"v3" "v4"-"v2" "v5"-"v7" "v6"-"v8"
}& \xygraph{
!{<0cm,.5cm>;<.5cm,0cm>;<0cm,.5cm>}
!{(0,2.5)}*+{\text{$(t_1t_2,t_1t_3,t_1t_4,t_1t_7,t_2t_3,t_2t_4,$}} 
!{(0,2)}*+{\text{$t_2t_8,t_3t_4,t_3t_8,t_4t_8,t_5t_6,t_5t_7,$}} 
!{(0,1.5)}*+{\text{$t_5t_8,t_6t_7,t_6t_8)$}}
!{(0,1)}*+{\text{}}
}
\\
\hline

%______________________________________________3____________________________________

\hline
\xygraph{
!{<0cm,.5cm>;<.5cm,0cm>;<0cm,.5cm>}
!{(0,2.4)}*+{\text{8}} 
!{(0,1.3)}*+{} 
}  & 
\xygraph{
 !{<0cm,.5cm>;<.5cm,0cm>;<0cm,.5cm>} 
 !{(-0.9,.8)}*+{\text{$t_7$}} 
 !{(.8,2)}*+{\text{$t_2$}} 
 !{(1.2,1.3)}*+{\text{$t_3$}} 
 !{(0.8,.4)}*+{\text{$t_4$}} 
 !{(-1.8,.4)}*+{\text{$t_5$}} 
 !{(-2.2,1.3)}*+{\text{$t_6$}} 
  !{(-1.8,2)}*+{\text{$t_1$}} 
   !{(-0.5,.2)}*+{\text{$t_8$}} 
 !{(-1.1,1)}*{\bullet}="v7"
  !{(0.5,1.8)}*{\bullet}="v2"
  !{(1,1)}*{\bullet}="v3"
  !{(0.5,.4)}*{\bullet}="v4"
 !{(-1.5,.4)}*{\bullet}="v5"
  !{(-2,1)}*{\bullet}="v6"
 !{(-1.7,1.8)}*{\bullet}="v1"
  !{(-.5,.4)}*{\bullet}="v8"
  "v1"-"v2""v3"-"v2""v4"-"v3""v8"-"v4""v8"-"v5""v6"-"v5""v1"-"v6""v7"-"v5""v1"-"v7""v1"-"v5""v6"-"v7"
  "v8"-"v6""v8"-"v3""v4"-"v2""v7"-"v2"              
}& \xygraph{
!{<0cm,.5cm>;<.5cm,0cm>;<0cm,.5cm>}
!{(0,2.5)}*+{\text{$(t_1t_2,t_1t_5,t_1t_6,t_1t_7,t_2t_3,t_2t_4,$}} 
!{(0,2)}*+{\text{$t_2t_7,t_3t_4,t_3t_8,t_4t_8,t_5t_6,t_5t_7,$}} 
!{(0,1.5)}*+{\text{$t_5t_8,t_6t_7,t_6t_8)$}}
!{(0,1)}*+{\text{}}
}
\\
\hline

%______________________________________________4____________________________________

\hline
\xygraph{
!{<0cm,.5cm>;<.5cm,0cm>;<0cm,.5cm>}
!{(0,2.4)}*+{\text{8}} 
!{(0,1.3)}*+{} 
}  & 
\xygraph{
 !{<0cm,.5cm>;<.5cm,0cm>;<0cm,.5cm>} 
 !{(-0.5,1.2)}*+{\text{$t_7$}} 
 !{(.8,2)}*+{\text{$t_2$}} 
 !{(1.2,1.3)}*+{\text{$t_3$}} 
 !{(0.8,.4)}*+{\text{$t_4$}} 
 !{(-1.8,.4)}*+{\text{$t_5$}} 
 !{(-2.2,1.3)}*+{\text{$t_6$}} 
  !{(-1.8,2)}*+{\text{$t_1$}} 
   !{(-0.5,.2)}*+{\text{$t_8$}} 
 !{(-.5,1)}*{\bullet}="v7"
  !{(0.5,1.8)}*{\bullet}="v2"
  !{(1,1)}*{\bullet}="v3"
  !{(0.5,.4)}*{\bullet}="v4"
 !{(-1.5,.4)}*{\bullet}="v5"
  !{(-2,1)}*{\bullet}="v6"
 !{(-1.7,1.8)}*{\bullet}="v1"
  !{(-.5,.4)}*{\bullet}="v8"   
  "v1"-"v2""v3"-"v2""v3"-"v4""v8"-"v4""v8"-"v5""v5"-"v6""v1"-"v6""v1"-"v5""v4"-"v2""v6"-"v7""v3"-"v7"
  "v5"-"v7""v4"-"v7""v8"-"v3""v8"-"v6""v8"-"v7"         
}& \xygraph{
!{<0cm,.5cm>;<.5cm,0cm>;<0cm,.5cm>}
!{(0,2.5)}*+{\text{$(t_1t_2,t_1t_5,t_1t_6,t_2t_3,t_2t_4,t_3t_4,$}} 
!{(0,2)}*+{\text{$t_3t_7,t_3t_8,t_4t_7,t_4t_8,t_5t_6,t_5t_7,$}} 
!{(0,1.5)}*+{\text{$t_5t_8,t_6t_7,t_6t_8,t_7t_8)$}}
!{(0,1)}*+{\text{}}
}
\\
\hline

%______________________________________________5____________________________________

\hline
\xygraph{
!{<0cm,.5cm>;<.5cm,0cm>;<0cm,.5cm>}
!{(0,2.4)}*+{\text{8}} 
!{(0,1.3)}*+{} 
}  & 
\xygraph{
 !{<0cm,.5cm>;<.5cm,0cm>;<0cm,.5cm>} 
 !{(-0.7,1.4)}*+{\text{$t_7$}} 
 !{(.8,2)}*+{\text{$t_2$}} 
 !{(1.2,1)}*+{\text{$t_3$}} 
 !{(0.8,0)}*+{\text{$t_4$}} 
 !{(-1.8,0)}*+{\text{$t_5$}} 
 !{(-2.2,1)}*+{\text{$t_6$}} 
  !{(-1.8,2)}*+{\text{$t_1$}} 
   !{(-0.7,.7)}*+{\text{$t_8$}} 
 !{(-.5,1.5)}*{\bullet}="v7"
  !{(0.5,1.8)}*{\bullet}="v2"
  !{(1,.9)}*{\bullet}="v3"
  !{(0.5,0)}*{\bullet}="v4"
 !{(-1.5,0)}*{\bullet}="v5"
  !{(-2,.9)}*{\bullet}="v6"
 !{(-1.7,1.8)}*{\bullet}="v1"
  !{(-.5,.6)}*{\bullet}="v8"   
  "v1"-"v2"  "v3"-"v2""v3"-"v4""v5"-"v6""v1"-"v6" "v4"-"v5"  "v1"-"v7""v2"-"v7""v3"-"v7" "v4"-"v7""v8"-"v7""v2"-"v4" "v2"-"v8""v3"-"v5""v3"-"v8" "v8"-"v4"   "v8"-"v5"                  
}& \xygraph{
!{<0cm,.5cm>;<.5cm,0cm>;<0cm,.5cm>}
!{(0,2.5)}*+{\text{$(t_1t_2,t_1t_6,t_1t_7,t_2t_3,t_2t_4,t_2t_7,$}} 
!{(0,2)}*+{\text{$t_2t_8,t_3t_4,t_3t_5,t_3t_7,t_3t_8,t_4t_5,$}} 
!{(0,1.5)}*+{\text{$t_4t_7,t_4t_8,t_5t_6,t_5t_8,t_7t_8)$}}
!{(0,1)}*+{\text{}}
}
\\
\hline

\end{tabular}
\label{tab:C-M,G1}
\end{table}

\newpage

\begin{table}[H]
%\caption{Edge-critical graphs with $I(G)^{(2)}$ Cohen-Macaulay}
%\caption{}
\centering 
\begin{tabular}{|c|c|c|c|}
\hline\hline 
$|V(G)|$ & Graph $G$ & $I=I(G)$\\

%______________________________________________6____________________________________

\hline
\xygraph{
!{<0cm,.5cm>;<.5cm,0cm>;<0cm,.5cm>}
!{(0,2.4)}*+{\text{8}} 
!{(0,1.3)}*+{} 
}  & 
\xygraph{
 !{<0cm,.5cm>;<.5cm,0cm>;<0cm,.5cm>}
 !{(-1.2,2.3)}*+{\text{$t_1$}} 
  !{(0.4,1.6)}*+{\text{$t_3$}} 
 !{(0.3,1)}*+{\text{$t_4$}}
 !{(-1.3,1)}*+{\text{$t_5$}} 
 !{(-1.4,1.6)}*+{\text{$t_6$}} 
 !{(0.2,2.3)}*+{\text{$t_2$}} 
 !{(-1.2,3)}*+{\text{$t_8$}} 
 !{(0.2,3)}*+{\text{$t_7$}} 
 !{(-1,2.1)}*{\bullet}="v1"
 !{(0.2,1.4)}*{\bullet}="v3"
  !{(0,.8)}*{\bullet}="v4"
 !{(-1,.8)}*{\bullet}="v5"
 !{(-1.2,1.4)}*{\bullet}="v6"
  !{(0,2.1)}*{\bullet}="v2"
  !{(0,3)}*{\bullet}="v7"
  !{(-1,3)}*{\bullet}="v8"
  "v3"-"v2""v4"-"v3""v5"-"v4""v6"-"v5""v1"-"v6""v1"-"v3""v1"-"v4""v1"-"v5""v4"-"v2""v5"-"v2""v6"-"v2"
"v5"-"v3""v6"-"v3""v6"-"v4""v1"-"v8""v8"-"v7""v2"-"v7"
}& \xygraph{
!{<0cm,.5cm>;<.5cm,0cm>;<0cm,.5cm>}
!{(0,2.5)}*+{\text{$(t_{1}t_{8},t_{1}t_{3},t_{1}t_{4},t_{8}t_{7},$}} 
!{(0,2)}*+{\text{$t_{1}t_{5},t_{1}t_{6},t_{2}t_{3},t_{7}t_{2},$}} 
!{(0,1.5)}*+{\text{$t_{2}t_{4},t_{2}t_{5},t_{2}t_{6},$}}
!{(0,1)}*+{\text{$t_{3}t_{4},t_{3}t_{5},t_{3}t_{6},$}}
!{(0,0.5)}*+{\text{$t_{4}t_{5},t_{4}t_{6},t_{5}t_{6})$}}
}
\\
\hline

%______________________________________________7____________________________________

\hline
\xygraph{
!{<0cm,.5cm>;<.5cm,0cm>;<0cm,.5cm>}
!{(0,2.4)}*+{\text{8}} 
!{(0,1.3)}*+{} 
}  & 
\xygraph{
 !{<0cm,.5cm>;<.5cm,0cm>;<0cm,.5cm>} 
 !{(-1.1,.1)}*+{\text{$t_7$}} 
 !{(-1.1,2)}*+{\text{$t_2$}} 
 !{(0,2)}*+{\text{$t_3$}} 
 !{(1,2)}*+{\text{$t_4$}} 
 !{(1,0.1)}*+{\text{$t_5$}} 
 !{(0,.1)}*+{\text{$t_6$}} 
  !{(-2.1,2)}*+{\text{$t_1$}} 
   !{(-2,.1)}*+{\text{$t_8$}} 
 !{(-1.1,.3)}*{\bullet}="v7"
  !{(-1.1,1.8)}*{\bullet}="v2"
  !{(0,1.8)}*{\bullet}="v3"
  !{(1,1.8)}*{\bullet}="v4"
 !{(1,.3)}*{\bullet}="v5"
  !{(0,.3)}*{\bullet}="v6"
 !{(-2,1.8)}*{\bullet}="v1"
  !{(-2,.3)}*{\bullet}="v8"  
  "v1"-"v2""v3"-"v2""v3"-"v4""v5"-"v4""v6"-"v5""v7"-"v6""v8"-"v7""v1"-"v8""v6"-"v2""v3"-"v7"                  
}& \xygraph{
!{<0cm,.5cm>;<.5cm,0cm>;<0cm,.5cm>}
!{(0,2.5)}*+{\text{$(t_1t_2,t_1t_8,t_2t_3,t_2t_6,t_3t_4,t_3t_7,$}} 
!{(0,2)}*+{\text{$t_4t_5,t_5t_6,t_6t_7,t_7t_8)$}} 
!{(0,1.5)}*+{\text{}}
!{(0,1)}*+{\text{}}
}
\\
\hline

%______________________________________________8____________________________________

\hline
\xygraph{
!{<0cm,.5cm>;<.5cm,0cm>;<0cm,.5cm>}
!{(0,2.4)}*+{\text{8}} 
!{(0,1.3)}*+{} 
}  & 
\xygraph{
 !{<0cm,.5cm>;<.5cm,0cm>;<0cm,.5cm>} 
 !{(-1,2.7)}*+{\text{$t_1$}} 
 !{(1,2)}*+{\text{$t_2$}} 
 !{(-2,2)}*+{\text{$t_7$}} 
 !{(1.2,1.3)}*+{\text{$t_3$}} 
 !{(-2.2,1.3)}*+{\text{$t_6$}} 
 !{(0,.6)}*+{\text{$t_4$}} 
 !{(-1,.6)}*+{\text{$t_5$}} 
  !{(0,2.7)}*+{\text{$t_8$}} 
 !{(-1,2.5)}*{\bullet}="v1"
  !{(.8,2)}*{\bullet}="v2"
  !{(1,1.3)}*{\bullet}="v3"
  !{(0,.8)}*{\bullet}="v4"
 !{(-1,.8)}*{\bullet}="v5"
  !{(-2,1.3)}*{\bullet}="v6"
 !{(-1.8,2)}*{\bullet}="v7"
 !{(0,2.5)}*{\bullet}="v8"
 "v1"-"v2"  "v2"-"v3""v4"-"v3""v5"-"v4""v7"-"v1""v6"-"v7""v6"-"v5""v3"-"v1""v6"-"v1""v6"-"v2""v7"-"v3""v7"-"v2"
 "v1"-"v4""v1"-"v5""v4"-"v2""v5"-"v2""v5"-"v3""v6"-"v3""v6"-"v4""v7"-"v4""v7"-"v5""v8"-"v5""v8"-"v1"
 "v8"-"v2""v8"-"v3""v8"-"v4""v8"-"v6""v8"-"v7"                                 
}& \xygraph{
!{<0cm,.5cm>;<.5cm,0cm>;<0cm,.5cm>}
!{(0,2.5)}*+{\text{$(t_1t_2,t_1t_3,t_1t_4,t_1t_5,t_1t_6,$}} 
!{(0,2)}*+{\text{$t_2t_3,t_2t_4,t_2t_5,t_2t_6,t_3t_4,$}} 
!{(0,1.5)}*+{\text{$t_3t_5,t_3t_6,t_4t_5,t_4t_6,t_5t_6,$}}
!{(0,1)}*+{\text{$t_7t_1,t_7t_2,t_7t_3,t_7t_4,t_7t_5,t_7t_6,$}}
!{(0,.5)}*+{\text{$t_8t_1,t_8t_2,t_8t_3,t_8t_4,t_8t_5,t_8t_6,t_7t_8)$}}
}
\\
\hline

%#################9 vertices###################

%______________________________________________1____________________________________

\hline 
\xygraph{
!{<0cm,.5cm>;<.5cm,0cm>;<0cm,.5cm>}
!{(0,2.4)}*+{\text{9}} 
!{(0,1.3)}*+{} 
}  &
\xygraph{
 !{<0cm,.5cm>;<.5cm,0cm>;<0cm,.5cm>} 
  !{(-2.1,2.5)}*+{\text{$t_1$}} 
 !{(-.8,2.5)}*+{\text{$t_2$}} 
 !{(-1.3,1.2)}*+{\text{$t_8$}} 
 !{(.4,1.6)}*+{\text{$t_3$}} 
 !{(0.4,.1)}*+{\text{$t_4$}} 
  !{(-.5,-.7)}*+{\text{$t_5$}} 
    !{(-.5,1.2)}*+{\text{$t_9$}} 
      !{(-2.6,-.3)}*+{\text{$t_6$}} 
      !{(-2.7,1)}*+{\text{$t_7$}} 
    !{(-2,2.3)}*{\bullet}="v1"
     !{(-.8,2.3)}*{\bullet}="v2"
  !{(-1.5,1)}*{\bullet}="v8"
 !{(.4,1.4)}*{\bullet}="v3"
  !{(.4,.3)}*{\bullet}="v4"
 !{(-.5,-.5)}*{\bullet}="v5"
 !{(-.5,1)}*{\bullet}="v9"  
 !{(-2.4,0)}*{\bullet}="v6"  
 !{(-2.5,1)}*{\bullet}="v7"
  "v1"-"v2"  "v2"-"v3"   "v3"-"v4"   "v4"-"v5"   "v5"-"v6"   "v6"-"v7"   "v7"-"v1"   "v1"-"v3"   "v2"-"v7"    "v5"-"v3"   "v9"-"v3"   "v9"-"v4"   "v8"-"v4"   "v4"-"v6"   "v5"-"v8"   "v5"-"v9"   "v6"-"v8"   "v6"-"v9"   "v7"-"v8"   "v8"-"v9"   
}& \xygraph{
!{<0cm,.5cm>;<.5cm,0cm>;<0cm,.5cm>}
!{(0,2.5)}*+{\text{$(t_1t_2,t_1t_3,t_1t_7,t_2t_3,t_2t_7,t_3t_5,$}} 
!{(0,2)}*+{\text{$t_4t_3,t_3t_9,t_5t_4,t_4t_6,t_4t_8,t_4t_9,t_5t_8,$}} 
!{(0,1.5)}*+{\text{$t_5t_9,t_5t_6,t_6t_7,t_6t_8,t_6t_9,t_7t_8,t_8t_9)$}}
!{(0,1)}*+{\text{}}
}
\\
\hline

%______________________________________________2____________________________________

\hline 
\xygraph{
!{<0cm,.5cm>;<.5cm,0cm>;<0cm,.5cm>}
!{(0,2.4)}*+{\text{9}} 
!{(0,1.3)}*+{} 
}  &
\xygraph{
 !{<0cm,.5cm>;<.5cm,0cm>;<0cm,.5cm>} 
  !{(-2.1,2.5)}*+{\text{$t_1$}} 
 !{(-.8,2.5)}*+{\text{$t_2$}} 
 !{(-.2,2)}*+{\text{$t_3$}} 
 !{(.4,1.6)}*+{\text{$t_4$}} 
 !{(0.4,.1)}*+{\text{$t_5$}} 
  !{(-.5,-.7)}*+{\text{$t_6$}} 
    !{(-1.2,-.7)}*+{\text{$t_7$}} 
      !{(-2.6,-.3)}*+{\text{$t_8$}} 
      !{(-2.7,1)}*+{\text{$t_9$}} 
    !{(-2,2.3)}*{\bullet}="v1"
     !{(-.8,2.3)}*{\bullet}="v2"
  !{(-.2,1.8)}*{\bullet}="v3"
 !{(.4,1.4)}*{\bullet}="v4"
  !{(.4,.3)}*{\bullet}="v5"
 !{(-.5,-.5)}*{\bullet}="v6"
 !{(-1.5,-.5)}*{\bullet}="v7"  
 !{(-2.4,0)}*{\bullet}="v8"  
 !{(-2.5,1)}*{\bullet}="v9"
  "v1"-"v2"  "v2"-"v3"   "v3"-"v4"   "v4"-"v5"   "v5"-"v6"   "v6"-"v7"   "v7"-"v8"   "v8"-"v9"   "v1"-"v9"    "v1"-"v3"   "v1"-"v8"   "v9"-"v2"   "v8"-"v2"   "v3"-"v9"   "v4"-"v6"   "v7"-"v4"   "v1"-"v2"   "v8"-"v5"   "v7"-"v5"   "v1"-"v2"   "v8"-"v6"     "v1"-"v2"   "v9"-"v8"   
}& \xygraph{
!{<0cm,.5cm>;<.5cm,0cm>;<0cm,.5cm>}
!{(0,2.5)}*+{\text{$(t_1t_2,t_1t_3,t_1t_8,t_1t_9,t_2t_9,t_2t_8,$}} 
!{(0,2)}*+{\text{$t_2t_3,t_3t_4,t_3t_9,t_4t_6,t_4t_5,t_7t_4,t_6t_8,$}} 
!{(0,1.5)}*+{\text{$t_5t_6,t_5t_7,t_5t_8,t_7t_8,t_8t_9,t_6t_7)$}}
!{(0,1)}*+{\text{}}
}
\\
\hline

%______________________________________________3____________________________________

\hline 
\xygraph{
!{<0cm,.5cm>;<.5cm,0cm>;<0cm,.5cm>}
!{(0,2.4)}*+{\text{9}} 
!{(0,1.3)}*+{} 
}  &
\xygraph{
 !{<0cm,.5cm>;<.5cm,0cm>;<0cm,.5cm>} 
  !{(-2.1,2.5)}*+{\text{$t_1$}} 
 !{(-.8,2.5)}*+{\text{$t_2$}} 
 !{(-1.2,1.1)}*+{\text{$t_8$}} 
 !{(.4,1.6)}*+{\text{$t_3$}} 
 !{(0.4,.1)}*+{\text{$t_4$}} 
  !{(-.5,-.7)}*+{\text{$t_5$}} 
    !{(-.7,.9)}*+{\text{$t_9$}} 
      !{(-2.6,-.3)}*+{\text{$t_6$}} 
      !{(-2.7,1)}*+{\text{$t_7$}} 
    !{(-2,2.3)}*{\bullet}="v1"
     !{(-.8,2.3)}*{\bullet}="v2"
  !{(-1.5,1)}*{\bullet}="v8"
 !{(.4,1.4)}*{\bullet}="v3"
  !{(.4,.3)}*{\bullet}="v4"
 !{(-.5,-.5)}*{\bullet}="v5"
 !{(-.8,.7)}*{\bullet}="v9"  
 !{(-2.4,0)}*{\bullet}="v6"  
 !{(-2.5,1)}*{\bullet}="v7"
  "v1"-"v2"  "v2"-"v3"   "v3"-"v4"   "v4"-"v5"   "v5"-"v6"   "v6"-"v7"   "v7"-"v1"   "v1"-"v3"   "v1"-"v8"    "v7"-"v2"   "v2"-"v8"   "v5"-"v3"   "v9"-"v4"   "v4"-"v6"   "v7"-"v9"   "v5"-"v9"   "v6"-"v8"   "v6"-"v9"   "v7"-"v8"   "v8"-"v9"   
}& \xygraph{
!{<0cm,.5cm>;<.5cm,0cm>;<0cm,.5cm>}
!{(0,2.5)}*+{\text{$(t_1t_2,t_2t_3,t_3t_4,t_4t_5,t_5t_6,t_6t_7,t_7t_9,$}} 
!{(0,2)}*+{\text{$t_7t_1,t_1t_3,t_1t_8,t_7t_2,t_2t_8,t_5t_3,t_9t_4,$}} 
!{(0,1.5)}*+{\text{$t_4t_6,t_7t_9,t_5t_9,t_6t_8,t_6t_9,t_7t_8,t_8t_9)$}}
!{(0,1)}*+{\text{}}
}
\\
\hline

\end{tabular}
%\label{tab:C-M,G}
\end{table}

\newpage

\begin{table}[H]
%\caption{Edge-critical graphs with $I(G)^{(2)}$ Cohen-Macaulay}
\centering 
\begin{tabular}{|c|c|c|c|}
\hline\hline 
$|V(G)|$ & Graph $G$ & $I=I(G)$ \\

%______________________________________________4____________________________________

\hline
\xygraph{
!{<0cm,.5cm>;<.5cm,0cm>;<0cm,.5cm>}
!{(0,2.4)}*+{\text{9}} 
!{(0,1.3)}*+{} 
}  &
\xygraph{
 !{<0cm,.5cm>;<.5cm,0cm>;<0cm,.5cm>} 
 !{(-0.5,2.7)}*+{\text{$t_1$}} 
 !{(.15,2.2)}*+{\text{$t_2$}} 
 !{(-1.15,2.2)}*+{\text{$t_7$}} 
 !{(.5,1.3)}*+{\text{$t_3$}} 
 !{(-1.5,1.3)}*+{\text{$t_6$}} 
 !{(0.2,.6)}*+{\text{$t_4$}} 
 !{(-1.2,.6)}*+{\text{$t_5$}} 
  !{(0.2,-.2)}*+{\text{$t_8$}} 
 !{(-1.2,-.2)}*+{\text{$t_9$}} 
 !{(-.5,2.5)}*{\bullet}="v1"
  !{(.15,2)}*{\bullet}="v2"
  !{(0.3,1.3)}*{\bullet}="v3"
  !{(0,.8)}*{\bullet}="v4"
 !{(-1,.8)}*{\bullet}="v5"
  !{(-1.3,1.3)}*{\bullet}="v6"
 !{(-1.15,2)}*{\bullet}="v7"
 !{(0,-.2)}*{\bullet}="v8"
 !{(-1,-.2)}*{\bullet}="v9"
 "v1"-"v2"  "v2"-"v3""v4"-"v3""v7"-"v1""v6"-"v7""v6"-"v5""v3"-"v1""v6"-"v1""v6"-"v2""v7"-"v3""v7"-"v2"
 "v1"-"v4""v1"-"v5""v4"-"v2""v5"-"v2""v5"-"v3""v6"-"v3""v6"-"v4""v7"-"v4""v7"-"v5"  "v9"-"v5" "v8"-"v4"
 "v9"-"v8"      
}& \xygraph{
!{<0cm,.5cm>;<.5cm,0cm>;<0cm,.5cm>}
!{(0,2.5)}*+{\text{$(t_1t_2,t_1t_3,t_1t_4,t_1t_5,t_1t_6,t_5t_9,$}} 
!{(0,2)}*+{\text{$t_2t_3,t_2t_4,t_2t_5,t_2t_6,t_3t_4,t_4t_8,$}} 
!{(0,1.5)}*+{\text{$t_3t_5,t_3t_6,t_4t_6,t_5t_6,t_8t_9,$}}
!{(0,1)}*+{\text{$t_7t_1,t_7t_2,t_7t_3,t_7t_4,t_7t_5,t_7t_6)$}}
}
\\
\hline

%______________________________________________5____________________________________

\hline 
\xygraph{
!{<0cm,.5cm>;<.5cm,0cm>;<0cm,.5cm>}
!{(0,2.4)}*+{\text{9}} 
!{(0,1.3)}*+{} 
}  &
\xygraph{
 !{<0cm,.5cm>;<.5cm,0cm>;<0cm,.5cm>} 
  !{(-2.1,2.5)}*+{\text{$t_1$}} 
 !{(-.8,2.5)}*+{\text{$t_2$}} 
  !{(.4,1.6)}*+{\text{$t_3$}} 
 !{(0.4,.1)}*+{\text{$t_4$}} 
  !{(-.5,-.7)}*+{\text{$t_5$}} 
    !{(-2.4,-.2)}*+{\text{$t_6$}} 
      !{(-2.7,1)}*+{\text{$t_7$}} 
      !{(-1.4,1.7)}*+{\text{$t_8$}} 
       !{(-.3,1.6)}*+{\text{$t_9$}} 
    !{(-2,2.3)}*{\bullet}="v1"
     !{(-.8,2.3)}*{\bullet}="v2"
 !{(.4,1.4)}*{\bullet}="v3"
  !{(.4,.3)}*{\bullet}="v4"
 !{(-.5,-.5)}*{\bullet}="v5"
 !{(-2.4,0)}*{\bullet}="v6"  
 !{(-2.5,1)}*{\bullet}="v7"
 !{(-1.5,1.5)}*{\bullet}="v8"
  !{(-.3,1.4)}*{\bullet}="v9"  
  "v1"-"v2"  "v2"-"v3"   "v3"-"v4"   "v4"-"v5"   "v5"-"v6"   "v6"-"v7"   "v7"-"v1" "v1"-"v8" "v1"-"v6" "v3"-"v9" "v3"-"v5" "v1"-"v8" "v4"-"v6" "v4"-"v7" "v4"-"v8" "v5"-"v8" "v7"-"v5" "v8"-"v9"  "v6"-"v8" "v6"-"v9" "v7"-"v8" "v7"-"v9" "v9"-"v8"  "v5"-"v9""v4"-"v9"        
  }& \xygraph{
!{<0cm,.5cm>;<.5cm,0cm>;<0cm,.5cm>}
!{(0,2.5)}*+{\text{$(t_1t_2,t_2t_3,t_3t_4,t_4t_5,t_5t_6,t_6t_7,t_8t_9,$}} 
!{(0,2)}*+{\text{$t_7t_1,t_1t_6,t_1t_8,t_3t_5,t_3t_9,t_4t_6,t_7t_4,t_7t_9,$}} 
!{(0,1.5)}*+{\text{$t_4t_8,t_7t_5,t_5t_8,t_5t_9,t_6t_9,t_6t_8,t_7t_8,t_4t_9)$}}
!{(0,1)}*+{\text{}}
}
\\
\hline

%______________________________________________6____________________________________

\hline 
\xygraph{
!{<0cm,.5cm>;<.5cm,0cm>;<0cm,.5cm>}
!{(0,2.4)}*+{\text{9}} 
!{(0,1.3)}*+{} 
}  &
\xygraph{
 !{<0cm,.5cm>;<.5cm,0cm>;<0cm,.5cm>} 
  !{(-2.1,2.5)}*+{\text{$t_1$}} 
 !{(-.8,2.5)}*+{\text{$t_2$}} 
  !{(.4,1.6)}*+{\text{$t_3$}} 
 !{(0.4,.1)}*+{\text{$t_4$}} 
  !{(-.5,-.7)}*+{\text{$t_5$}} 
    !{(-2.4,-.2)}*+{\text{$t_6$}} 
      !{(-2.9,.8)}*+{\text{$t_7$}} 
      !{(-2.7,1.7)}*+{\text{$t_8$}} 
       !{(-1.3,1.3)}*+{\text{$t_9$}} 
    !{(-2,2.3)}*{\bullet}="v1"
     !{(-.8,2.3)}*{\bullet}="v2"
 !{(.4,1.4)}*{\bullet}="v3"
  !{(.4,.3)}*{\bullet}="v4"
 !{(-.5,-.5)}*{\bullet}="v5"
 !{(-2.4,0)}*{\bullet}="v6"  
 !{(-2.7,.8)}*{\bullet}="v7"
 !{(-2.5,1.7)}*{\bullet}="v8"
  !{(-1,1.2)}*{\bullet}="v9"  
  "v1"-"v2"  "v2"-"v3"   "v3"-"v4"   "v4"-"v5"   "v5"-"v6"   "v6"-"v7"  "v1"-"v8" "v8"-"v7" "v1"-"v7" "v1"-"v3" "v8"-"v2" "v4"-"v2" "v9"-"v2" "v8"-"v7" "v9"-"v3" "v5"-"v3" "v6"-"v4" "v4"-"v6" "v5"-"v7" "v5"-"v9" "v8"-"v6" "v9"-"v6" "v4"-"v9"          
  }& \xygraph{
!{<0cm,.5cm>;<.5cm,0cm>;<0cm,.5cm>}
!{(0,2.5)}*+{\text{$(t_1t_2,t_2t_3,t_3t_4,t_4t_5,t_5t_6,t_6t_7,t_4t_9$,}} 
!{(0,2)}*+{\text{$t_7t_8,t_1t_8,t_7t_1,t_1t_3,t_2t_8,t_2t_9,t_2t_4,$}} 
!{(0,1.5)}*+{\text{$t_3t_9,t_3t_5,t_4t_6,t_7t_5,t_5t_9,t_6t_8,t_6t_9)$}}
!{(0,1)}*+{\text{}}
}
\\
\hline

%______________________________________________7____________________________________

\hline 
\xygraph{
!{<0cm,.5cm>;<.5cm,0cm>;<0cm,.5cm>}
!{(0,2.4)}*+{\text{9}} 
!{(0,1.3)}*+{} 
}  &
\xygraph{
 !{<0cm,.5cm>;<.5cm,0cm>;<0cm,.5cm>} 
  !{(-2.5,2.4)}*+{\text{$t_1$}} 
 !{(-1.3,2.7)}*+{\text{$t_2$}} 
  !{(.1,2.4)}*+{\text{$t_3$}} 
 !{(0.1,-.4)}*+{\text{$t_4$}} 
  !{(-1.3,-.8)}*+{\text{$t_5$}} 
    !{(-2.5,-.4)}*+{\text{$t_6$}} 
      !{(-2.3,.6)}*+{\text{$t_7$}} 
      !{(-1.1,.7)}*+{\text{$t_8$}} 
       !{(-.1,.6)}*+{\text{$t_9$}} 
    !{(-2.5,2.1)}*{\bullet}="v1"
     !{(-1.3,2.5)}*{\bullet}="v2"
 !{(.1,2.1)}*{\bullet}="v3"
  !{(.1,-.1)}*{\bullet}="v4"
 !{(-1.3,-.5)}*{\bullet}="v5"
 !{(-2.5,-.1)}*{\bullet}="v6"  
 !{(-2,.5)}*{\bullet}="v7"
 !{(-1.1,1)}*{\bullet}="v8"
  !{(-.4,.5)}*{\bullet}="v9"  
  "v1"-"v2" "v3"-"v2"   "v3"-"v4"   "v4"-"v5"   "v5"-"v6"   "v1"-"v6" "v1"-"v3" "v1"-"v8"  "v1"-"v7" "v7"-"v2"                                
  "v8"-"v2" "v9"-"v2" "v3"-"v8" "v3"-"v9"    "v4"-"v6" "v4"-"v8" "v4"-"v9" "v5"-"v9" "v5"-"v7" "v7"-"v6" "v8"-"v6" "v9"-"v7"
  }& \xygraph{
!{<0cm,.5cm>;<.5cm,0cm>;<0cm,.5cm>}
!{(0,2.5)}*+{\text{$(t_1t_2,t_2t_3,t_3t_4,t_4t_5,t_5t_6,t_1t_6,$}} 
!{(0,2)}*+{\text{$t_1t_3,t_1t_7,t_1t_8,t_2t_9,t_2t_7,t_2t_8,t_9t_3,t_3t_8,$}} 
!{(0,1.5)}*+{\text{$t_4t_6,t_4t_8,t_4t_9,t_9t_5,t_5t_7,t_7t_6,t_8t_6,t_7t_9)$}}
!{(0,1)}*+{\text{}}
}
\\
\hline

%______________________________________________8____________________________________

\hline
\xygraph{
!{<0cm,.5cm>;<.5cm,0cm>;<0cm,.5cm>}
!{(0,2.4)}*+{\text{9}} 
!{(0,1.3)}*+{} 
}  & 
\xygraph{
 !{<0cm,.5cm>;<.5cm,0cm>;<0cm,.5cm>} 
  !{(-2.5,2.4)}*+{\text{$t_1$}} 
 !{(.1,2.3)}*+{\text{$t_2$}} 
  !{(.8,.2)}*+{\text{$t_3$}} 
 !{(0,-.3)}*+{\text{$t_4$}} 
  !{(-1.8,-.4)}*+{\text{$t_5$}} 
    !{(-2.7,.1)}*+{\text{$t_6$}} 
      !{(-2.3,1.2)}*+{\text{$t_7$}} 
      !{(-1.2,1.9)}*+{\text{$t_8$}} 
       !{(-.1,1.2)}*+{\text{$t_9$}} 
    !{(-2.5,2.1)}*{\bullet}="v1"
     !{(-0,2.1)}*{\bullet}="v2"
 !{(.5,.2)}*{\bullet}="v3"
  !{(-.2,-.1)}*{\bullet}="v4"
 !{(-1.8,-.2)}*{\bullet}="v5"
 !{(-2.7,.3)}*{\bullet}="v6"  
 !{(-2,1.2)}*{\bullet}="v7"
 !{(-1.2,1.6)}*{\bullet}="v8"
  !{(-.4,1.2)}*{\bullet}="v9"  
  "v1"-"v2" "v2"-"v3"   "v3"-"v4"   "v4"-"v5"   "v5"-"v6"   "v6"-"v1"  "v1"-"v7" "v2"-"v9"  "v3"-"v5"  "v3"-"v8"  "v3"-"v9" "v4"-"v6" "v4"-"v7" "v4"-"v8" "v4"-"v9" "v5"-"v7" "v5"-"v8" "v5"-"v9"   "v6"-"v7"  "v6"-"v8"  "v8"-"v7"  "v9"-"v8"  
  }& \xygraph{
!{<0cm,.5cm>;<.5cm,0cm>;<0cm,.5cm>}
!{(0,2.5)}*+{\text{$(t_1t_2,t_2t_3,t_3t_4,t_4t_5,t_5t_6,t_1t_6,t_7t_8,$}} 
!{(0,2)}*+{\text{$t_7t_1,t_2t_9,t_3t_5,t_3t_9,t_3t_8,t_4t_6,t_7t_4,t_8t_9,$}} 
!{(0,1.5)}*+{\text{$t_4t_8,t_4t_9,t_5t_9,t_5t_8,t_5t_7,t_7t_6,t_8t_6)$}}
!{(0,1)}*+{\text{}}
}
\\
\hline

\end{tabular}
%\label{tab:C-M,G}
\end{table}

\newpage

\begin{table}[H]
%\caption{Edge-critical graphs with $I(G)^{(2)}$ Cohen-Macaulay}
\centering 
\begin{tabular}{|c|c|c|}
\hline\hline 
$|V(G)|$ &Graph $G$ & $I=I(G)$ \\

%______________________________________________9____________________________________

\hline 
\xygraph{
!{<0cm,.5cm>;<.5cm,0cm>;<0cm,.5cm>}
!{(0,2.4)}*+{\text{9}} 
!{(0,1.3)}*+{} 
}  &
\xygraph{
 !{<0cm,.5cm>;<.5cm,0cm>;<0cm,.5cm>} 
  !{(-2.7,2.4)}*+{\text{$t_9$}} 
 !{(-1.8,2.4)}*+{\text{$t_8$}} 
  !{(-1,2.4)}*+{\text{$t_7$}} 
 !{(-.1,2.4)}*+{\text{$t_6$}} 
  !{(.8,2.4)}*+{\text{$t_5$}} 
    !{(.4,.1)}*+{\text{$t_4$}} 
      !{(-.4,.1)}*+{\text{$t_3$}} 
      !{(-1.4,.1)}*+{\text{$t_2$}} 
       !{(-2.3,.1)}*+{\text{$t_1$}} 
    !{(-2.7,2.1)}*{\bullet}="v9"
     !{(-1.8,2.1)}*{\bullet}="v8"
 !{(-1,2.1)}*{\bullet}="v7"
  !{(-.1,2.1)}*{\bullet}="v6"
 !{(.8,2.1)}*{\bullet}="v5"
 !{(.4,.3)}*{\bullet}="v4"  
 !{(-.4,.3)}*{\bullet}="v3"
 !{(-1.4,.3)}*{\bullet}="v2"
  !{(-2.3,.3)}*{\bullet}="v1"  
     "v1"-"v2" "v2"-"v3"   "v3"-"v4"   "v4"-"v5"   "v5"-"v6"  "v6"-"v7" "v7"-"v8" "v8"-"v9" "v9"-"v1"   "v2"-"v7""v8"-"v3" "v5"-"v3" "v4"-"v6"    
  }& \xygraph{
!{<0cm,.5cm>;<.5cm,0cm>;<0cm,.5cm>}
!{(0,2.5)}*+{\text{$(t_1t_2,t_2t_3,t_3t_4,t_4t_5,t_5t_6,t_6t_7,$}} 
!{(0,2)}*+{\text{$t_7t_8,t_8t_9,t_1t_9,t_7t_2,t_3t_8,t_5t_3,t_6t_4)$}} 
!{(0,1.5)}*+{\text{}}
!{(0,1)}*+{\text{}}
}
\\
\hline

%______________________________________________10____________________________________

\hline  
\xygraph{
!{<0cm,.5cm>;<.5cm,0cm>;<0cm,.5cm>}
!{(0,2.4)}*+{\text{9}} 
!{(0,1.3)}*+{} 
}  &
\xygraph{
 !{<0cm,.5cm>;<.5cm,0cm>;<0cm,.5cm>} 
  !{(-2,-0.4)}*+{\text{$t_1$}} 
 !{(-1,0)}*+{\text{$t_2$}} 
  !{(.3,-.4)}*+{\text{$t_3$}} 
 !{(1.2,0)}*+{\text{$t_4$}} 
  !{(1.2,1.5)}*+{\text{$t_5$}} 
   !{(.3,1.9)}*+{\text{$t_6$}} 
      !{(-1,1.5)}*+{\text{$t_7$}} 
      !{(-2,1.9)}*+{\text{$t_8$}} 
        !{(-2.6,.91)}*+{\text{$t_9$}} 
    !{(-2,-.2)}*{\bullet}="v1"
     !{(-1,.2)}*{\bullet}="v2"
 !{(.3,-.2)}*{\bullet}="v3"
  !{(1.2,.2)}*{\bullet}="v4"
 !{(1.2,1.3)}*{\bullet}="v5"
 !{(.3,1.7)}*{\bullet}="v6"  
 !{(-1,1.3)}*{\bullet}="v7"
 !{(-2,1.7)}*{\bullet}="v8"
  !{(-2.5,.71)}*{\bullet}="v9"    
  "v1"-"v2" "v2"-"v3"   "v3"-"v4"   "v4"-"v5"   "v5"-"v6"  "v6"-"v7" "v7"-"v8" "v8"-"v9" "v9"-"v1"   "v1"-"v7"  "v1"-"v8"  "v2"-"v4" "v2"-"v8" "v2"-"v9"  "v3"-"v5"  "v3"-"v6"   "v4"-"v5"  "v4"-"v6"   "v5"-"v7" "v9"-"v7"
  }& \xygraph{
!{<0cm,.5cm>;<.5cm,0cm>;<0cm,.5cm>}
!{(0,2.5)}*+{\text{$(t_1t_2,t_2t_3,t_3t_4,t_4t_5,t_5t_6,t_6t_7,$}} 
!{(0,2)}*+{\text{$t_7t_8,t_8t_9,t_1t_9,t_7t_1,t_1t_8,t_2t_9,t_2t_8,t_2t_4,$}} 
!{(0,1.5)}*+{\text{$t_3t_6,t_3t_5,t_6t_4,t_5t_7,t_7t_9)$}}
!{(0,1)}*+{\text{}}
}
\\
\hline

%______________________________________________11____________________________________

\hline  
\xygraph{
!{<0cm,.5cm>;<.5cm,0cm>;<0cm,.5cm>}
!{(0,2.4)}*+{\text{9}} 
!{(0,1.3)}*+{} 
}  &
\xygraph{
 !{<0cm,.5cm>;<.5cm,0cm>;<0cm,.5cm>} 
  !{(-2,2.5)}*+{\text{$t_1$}} 
 !{(0,2.5)}*+{\text{$t_2$}} 
  !{(1.2,1.7)}*+{\text{$t_3$}} 
 !{(.8,1)}*+{\text{$t_4$}} 
  !{(-1,.3)}*+{\text{$t_5$}} 
   !{(-2.7,1)}*+{\text{$t_6$}} 
      !{(-3.2,1.7)}*+{\text{$t_7$}} 
      !{(-1.3,1.7)}*+{\text{$t_8$}} 
        !{(-.9,2.1)}*+{\text{$t_9$}} 
    !{(-2,2.3)}*{\bullet}="v1"
     !{(0,2.3)}*{\bullet}="v2"
 !{(1,1.7)}*{\bullet}="v3"
  !{(.5,1)}*{\bullet}="v4"
 !{(-1,.5)}*{\bullet}="v5"
 !{(-2.5,1)}*{\bullet}="v6"  
 !{(-3,1.7)}*{\bullet}="v7"
 !{(-1.3,1.5)}*{\bullet}="v8"
  !{(-.7,2)}*{\bullet}="v9"  
    "v1"-"v2" "v2"-"v3"   "v3"-"v4"   "v4"-"v5"   "v5"-"v6"  "v6"-"v7" "v7"-"v1" "v6"-"v1" "v2"-"v9" "v2"-"v4"   
    "v3"-"v5" "v3"-"v8" "v3"-"v9"  "v4"-"v8" "v4"-"v9" "v5"-"v7" "v5"-"v8" "v5"-"v9"   "v6"-"v8"  "v7"-"v8" "v8"-"v9"  
  }& \xygraph{
!{<0cm,.5cm>;<.5cm,0cm>;<0cm,.5cm>}
!{(0,2.5)}*+{\text{$(t_1t_2,t_2t_3,t_3t_4,t_4t_5,t_5t_6,t_6t_7,t_8t_9,$}} 
!{(0,2)}*+{\text{$t_7t_1,t_1t_6,t_2t_9,t_2t_4,t_3t_8,t_3t_9,t_3t_5,$}} 
!{(0,1.5)}*+{\text{$t_4t_9,t_4t_8,t_5t_9,t_5t_7,t_5t_8,t_8t_6,t_7t_8)$}}
!{(0,1)}*+{\text{}}
}
\\
\hline

%______________________________________________12____________________________________

\hline  
\xygraph{
!{<0cm,.5cm>;<.5cm,0cm>;<0cm,.5cm>}
!{(0,2.4)}*+{\text{9}} 
!{(0,1.3)}*+{} 
}  &
\xygraph{
 !{<0cm,.5cm>;<.5cm,0cm>;<0cm,.5cm>} 
  !{(-1.3,3.2)}*+{\text{$t_1$}} 
 !{(1,2.5)}*+{\text{$t_2$}} 
  !{(1.3,1.5)}*+{\text{$t_3$}} 
 !{(1,.5)}*+{\text{$t_4$}} 
  !{(-.5,-.4)}*+{\text{$t_5$}} 
   !{(-2,-.4)}*+{\text{$t_6$}} 
      !{(-3.6,.5)}*+{\text{$t_7$}} 
      !{(-4,1.5)}*+{\text{$t_8$}} 
        !{(-3.6,2.5)}*+{\text{$t_9$}} 
    !{(-1.3,3)}*{\bullet}="v1"
     !{(.8,2.5)}*{\bullet}="v2"
 !{(1.1,1.5)}*{\bullet}="v3"
  !{(.8,.5)}*{\bullet}="v4"
 !{(-.5,-.2)}*{\bullet}="v5"
 !{(-2,-.2)}*{\bullet}="v6"  
 !{(-3.4,.5)}*{\bullet}="v7"
 !{(-3.8,1.5)}*{\bullet}="v8"
  !{(-3.4,2.5)}*{\bullet}="v9"    
  "v1"-"v2" "v1"-"v3" "v1"-"v4" "v1"-"v5" "v1"-"v6" "v1"-"v7" "v1"-"v8" "v1"-"v9" "v2"-"v3" "v2"-"v4" "v2"-"v5" "v2"-"v6" "v2"-"v7" "v2"-"v8" "v2"-"v9" "v3"-"v4" "v3"-"v5" "v3"-"v6" "v3"-"v7" "v3"-"v8" "v3"-"v9" "v4"-"v5" "v4"-"v6" "v4"-"v7" "v4"-"v8" "v4"-"v9" "v5"-"v6" "v5"-"v7" "v5"-"v8" "v5"-"v9"
  "v6"-"v7" "v6"-"v8" "v6"-"v9" "v7"-"v8" "v7"-"v9"  "v8"-"v9"
  }& \xygraph{
!{<0cm,.5cm>;<.5cm,0cm>;<0cm,.5cm>}
!{(0,2.5)}*+{\text{$(t_1t_2,t_1t_3,t_1t_4,t_1t_5,t_1t_6,t_1t_7,t_1t_8,t_1t_9,$}} 
!{(0,2)}*+{\text{$t_2t_3,t_2t_4,t_2t_5,t_2t_6,t_2t_7,t_2t_8,t_2t_9,t_3t_4,$}} 
!{(0,1.5)}*+{\text{$t_3t_5,t_3t_6,t_3t_7,t_3t_8,t_3t_9,t_4t_5,t_4t_6,t_4t_7,$}}
!{(0,1)}*+{\text{$t_4t_8,t_4t_9,t_5t_6,t_5t_7,t_5t_8,t_5t_9,t_6t_7,$}}
!{(0,.5)}*+{\text{$t_6t_8,t_6t_9,t_7t_8,t_7t_9,t_8t_9)$}}
}
\\
\hline

%______________________________________________13____________________________________

\hline  
\xygraph{
!{<0cm,.5cm>;<.5cm,0cm>;<0cm,.5cm>}
!{(0,2.4)}*+{\text{9}} 
!{(0,1.3)}*+{} 
}  &
\xygraph{
 !{<0cm,.5cm>;<.5cm,0cm>;<0cm,.5cm>} 
  !{(-1,3.2)}*+{\text{$t_1$}} 
 !{(1,3.2)}*+{\text{$t_2$}} 
  !{(2.2,2)}*+{\text{$t_3$}} 
 !{(1,-.2)}*+{\text{$t_4$}} 
  !{(-1,-.2)}*+{\text{$t_5$}} 
   !{(-2.2,2)}*+{\text{$t_6$}} 
      !{(0,2.6)}*+{\text{$t_7$}} 
      !{(-.8,1.3)}*+{\text{$t_8$}} 
        !{(.8,1.3)}*+{\text{$t_9$}} 
    !{(-1,3)}*{\bullet}="v1"
     !{(1,3)}*{\bullet}="v2"
 !{(2,2)}*{\bullet}="v3"
  !{(1,0)}*{\bullet}="v4"
 !{(-1,0)}*{\bullet}="v5"
 !{(-2,2)}*{\bullet}="v6"  
 !{(0,2.4)}*{\bullet}="v7"
 !{(-.6,1.3)}*{\bullet}="v8"
  !{(.6,1.3)}*{\bullet}="v9"    
   "v1"-"v2" "v2"-"v3"   "v3"-"v4"   "v4"-"v5"   "v5"-"v6"  "v6"-"v1"  "v1"-"v8" "v2"-"v9"   "v3"-"v7"   "v4"-"v8" "v5"-"v9" "v6"-"v7" "v8"-"v7" "v9"-"v7" "v8"-"v9"    
  }& \xygraph{
!{<0cm,.5cm>;<.5cm,0cm>;<0cm,.5cm>}
!{(0,2.5)}*+{\text{$(t_1t_2,t_2t_3,t_3t_4,t_4t_5,t_5t_6,t_1t_6,$}} 
!{(0,2)}*+{\text{$t_1t_8,t_2t_9,t_3t_7,t_4t_8,t_5t_9,t_6t_7,t_7t_8,$}} 
!{(0,1.5)}*+{\text{$t_7t_9,t_8t_9)$}}
!{(0,1)}*+{\text{}}
}
\\
\hline

\end{tabular}
%\label{tab:C-M,G}
\end{table}

\newpage

\begin{table}[H]
%\caption{Edge-critical graphs with $I(G)^{(2)}$ Cohen-Macaulay}
\centering 
\begin{tabular}{|c|c|c|}
\hline\hline 
$|V(G)|$ &Graph $G$ & $I=I(G)$ \\

%______________________________________________14____________________________________

\hline  
\xygraph{
!{<0cm,.5cm>;<.5cm,0cm>;<0cm,.5cm>}
!{(0,2.4)}*+{\text{9}} 
!{(0,1.3)}*+{} 
}  &
\xygraph{
 !{<0cm,.5cm>;<.5cm,0cm>;<0cm,.5cm>} 
  !{(-1.5,-1.7)}*+{\text{$t_1$}} 
 !{(0,-1.7)}*+{\text{$t_2$}} 
  !{(.9,-.7)}*+{\text{$t_3$}} 
 !{(2,-1.2)}*+{\text{$t_4$}} 
  !{(2.2,1.1)}*+{\text{$t_5$}} 
   !{(.2,1.7)}*+{\text{$t_6$}} 
      !{(-1.5,1.7)}*+{\text{$t_7$}} 
      !{(-2.2,0)}*+{\text{$t_8$}} 
        !{(-.7,-.4)}*+{\text{$t_9$}} 
    !{(-1.5,-1.5)}*{\bullet}="v1"
     !{(0,-1.5)}*{\bullet}="v2"
 !{(.8,-.5)}*{\bullet}="v3"
  !{(2,-1)}*{\bullet}="v4"
 !{(2,1)}*{\bullet}="v5"
 !{(0,1.5)}*{\bullet}="v6"  
 !{(-1.5,1.5)}*{\bullet}="v7"
 !{(-2,0)}*{\bullet}="v8"
  !{(-.5,-.5)}*{\bullet}="v9"    
  "v1"-"v2" "v2"-"v3"   "v3"-"v4"   "v4"-"v5"   "v5"-"v6"  "v6"-"v7" "v7"-"v8" "v8"-"v1" "v1"-"v9" "v9"-"v2" "v6"-"v2" "v7"-"v3" "v9"-"v3" "v9"-"v6"     
  }& \xygraph{
!{<0cm,.5cm>;<.5cm,0cm>;<0cm,.5cm>}
!{(0,2)}*+{\text{$(t_1t_2,t_2t_3,t_3t_4,t_4t_5,t_5t_6,t_6t_7,$}} 
!{(0,1.5)}*+{\text{$t_7t_8,t_8t_1,t_1t_9,t_2t_6,t_2t_9,t_3t_9,t_3t_7,$}} 
!{(0,1)}*+{\text{$t_6t_9)$}}
}
\\
\hline

%______________________________________________15____________________________________

\hline  
\xygraph{
!{<0cm,.5cm>;<.5cm,0cm>;<0cm,.5cm>}
!{(0,2.4)}*+{\text{9}} 
!{(0,1.3)}*+{} 
}  &
\xygraph{
 !{<0cm,.5cm>;<.5cm,0cm>;<0cm,.5cm>} 
  !{(-1.5,1.7)}*+{\text{$t_1$}} 
 !{(-2.5,.3)}*+{\text{$t_2$}} 
  !{(-1,-.7)}*+{\text{$t_3$}} 
 !{(0,-1,2)}*+{\text{$t_4$}} 
  !{(1,-.7)}*+{\text{$t_5$}} 
   !{(2.5,.3)}*+{\text{$t_6$}} 
      !{(1.5,1.7)}*+{\text{$t_7$}} 
      !{(.5,1)}*+{\text{$t_8$}} 
        !{(-.5,1)}*+{\text{$t_9$}} 
    !{(-1.5,1.5)}*{\bullet}="v1"
     !{(-2.5,.5)}*{\bullet}="v2"
 !{(-1,-.5)}*{\bullet}="v3"
  !{(0,-1)}*{\bullet}="v4"
 !{(1,-.5)}*{\bullet}="v5"
 !{(2.5,.5)}*{\bullet}="v6"  
 !{(1.5,1.5)}*{\bullet}="v7"
 !{(.5,.8)}*{\bullet}="v8"
  !{(-.5,.8)}*{\bullet}="v9"    
  "v1"-"v2" "v2"-"v3"   "v3"-"v4"   "v4"-"v5"   "v5"-"v6"  "v6"-"v7" "v7"-"v8" "v8"-"v9" "v9"-"v1"  "v3"-"v5"  "v3"-"v8"   "v4"-"v9"  "v4"-"v8"  "v9"-"v5"  "v8"-"v9" 
  }& \xygraph{
!{<0cm,.5cm>;<.5cm,0cm>;<0cm,.5cm>}
!{(0,2.5)}*+{\text{$(t_1t_2,t_2t_3,t_3t_4,t_4t_5,t_5t_6,t_6t_7,$}} 
!{(0,2)}*+{\text{$t_7t_8,t_8t_9,t_1t_9,t_3t_5,t_3t_8,t_4t_9,t_4t_8,$}} 
!{(0,1.5)}*+{\text{$t_5t_9,t_8t_9)$}}
}
\\
\hline

%______________________________________________16____________________________________

\hline  
\xygraph{
!{<0cm,.5cm>;<.5cm,0cm>;<0cm,.5cm>}
!{(0,2.4)}*+{\text{9}} 
!{(0,1.3)}*+{} 
}  &
\xygraph{
 !{<0cm,.5cm>;<.5cm,0cm>;<0cm,.5cm>} 
  !{(-2.5,-1.7)}*+{\text{$t_1$}} 
 !{(-1.2,-1.4)}*+{\text{$t_2$}} 
  !{(-.1,-.4)}*+{\text{$t_3$}} 
 !{(1,-1.1)}*+{\text{$t_4$}} 
  !{(1,1.1)}*+{\text{$t_5$}} 
   !{(-.2,1.1)}*+{\text{$t_6$}} 
      !{(-1.2,1.4)}*+{\text{$t_7$}} 
      !{(-2.5,1.7)}*+{\text{$t_8$}} 
        !{(-3.6,0)}*+{\text{$t_9$}} 
    !{(-2.5,-1.5)}*{\bullet}="v1"
     !{(-1.2,-1.2)}*{\bullet}="v2"
 !{(-.2,-.1)}*{\bullet}="v3"
  !{(1,-.9)}*{\bullet}="v4"
 !{(1,.9)}*{\bullet}="v5"
 !{(-.2,.9)}*{\bullet}="v6"  
 !{(-1.2,1.2)}*{\bullet}="v7"
 !{(-2.5,1.5)}*{\bullet}="v8"
  !{(-3.3,0)}*{\bullet}="v9"  
   "v1"-"v2" "v2"-"v3"   "v3"-"v4"   "v4"-"v5"   "v5"-"v6"  "v6"-"v7" "v7"-"v8" "v8"-"v9" "v9"-"v1"    "v1"-"v3"  "v1"-"v7"  "v1"-"v8"      "v2"-"v7"  "v2"-"v8"   "v2"-"v9"       "v3"-"v5"  "v3"-"v9"  "v3"-"v8"  "v6"-"v4"  "v7"-"v9"   
  }& \xygraph{
!{<0cm,.5cm>;<.5cm,0cm>;<0cm,.5cm>}
!{(0,2)}*+{\text{$(t_1t_2,t_2t_3,t_3t_4,t_4t_5,t_5t_6,t_6t_7,t_7t_9,$}} 
!{(0,1.5)}*+{\text{$t_7t_8,t_8t_9,t_1t_9,t_1t_3,t_1t_7,t_1t_8,t_2t_7,$}} 
!{(0,1)}*+{\text{$t_2t_8,t_2t_9,t_3t_5,t_3t_8,t_3t_9,t_4t_6)$}}
}
\\
\hline

%______________________________________________17____________________________________

\hline  
\xygraph{
!{<0cm,.5cm>;<.5cm,0cm>;<0cm,.5cm>}
!{(0,2.4)}*+{\text{9}} 
!{(0,1.3)}*+{} 
}  &
\xygraph{
 !{<0cm,.5cm>;<.5cm,0cm>;<0cm,.5cm>} 
  !{(-1,2.3)}*+{\text{$t_1$}} 
  !{(0.4,1.6)}*+{\text{$t_3$}} 
 !{(0.3,1)}*+{\text{$t_4$}}
 !{(-1.3,1)}*+{\text{$t_5$}} 
 !{(-1.4,1.6)}*+{\text{$t_6$}} 
 !{(0,2.3)}*+{\text{$t_2$}} 
      !{(-.6,0)}*+{\text{$t_7$}} 
      !{(-2.5,1.7)}*+{\text{$t_8$}} 
        !{(-3.6,0)}*+{\text{$t_9$}} 
     !{(-1,2.1)}*{\bullet}="v1"
 !{(0.2,1.4)}*{\bullet}="v3"
  !{(0,.8)}*{\bullet}="v4"
 !{(-1,.8)}*{\bullet}="v5"
 !{(-1.2,1.4)}*{\bullet}="v6"
  !{(0,2.1)}*{\bullet}="v2"
 !{(-.6,0.2)}*{\bullet}="v7"
 !{(-2.5,1.5)}*{\bullet}="v8"
  !{(-3.3,0)}*{\bullet}="v9"  
 "v1"-"v2""v3"-"v2""v4"-"v3""v5"-"v4""v6"-"v5""v1"-"v6""v1"-"v3""v1"-"v4""v1"-"v5""v4"-"v2""v5"-"v2""v6"-"v2"
"v5"-"v3""v6"-"v3""v6"-"v4" "v7"-"v4""v7"-"v3" "v7"-"v2" "v7"-"v5""v7"-"v9" "v8"-"v6" "v9"-"v8" "v1"-"v8"
  }& \xygraph{
!{<0cm,.5cm>;<.5cm,0cm>;<0cm,.5cm>}
!{(0,2)}*+{\text{$(t_1t_2,t_1t_3,t_1t_4,t_1t_5,t_2t_3,t_2t_4,t_2t_5,t_2t_6,t_2t_7,$}} 
!{(0,1.5)}*+{\text{$t_3t_4,t_3t_5,t_3t_6,t_3t_7,t_4t_5,t_4t_6,t_4t_7,t_5t_6,$}} 
!{(0,1)}*+{\text{$t_5t_7,t_6t_1,t_6t_8,t_1t_8,t_7t_9,t_8t_9)$}}
!{(0,.5)}*+{\text{}}
}
\\
\hline
\end{tabular}
\end{table}

\newpage

\bibliographystyle{plain}

\begin{thebibliography}{10}
\bibitem{Allan-Laskar} R. B. Allan and R. Laskar, 
On domination and independent domination numbers of a graph, Discrete
Math. {\bf 23} (1978), 73--76.

\bibitem{Beineke-Harary-Plummer} L. W. Beineke, F. Harary and M. D.
Plummer, On the critical lines of a graph, Pacific J. Math. {\bf 22}
(1967), 205--212. 


\bibitem{bjorner-topological} A. Bj\"orner, Topological methods, 
Handbook of combinatorics, Vol. 1, 2,  1819--1872, Elsevier,
Amsterdam, 
1995. 


\bibitem{Caviglia-et-al} G. Caviglia, H. T.
H\`{a}, J. Herzog, M. Kummini, N. Terai and N. V. Trung, 
Depth and regularity modulo a principal ideal, 
J. Algebraic Combin. {\bf 49} (2019), no. 1, 1--20.

\bibitem{min-dis-generalized} S. M. Cooper, A. Seceleanu, S. O. Toh\v{a}neanu,
M. Vaz Pinto and R. H. Villarreal, 
Generalized minimum distance functions and algebraic invariants of
Geramita ideals, Adv. in Appl. Math. {\bf 112} (2020), 101940.

\bibitem{Dao-Stefani-Grifo-Huneke-Nunez} H. Dao, A. De Stefani, E. Grifo, C. Huneke and L.
N\'u\~nez-Betancourt, \textit{Symbolic powers of ideals}, Singularities and foliations, 
geometry, topology and applications, Springer Proc. Math. Stat., vol. 222,
Springer, Cham, 2018, pp. 387--432.

\bibitem{Eisen}{D. Eisenbud, {\it Commutative Algebra with a view
toward Algebraic Geometry\/}, Graduate
Texts in  Mathematics {\bf 150}, Springer-Verlag, New York, 1995.}

%\bibitem{eisenbud-syzygies} D. Eisenbud, {\it The geometry of syzygies: A
%second course in commutative algebra and algebraic geometry},
%Graduate Texts in Mathematics {\bf 229}, Springer-Verlag, New York,
%2005.

\bibitem{Faridi} S. Faridi, Simplicial trees are sequentially
Cohen-Macaulay, J. Pure Appl. Algebra {\bf 190} (2004), 121--136. 


\bibitem{Fro4}{R. Fr\"oberg, On Stanley--Reisner rings, 
in {\em Topics in algebra\/}, Part 2. Polish Scientific Publishers,
1990, pp. 57--70.}

\bibitem{independent-domination} W. Goddard and M. A. Henning, 
Independent domination in graphs: a survey and recent results, 
Discrete Math. {\bf 313} (2013), no. 7, 839--854. 

\bibitem{mac2} D. Grayson and M. Stillman,
{\em Macaulay\/}$2$, {\em a software system for research in algebraic geometry},
available at \texttt{http://www.math.uiuc.edu/Macaulay2/}.

\bibitem{HaM} H. T. H$\rm \grave{a}$ and S. Morey, 
Embedded associated
primes of powers of square-free monomial ideals, 
J. Pure Appl. Algebra {\bf 214} (2010), no. 4, 301--308.

\bibitem{Ha-Woodroofe} H. T. H$\rm \grave{a}$ and R. Woodroofe, 
Results on the regularity of square-free monomial ideals, Adv. in
Appl. Math. {\bf 58} (2014), 21--36.

\bibitem{Har}{F. Harary, {\it Graph Theory\/}, Addison-Wesley, 
Reading, MA, 1972.}

\bibitem{hhz-ejc} J. Herzog, T. Hibi and X. Zheng, 
Dirac's theorem on chordal graphs and Alexander duality. 
European J. Combin. {\bf 25} (2004), no. 7, 949--960.

\bibitem{Radical-Herzog}{J. Herzog, 
Y. Takayama and N. Terai, On the radical of a monomial ideal, 
Arch. Math. {\bf 85} (2005), 397--408.}

\bibitem{Hoang-VJM} D. T. Hoang, Cohen--Macaulayness of saturation of
the second power of edge ideals, Vietnam J. Math. {\bf 44} (2016),
no. 4, 649--664.

\bibitem{Hoang-etal} D. T. Hoang, N. C. Minh and T. N. Trung, 
Combinatorial characterizations of the Cohen-Macaulayness of the
second power of 
edge ideals, J. Combin. Theory Ser. A  {\bf 120}  (2013),  no. 5,
1073--1086. 

\bibitem{Hoang-Rinaldo-Terai} D. T. Hoang, G. Rinaldo and N. Terai, 
Cohen-Macaulay and $\mathrm(S_2)$ properties of the second power of
squarefree monomial ideals, Mathematics {\bf 7} (2019), no. 8, 
Article number 684.

\bibitem{hoang-gorenstein-second-jaco}
D. T. Hoang and T. N. Trung, 
A characterization of triangle-free Gorenstein graphs and
Cohen-Macaulayness of 
second powers of edge ideals, J. Algebraic Combin. {\bf 43}  (2016),
no. 2, 325--338. 

\bibitem{kalai-meshulam} G. Kalai and R. Meshulam,  Unions and
intersections of Leray complexes, J. Combin. Theory Ser. A {\bf 113} 
(2006), 1586--1592.

\bibitem{katzman1} M. Katzman, Characteristic-independence of Betti numbers of
graph ideals, J. Combin. Theory Ser. A {\bf 113} (2006), no. 3,
435--454.

\bibitem{Levit-Mandrescu} V. E. Levit and E. Mandrescu,
$1$-well-covered graphs revisited, European J. Combin. {\bf 80}
(2019), 261--272.

%\bibitem{Mat}{H. Matsumura, {\it Commutative Algebra\/}, Benjamin-Cummings, Reading, MA, 1980.}


\bibitem{Mats}{H. Matsumura, {\it Commutative Ring Theory\/},
Cambridge
Studies in Advanced Mathematics {\bf 8},
Cambridge University Press, 1986.}

\bibitem{Minh-Trung} N. C. Minh and N. V. Trung, 
Cohen--Macaulayness of powers of two-dimensional squarefree monomial
ideals, J. Algebra {\bf 322} (2009), no. 12, 4219--4227. 

\bibitem{Minh-Trung-adv} N. C. Minh and N. V. Trung, Cohen--Macaulayness of monomial ideals
and symbolic powers of Stanley--Reisner ideals, Adv. Math. {\bf 226}
(2011), no. 2, 1285--1306.

\bibitem{edge-ideals} S. Morey and R. H. Villarreal, Edge ideals:
algebraic and combinatorial properties, in {\it Progress in
Commutative Algebra, Combinatorics and Homology, Vol. 1\/} (C.
Francisco, L. C.  Klingler, S. Sather-Wagstaff and J. C. Vassilev,
Eds.), De Gruyter, Berlin, 2012, pp. 85--126. 

\bibitem{footprint-ci}  L. N\'u\~nez-Betancourt, Y. Pitones and R. H.
Villarreal, Footprint and minimum distance functions,
Commun. Korean Math. Soc. {\bf 33} (2018), no. 1, 85--101.

\bibitem{Ore} O. Ore, \textit{Theory of graphs}, American
Mathematical Society Colloq. Publ. Vol. 38, Providence,
R.I. 
1962.

\bibitem{Pinter-jgt} M. R. Pinter, A class of planar well-covered graphs
with girth four, J. Graph Theory {\bf 19} (1995), 69--81.

%\bibitem{Pinter-ars} M. R. Pinter, A class of well-covered graphs
%with girth four, Ars Combin. {\bf 45} (1997), 241--255. 

\bibitem{Plummer}{M.~D. Plummer, On a family of line-critical graphs,
Monatsh. Math. {\bf 71} (1967), 40--48.}

\bibitem{Plummer-survey} M. D. Plummer, Well-covered graphs: a
survey, Quaestiones Math. {\bf 16} (1993), no. 3, 253--287. 

\bibitem{provan-billera} J. S. Provan and L. J. Billera, 
Decompositions of simplicial complexes related to diameters of convex
polyhedra, Math. Oper. Res. {\bf 5} (1980), no. 4, 576--594.

\bibitem{B-Small}{B. Small, \textit{On $\alpha$-critical graphs and
their construction}, Ph.D. thesis, Washington State University, 2015.}

\bibitem{Sta2}{R. Stanley, {\it Combinatorics and Commutative
Algebra\/}, Birkh{\"a}user Boston, 2nd ed., 1996.}

\bibitem{Staples-thesis} J. W. Staples, On some subclasses of
well-covered graphs, Ph.D. Dissertation, Vanderbilt University, Nashville, TN, 1975.


\bibitem{Staples} J. Staples, On some subclasses of well-covered
graphs, J. Graph Theory {\bf 3} (1979), no. 2, 197--204.

%\bibitem{terai} N. Terai, Alexander duality theorem and
%Stanley--Reisner rings, S\=urikaisekikenky\=usho K\=oky\=uroku {\bf
%1078} (1999), 174--184.  

\bibitem{Toft}{B. Toft, Colouring, stable sets and 
perfect graphs, in {\it Handbook of Combinatorics I}, Elsevier, 1995,
pp. 233--288.} 

\bibitem{vantuyl} A. Van Tuyl, Sequentially Cohen-Macaulay bipartite graphs:
vertex decomposability and 
regularity,  Arch. Math. {\bf 93} (2009), 451--459. 

\bibitem{bipartite-scm} A. Van Tuyl and R. H. Villarreal, Shellable graphs and 
sequentially Cohen-Macaulay bipartite graphs, 
J. Combin. Theory Ser. A {\bf 115} (2008), no. 5, 799-814.

%\bibitem{monalg}{R. H. Villarreal, {\it Monomial Algebras\/},
%Monographs and
%Textbooks in Pure and Applied Mathematics {\bf 238}, Marcel
%Dekker, Inc., New York, 2001.}


\bibitem{monalg-rev} R. H. Villarreal, {\it Monomial Algebras, Second Edition\/},
Monographs and Research Notes in Mathematics, Chapman and Hall/CRC,
Boca Raton, FL, 2015.

\bibitem{Woodroofe} R. Woodroofe, Vertex decomposable graphs and obstructions 
to shellability, Proc. Amer. Math. Soc. {\bf 137} (2009), no. 10, 3235--3246.

\bibitem{woodroofe-matchings} R. Woodroofe, Matchings, coverings, and
Castelnuovo-Mumford regularity, J. Commut. Algebra {\bf 6} (2014),
no. 2, 287--304. 

\end{thebibliography}

\end{document}